\documentclass[11pt,letterpaper]{amsart}
\usepackage[latin9]{inputenc}
\usepackage{amsmath}
\usepackage{amsfonts}
\usepackage{amssymb}
\usepackage{amsthm}
\usepackage{subfigure}
\usepackage[usenames,dvipsnames]{color}
\usepackage{pstricks}
\usepackage{graphicx}
\usepackage[all]{xy}
\newtheorem{theo}{Theorem}[section]
\newtheorem{prop}[theo]{Proposition}
\newtheorem{coro}[theo]{Corollary}
\newtheorem{lemm}[theo]{Lemma}

\theoremstyle{definition}
\newtheorem{def1}[theo]{Definition}
\theoremstyle{remark}
\newtheorem{rema}[theo]{Remark}

\newcommand{\nwc}{\newcommand}
\nwc{\eps}{\epsilon}
\nwc{\ep}{\epsilon}
\nwc{\vareps}{\varepsilon}
\nwc{\Oph}{\operatorname{Op}_\hbar}
\nwc{\la}{\langle}
\nwc{\ra}{\rangle}

\nwc{\mf}{\mathbf} 
\nwc{\blds}{\boldsymbol} 
\nwc{\ml}{\mathcal} 

\nwc{\defeq}{\stackrel{\rm{def}}{=}}

\nwc{\cE}{\ml{E}}
\nwc{\cN}{\ml{N}}
\nwc{\cO}{\ml{O}}
\nwc{\cP}{\ml{P}}
\nwc{\cU}{\ml{U}}
\nwc{\cV}{\ml{V}}
\nwc{\cW}{\ml{W}}
\nwc{\tU}{\widetilde{U}}
\nwc{\IN}{\mathbb{N}}
\nwc{\IR}{\mathbb{R}}
\nwc{\IZ}{\mathbb{Z}}
\nwc{\IC}{\mathbb{C}}
\nwc{\IT}{\mathbb{T}}
\nwc{\IS}{\mathbb{S}}
\nwc{\tP}{\widetilde{P}}
\nwc{\tPi}{\widetilde{\Pi}}
\nwc{\tV}{\widetilde{V}}
\nwc{\supp}{\operatorname{supp}}
\nwc{\rest}{\restriction}

\begin{document}

\title[Poincar\'e series and linking of Legendrian knots]{Poincar\'e series and linking of Legendrian knots}

\author[Nguyen Viet Dang]{Nguyen Viet Dang}

\address{Sorbonne Universite, IMJ-PRG, 75252 Paris Cedex 05, France}
\address{Institut Universitaire de France, Paris, France}

\email{dang@imj-prg.fr}

\author[Gabriel Rivi\`ere]{Gabriel Rivi\`ere}

\address{Laboratoire de math\'ematiques Jean Leray (UMR CNRS 6629), Universit\'e de Nantes, 2 rue de la Houssini\`ere, BP92208, 44322 Nantes C\'edex 3, France}

\address{Institut Universitaire de France, Paris, France}

\email{gabriel.riviere@univ-nantes.fr}

\begin{abstract} 
On a negatively curved surface,
we show that the Poincar\'e series counting geodesic arcs orthogonal to some pair of closed geodesic curves has a meromorphic continuation to the whole complex plane. When both curves are homologically trivial, we prove that the Poincar\'e series has an explicit rational value at $0$ interpreting it in terms of linking number of Legendrian knots. 
In particular, for any pair of points on the surface, the lengths of all geodesic arcs connecting the two points determine its genus, and, for any pair of homologically trivial closed geodesics, the lengths of all geodesic arcs orthogonal to both geodesics determine the linking number of the two geodesics. 
\end{abstract}

\maketitle


\section{Introduction}

Let $(X,g)$ be a smooth ($\mathcal{C}^{\infty}$), oriented, connected, closed Riemannian surface and which has \emph{negative curvature}. Given a nontrivial homotopy class $\mathbf{c}\in\pi_1(X)$, one can find a unique oriented geodesic $c$ (parametrized by arc length) in the conjugacy class of $\mathbf{c}$~\cite[\S 3.8]{Kl}. 
Similarly, any point $c\in X$ will be understood in the following as a closed geodesic representing the trivial homotopy class in $\pi_1(X)$. When the closed geodesic $c$ is embedded, we say that the geodesic is simple (including the case of a point).

A classical problem in Riemannian geometry consists in studying the lengths of the geodesic arcs joining two closed geodesics $c_1$ and $c_2$ which are primitive\footnote{It means that either the homotopy class $\mathbf{c}_i$ is trivial in $\pi_1(X)$, or the equation $\mathbf{c}^p=\mathbf{c}_i$ has no solution for every $p>1$. We will implicitely suppose this all along the article. The general case would follow from this case anyway.} in $X$. More precisely, for $T>0$, we denote by $\mathcal{N}_T(c_1,c_2)\in[0,+\infty]$ the number of geodesics $\gamma$ of length $0<\ell(\gamma)\leqslant T$ (parametrized by arc length) that join $c_1$ to $c_2$ and that are \emph{directly orthogonal} to $c_1$ and $c_2$. In that framework and when $c_1,c_2$ are points, Margulis proved, using purely dynamical methods~\cite{Mar69, Mar}, the existence of $A_{c_1,c_2}>0$ such that
\begin{equation}\label{e:margulis}
 \ml{N}_T(c_1,c_2)\sim A_{c_1,c_2}e^{Th_{\text{top}}},\quad \text{as}\quad T\rightarrow+\infty,
\end{equation}
where $h_{\text{top}}>0$ is the topological entropy of the geodesic flow. See also~\cite{Del42, Hub56, Hub} for earlier results of Delsarte and Huber in constant negative curvature using the spectral decomposition of the Laplace-Beltrami operator. This asymptotic formula was further generalized by Pollicott in the framework of Axiom A dynamical systems~\cite{Po}. Parkkonen and Paulin showed that~\eqref{e:margulis} remains true when $\mathbf{c}_1$ and $\mathbf{c}_2$ are any elements in $\pi_1(X)$ and when $X$ is not necessarily compact~\cite{ParPau17}. For smooth compact Riemannian manifolds without any assumption on their curvature and when $c_1$ and $c_2$ are points, Ma\~{n}\'e proved~\cite{Ma} that
$$\lim_{T\rightarrow+\infty}\frac{1}{T}\log\int_{X\times X}\ml{N}_T(c_1,c_2)d\text{vol}_g(c_1)d\text{vol}_g(c_2)=h_{\text{top}},$$
where $\text{vol}_g$ is the Riemannian volume induced by $g$ -- see also~\cite{Pa92, BuPa, PaPa, Pa00}. 

In this work, we shall focus on the case of negatively curved surfaces as in the works of Margulis. Recall that his strategy consisted in relating the study of these asymptotics with the mixing properties of the measure maximizing the Kolmogorov-Sinai entropy, the so-called Bowen-Margulis measure. We refer to the review of Parkkonen and Paulin for more details on these questions and for some recent developments~\cite{ParPau16}. In particular, estimates on the size of the remainder in~\eqref{e:margulis} can be derived~\cite[Th.~27]{ParPau17} from quantitative estimates on the rate of mixing of the Bowen-Margulis measure~\cite{Ra, Mo, Do, GLP, GoSto}. Such quantitative estimates can be obtained
for instance from the spectral analysis of transfer operators on appropriate Banach spaces of currents~\cite{Do, GLP, GoSto} which is also referred to as the study of Pollicott-Ruelle resonances. In the \emph{hyperbolic} case, estimates on the size of the remainder were previously obtained using the spectral decomposition of the Laplacian~\cite{Patt75, Gun80, Sel70, Goo}.

Instead of searching for improvements on the size of the remainder in~\eqref{e:margulis}, the aim of this work is to study more specifically the zeta renormalization of $\ml{N}_T(c_1,c_2)$:
$$\forall s\in\mathbb{C},\quad\ml{N}_T(c_1,c_2,s):=\sum_{\gamma\in\mathcal{P}_{c_1,c_2}:0<\ell(\gamma)\leqslant T}e^{-s\ell(\gamma)},$$
where $\mathcal{P}_{c_1,c_2}$ denotes the set of geodesic arcs $\gamma$ joining $c_1$ and $c_2$ and directly orthogonal\footnote{In other words, $\gamma^\prime(0)\perp T_{\gamma(0)}c_1$, $\gamma^\prime(\ell)\perp T_{\gamma(\ell)}c_2$ and $\gamma^\prime(0)\wedge  c_1'(\gamma(0))$, $\gamma^\prime(\ell)\wedge  c_2'(\gamma(\ell))$ have direct orientations.} to $c_1$ and $c_2$. Note that $\ml{N}_T(c_1,c_2,0)=\ml{N}_T(c_1,c_2)$. 

\subsection{Meromorphic continuation}

Thanks to~\eqref{e:margulis}, we can define, in the region $\text{Re}(s)>h_{\text{top}}$, the \emph{generalized} Poincar\'e series\footnote{This is just the Laplace transform of the measure $\sum_{\gamma\in\mathcal{P}_{c_1,c_2}:\ell(\gamma)>0}\delta_0(t-\ell(\gamma)).$}
\begin{equation}\label{e:poincare}
 \ml{N}_\infty(c_1,c_2,s):=\lim_{T\rightarrow+\infty}\ml{N}_T(c_1,c_2,s)=\sum_{\gamma\in\mathcal{P}_{c_1,c_2}:\ell(\gamma)>0}e^{-s\ell(\gamma)}.
\end{equation}
This defines a holomorphic function in the region $\text{Re}(s)>h_{\text{top}}$ and we will first prove the following result:
\begin{theo}\label{t:meromorphic} Let $(X,g)$ be a smooth ($\mathcal{C}^{\infty}$),
closed, oriented, connected, Riemannian surface which has negative curvature. Then, for every closed geodesics $(c_1,c_2)$ (including the case of points), the holomorphic function 
$$s\in\{w:\operatorname{Re}(w)>h_{\operatorname{top}}\}\mapsto \ml{N}_\infty(c_1,c_2,s)\in \mathbb{C}$$ 
has a meromorphic continuation to $\IC$.
\end{theo}

Our proof will use the spectral properties of transfer operators for uniformly hyperbolic flows developed by many authors over the last fifteen years~\cite{BL07, BL13, Ts10, FaSj, Ts12, GLP, DFG, DyZw13, FaTs, DyGu, GBWe17, Je19}. More precisely, we will interpret this Poincar\'e series in terms of a certain spectral resolvent applied to the conormal cycle of $c_1$ and $c_2$. Then, we will derive this theorem from the meromorphic continuation of this spectral resolvent. Our proof allows to encompass the case of much more general Anosov flows and Poincar\'e series (see Appendix~\ref{a:anosov}) but we limit ourselves to this simplified version for the introduction. As a byproduct of this argument, we will verify that the poles of this meromorphic continuation are included in the set of Pollicott-Ruelle resonances for currents of degree $1$. This spectral approach is in some sense close in spirit to what is done when proving the meromorphic continuation of dynamical zeta functions~\cite{Ba}. For instance, Giulietti-Liverani-Pollicott~\cite{GLP} and Dyatlov-Zworski~\cite{DyZw13} showed by spectral methods the meromorphic continuation of the Ruelle zeta function
\begin{equation}\label{e:ruelle}\zeta_{\text{Ruelle}}(s)=\prod_{\gamma\in\mathcal{P}}\left(1-e^{-s\ell(\gamma)}\right)^{-1},\end{equation}
where $\mathcal{P}$ denotes the set of primitive \emph{closed} geodesics. See also~\cite{Rue, Ru, Fr} for earlier results of Ruelle, Rugh and Fried in the analytic case, \cite{Ba} for a detailed account of Baladi in the case of Axiom A diffeomorphisms or~\cite{FaTs, Go} for the semiclassical zeta function of Faure--Tsujii. For all these other zeta functions, the meromorphic continuation was also established by spectral methods and their zeroes and poles were related to the Pollicott-Ruelle resonances on anisotropic spaces of currents as it is the case here.

While there are many results on Ruelle zeta functions, there are not so many works on the meromorphic continuation of Poincar\'e series. The only results we are aware of concern \emph{hyperbolic} manifolds where one can connect Poincar\'e series to a certain spectral resolvent of the Laplacian -- see for instance~\cite[Satz A]{Hub56},~\cite[Satz 2]{Hub} or~\cite[Lemme~3.1]{Gu}. Yet, such a correspondence is not available for general negatively curved manifolds and one has to work directly with the dynamical problem as we shall do here. The only dynamical proof of a meromorphic continuation of Poincar\'e series we are aware of is due to Paternain~\cite[p.~138]{Pa00} in the case where $(X,g)$ is hyperbolic. Under these assumptions, he proved by purely geometrical arguments that
\begin{equation}\label{e:paternain}\lim_{T\rightarrow+\infty}\int_{X\times X}\ml{N}_T(c_1,c_2,s)d\text{vol}_g(c_1)d\text{vol}_g(c_2)=\frac{4\pi^2\chi(X)}{1-s^2},\end{equation}
where $c_1,c_2$ are points and $\chi(X)$ is the Euler characteristic of $X$. He obtained this formula by interpreting this integrated Poincar\'e series via a convenient coarea formula -- see also~\cite{Ma, PaPa} for earlier related results. In some sense, our proof of the meromorphic continuation will use similar ideas but with the addition of the spectral analysis of Anosov flows to compensate the absence of simplifications due to constant curvature and to the integration over $X\times X$. Note that as a corollary of this result, of Theorem~\ref{t:meromorphic} and of Proposition~\ref{p:value-at-0} below, we recover the Euler characteristic of an \emph{hyperbolic surface} as a special value of Poincar\'e series:
\begin{equation}\label{e:paternain2}\boxed{\int_{X\times X}\ml{N}_{\infty}(c_1,c_2,0)d\text{vol}_g(c_1)d\text{vol}_g(c_2)=4\pi^2\chi(X).}\end{equation}
We will now show how to generalize this formula via our spectral approach.

\subsection{Value at $0$}

In a series of recent works, it was observed by Dyatlov-Zworski~\cite{DyZw} and the authors~\cite{DaRi16, DaRi17c, DaRi17d} that among Pollicott-Ruelle resonances, the one at $0$ plays a special role as its resonant states encode the De Rham cohomology of the manifold 
where the dynamics takes place. See also~\cite{Ha18, KuWe, CePa} for related results of Hadfield, K\"uster-Weich and Ceki\'{c}-Paternain. In the case of geodesic flows on Riemannian surfaces, the resonant states were in some sense computed explicitly by Dyatlov and Zworski~\cite[\S 3]{DyZw}. As a consequence, they proved
$$\boxed{s^{\chi(X)}\zeta_{\text{Ruelle}}(s)|_{s=0}\neq 0,}$$
and thus they generalized earlier results due to Fried in the case of constant curvature~\cite{Fr86}. 

The behaviour at $0$ of Poincar\'e series will be of slightly different nature as we will consider situations where there will be no pole or zero at $s=0$ even if there is a Pollicott-Ruelle resonance at $0$. Despite this and working out on the ideas introduced in~\cite{DaRi16, DaRi17c, DaRi17d, DyZw}, we will verify that the value at $0$ still has a topological meaning and that it is a rational number. To that aim, we 
set
$$\varepsilon(\mathbf{c})=1\quad\text{if $\mathbf{c}$ is trivial in $\pi_1(X)$,\quad and \quad } \varepsilon(\mathbf{c})=-1\quad\text{otherwise,}$$
and  the main result of this article reads:
\begin{theo}\label{t:zero} Let $(X,g)$ be a smooth ($\mathcal{C}^{\infty}$), closed, oriented, connected, Riemannian surface which has negative curvature. Given a pair $c_1$ and $c_2$ of two simple closed geodesics in $X$ which are trivial in homology and such that 
\begin{itemize}
 \item either $\mathbf{c}_1$ and $\mathbf{c}_2$ are distinct nontrivial homotopy classes,
 \item or at least one $c_i$ is a point and $c_1\cap c_2=\emptyset$.
\end{itemize}

Then one has
$$\boxed{\chi(X)\ml{N}_\infty(c_1,c_2,0)\in\IZ.}$$
Moreover, if $X(c_i)$ denotes the compact surface\footnote{When $c_i$ is a point, we take the convention that $X(c_i)=c_i$. When $c_i$ is not a point, $X\setminus c_i$ has two connected components (as $c_i$ is homologically trivial and embedded) and $X(c_i)$ is the closure of the component whose oriented boundary is $c_i$.} whose oriented boundary is given by $c_i$, then one has
\begin{equation}\label{e:main-linking}
\boxed{\ml{N}_\infty(c_1,c_2,0)=\varepsilon(\mathbf{c}_1)\left(\frac{\chi(X(c_1))\chi(X(c_2))}{\chi(X)}-\chi(X(c_1)\cap X(c_2))+\frac{1}{2}\chi(c_1\cap c_2)\right).}
\end{equation}
\end{theo}

Note that as $c_1$ and $c_2$ are homologically trivial, they intersect an even number of times and the contribution $\frac{1}{2}\chi(c_1\cap c_2)$ yields an integer\footnote{When $c_1$ corresponds to $c_2$ with its reverse orientation, one has $X(c_2)=\overline{X(c_1)^c}$ and $X(c_1)\cap X(c_2)=c_1\cap c_2=c_1$ has $0$ Euler characteristic.}.

  When $c_1$ and $c_2$ are \emph{distinct points}, 
we get
  $$\ml{N}_\infty(c_1,c_2,0)=\frac{1}{\chi(X)}.$$
    In fact, as we shall explain it in Remark~\ref{r:sign-anosov} and as it was pointed to us by one of the referee, this last formula remains true in the case of Anosov surfaces without focal points. If we only make the Anosov assumption, then the value at $0$ may differ from an integer depending on the choice of points. Thus, as a corollary of this result, the Euler characteristic of a negatively curved surface $(X,g)$ can be recovered by the set of lengths of the geodesic arcs joining two points of $X$. Still in the case of points, we also observe that we recover Paternain's formula~\eqref{e:paternain2} without integrating over $X$ and that it can be extended to variable curvature as follows:
  $$\int_{X\times X}\ml{N}_{\infty}(c_1,c_2,0)d\text{vol}_g(c_1)d\text{vol}_g(c_2)=\frac{\text{vol}_g(X)^2}{\chi(X)}.$$
 For the sake of illustration, we can also write down the formula when $c_1$ is not reduced to a point but $c_2$ (with $c_2\cap c_1=\emptyset$) is:
 $$\ml{N}_\infty(c_1,c_2,0)=-\frac{\chi(X(c_1))}{\chi(X)}+|X(c_1)\cap c_2|.$$
 
 More generally, this Theorem shows that the value at $0$ of Poincar\'e series is rational under homological assumptions on the homotopy classes we consider. As we shall see, the integer $\chi(X)\ml{N}_\infty(c_1,c_2,0)$ has a clear geometric interpretation if we lift the problem to the unit cotangent bundle. It will \emph{correspond to the linking of two Legendrian knots} given by the unit conormal bundles of the geodesic representatives $c_1$ and $c_2$. See \S\ref{s:morse} for details. In particular, the reason why we need $c_1,c_2$ to be homologically trivial is that their conormal bundles should also be homologically trivial so that they have a well--defined linking number in $S^*X$. The two Theorems above could be extended to Anosov surfaces and to more general simple closed curves on $X$ (not necessarily geodesics) but this requires to make a certain transversality assumption that will be described in~\eqref{eq:transversalitymovingframe}. Note however that, in this generalized case, the set $\{\gamma\in\mathcal{P}_{c_1,c_2}:0<\ell(\gamma)\leq T\}$ may not be finite and one needs to consider Poincar\'e series starting with geodesics arcs having a large enough length. In particular, the value at $0$ may differ by an element in $\IZ$ depending on this choice of minimal length. In the case of geodesics, this assumption is satisfied in the case of Anosov surfaces verifying in some sense a ``strong'' nonfocal point property.
 
 Let us also mention that the reduction to simple geodesics is only required for simplicity of exposition but the case of general closed geodesics can be handled similarly once we have defined properly what we mean for $X(c_i)$ in that case. See Section~\ref{s:intersection} for more details.
  

 The main reason for restricting ourselves to dimension $2$ is due to our spectral interpretation of Poincar\'e series. In particular, their value at $0$ is related to the properties of the eigenspace associated with the Pollicott-Ruelle resonance at $0$. As already alluded to, a rather precise description of that eigenspace was recently given by Dyatlov and Zworski in dimension $2$~\cite{DyZw} and we will crucially use this result (together with some ideas from~\cite{DaRi17c}) in order to identify the value at $0$.
 
\subsection{Perspectives and related results} 
 
  In higher dimensions, very few things are known on the eigenspace at $0$~\cite{KuWe, DGRS18, CDDP22} but any progress in that direction should in principle give some insights on the behaviour at $0$ of Poincar\'e series in higher dimensions following the lines of our proof. For instance, one could try to implement the recent results from~\cite{CDDP22} valid for nearly hyperbolic $3$-manifolds. Recall from~\cite[Prop.~3.2]{Pa00} that, for a hyperbolic manifold $(X,g)$ of dimension $n_0\geqslant 2$ and for trivial homotopy classes, Paternain's formula~\eqref{e:paternain} becomes
  \begin{multline*}\lim_{T\rightarrow+\infty}\int_{X\times X}\ml{N}_T(c_1,c_2,s)d\text{vol}_g(c_1)d\text{vol}_g(c_2)\\=\frac{4\pi^{\frac{n_0}{2}}\text{vol}_g(X)}{2^{n_0}\Gamma\left(\frac{n_0}{2}\right)}\sum_{k=0}^{n_0-1}\frac{(-1)^k\left(\begin{array}{c} n_0-1\\ k\end{array}\right)}{s+2k+1-n_0}.
  \end{multline*}
   In particular, there is a pole at $0$ in odd dimensions. Coming back to dimension $2$, the results of Hadfield~\cite{Ha18} for geodesic flows on surfaces with boundary should allow to find a formula similar to the one from Theorem~\ref{t:zero} in that case. Yet, this would be much beyond the scope of the present article. Thus, we shall not discuss this here and this was in fact recently achieved in~\cite{Ch21} using the methods of the present article combined with~\cite{DyGu, Ha18}.

Studying the meromorphic continuation of Poincar\'e series and their special value at $0$ is reminiscent from classical questions in number theory. The most famous example is given by the Riemann zeta function which equals $-\frac{1}{2}$ at $0$. More generally, for totally real fields, the Siegel-Klingen Theorem~\cite{Sie37, Klin} shows that the corresponding zeta function takes a rational value at $0$ (in fact at each nonpositive integer). Bergeron, Charollois, Garcia and Venkatesh show that this special value at $0$ can be interpreted as a linking number between periodic orbits of the suspension of \emph{hyperbolic} toral automorphisms of the $2$-torus~\cite{Ber18}. As the geodesic flow on negatively curved surfaces, these are examples of Anosov flows in dimension $3$. Coming back to geodesic flows on surfaces, we also mention the works of Ghys. He showed that, on the unit tangent bundle\footnote{It is homeomorphic to the complement of the trefoil knot in the $3$-sphere.} $\text {PSL}(2,\IR)/\text{PSL}(2,\mathbb{Z})$ of the modular surface, the linking number of a closed geodesic with the trefoil knot can be identified with the value of the Rademacher function on the given geodesic~\cite[\S~3.3]{Gh07}. The Rademacher function is an integer valued function defined on the set of closed orbits of the geodesic flow. More recently, Duke, Imamo\={g}lu and T\'oth showed how to express the linking number of two given (homologically trivial) geodesics of the modular surface as the special value of a certain Dirichlet series~\cite[Th.~4]{DIT}.

Hence, once reinterpreted in terms of linking between Legendrian knots (see~\S\ref{s:morse}), Theorem~\ref{t:zero} can be viewed as another occurrence of these interactions between knots and dynamics but in a context where no arithmetical tool is available. 
In the main part of our work, our knots will be Legendrian, thus ``orthogonal'' to the closed orbits of the geodesic flow. Yet, in Theorem~\ref{t:ghyslike}, we will verify that for two homologically trivial closed geodesics curves $c_1,c_2$ in $X$, the number $\mathcal{N}_\infty(c_1,c_2,0)$ actually computes the linking number of the two geodesics $\gamma_1,\gamma_2$ lifting $c_1,c_2$ in $S^*X$, yielding a direct connection with the linking numbers appearing in the above works.

\subsection*{Organization of the article} 

In Section~\ref{s:keylemma}, we prove a simple geometric lemma that is instrumental in our argument. In Section~\ref{a:geometry}, we review a few standard facts on Riemannian geometry and dynamical systems that are used throughout the article and we also apply the Lemma of Section~\ref{s:keylemma} in the case of geodesic flows. Section~\ref{s:zeta} is the main analytical part where we prove Theorem~\ref{t:meromorphic}. This proof relies on the microlocal methods introduced by Faure-Sj\"ostrand in~\cite{FaSj} and subsequently developed by Dyatlov-Zworski in~\cite{DyZw13} to study the meromorphic continuation of Ruelle zeta functions. The results in that section could be as well obtained via the geometric approach of Pollicott-Ruelle spectra previously developed by Liverani et al.~\cite{BL07, BL13, GLP}. Yet, the microlocal point of view and more specifically the notion of wavefront sets is quite convenient for the study of Poincar\'e series and more specifically of their value at $0$. After that, in Section~\ref{s:zero-contact} and in view of proving Theorem~\ref{t:zero}, we briefly review the recent results of Dyatlov and Zworski on Pollicott-Ruelle resonant states at $0$ for contact Anosov flows in dimension $3$. Then, we apply them to compute the residue of Poincar\'e series at $0$ and we show that this residue can be expressed in terms of representatives of the De Rham cohomology. In Section~\ref{s:morse}, we show that this residue is equal to $0$ for homologically trivial geodesic curves and we express the value at $0$ as the linking between two Legendrian knots. We conclude the proof of Theorem~\ref{t:zero} by appealing to classical results from differential topology such as the Poincar\'e-Hopf formula for manifolds with boundary derived by Morse~\cite{Mor29}. In Section~\ref{s:intersection}, we introduce the notion of constructible functions and show how it allows to extend Theorem~\ref{t:zero} for geodesics that are not necessarily simple. Finally, the article contains 4 appendices. In Appendix~\ref{a:WF}, we review some facts on wavefront sets of distributions that are used in this article. In Appendix~\ref{a:anosov}, we explain how to extend Theorem~\ref{t:meromorphic} to general Anosov flows while Appendix~\ref{ss:ghys} shows how to relate the linking number appearing in Section~\ref{s:morse} to the linking number of two geodesics in $S^*X$. The last appendix is devoted to the proof of a technical lemma that is used in Section~\ref{s:intersection}.

\subsection*{Conventions} All along the article, $(X,g)$ is a smooth, closed, oriented, connected, Riemannian \emph{surface}. Recall that, by closed, one means compact and boundaryless. We denote by $M:=S^*X$ the unit cotangent bundle which is naturally endowed with the Sasaki metric $g_S$~\cite{Pa99, Rug07}. Recall also that the Riemannian metric $g$ on $X$ induces natural isometries between $TX$ and $T^*X$ (endowed with dual metric $g^*$) and the one from $TX$ to $T^*X$ is denoted by $\flat$.

For $0\leqslant k\leqslant 3$, we denote by $\Omega^k(M)$ the space of smooth differential (complex-valued) forms of degree $k$ on $M$. Equivalently, it is the space of smooth complex-valued sections $s:M\rightarrow\Lambda^k(T^*M)$. The topological dual to $\Omega^{3-k}(M)$ (with the topology induced by $\mathcal{C}^{\infty}$-topology) is denoted by $\mathcal{D}^{\prime  k}(M)$ and is called the space of currents of degree $k$ -- see~\cite[Ch.~5]{Schwartz-66} for an introduction to the theory of currents. In particular, we shall denote by $[\Sigma]$ the current of integration over an oriented and embedded closed curve $\Sigma$ of $M$. In that case, $[\Sigma]$ is a current of degree $2$.

\subsection*{Acknowledgements} This work highly benefited from useful discussions and comments of many colleagues that we would like to warmly thank: N.~Bergeron, G.~Carron, B.~Chantraine, V.~Colin, M.~Golla, S.~Gou\"ezel, Y.~Guedes Bonthonneau, C.~Guillarmou, L.~Guillop\'e, J.~March\'e, G.~Paternain, F.~Paulin, S.~Tapie, N.~Vichery, T.~Weich and J.Y.~Welschinger. A special thank to our topologist colleagues for their precious help and for several useful discussions regarding the topological issues from Sections~\ref{s:morse} and~\ref{s:intersection}. Finally, we thank the two anonymous referees for their careful reading and for the helpful criticisms and suggestions that allowed to improve the exposition of the article.

Both authors are supported by the Insitut Universitaire de France and GR also acknowledges the support of the Agence Nationale de la Recherche through the PRC grant ADYCT (ANR-
20-CE40-0017).

\section{A fundamental geometric lemma}
\label{s:keylemma}

We begin with a simple geometric lemma that will be at the heart of our proof of the meromorphic continuation of Poincar\'e series. The reader more familiar with the case of the Ruelle zeta functions can view this result as an analogue in our set-up of the Guillemin (flat) trace formulas used to relate closed geodesics and the distributional kernel of the geodesic flow~\cite{GLP, DyZw13}.
\begin{lemm}\label{l:geomlemm}
Let $N_1,N_2$ be smooth, compact, embedded and oriented submanifolds (without boundary) of a manifold $N$ and let $Y$ be a nonsingular vector field which generates a flow $\varphi^t_Y$ and which is transverse to $N_2$, i.e.
$$\forall x\in N_2,\ Y(x)\notin T_xN_2.$$
Assume that
\begin{itemize}
 \item $\operatorname{dim}(N_1)+\operatorname{dim}(N_2)+1=\operatorname{dim}( N)$;
 \item $N_1\cap N_2=\emptyset$;
 \item for all $T\geqslant 0$ such that $N_1\cap \varphi^T_Y(N_2)=\emptyset$, the submanifolds $N_1$ and
$\{\varphi_Y^t(x): t\in [0,T], x\in N_2\}$ intersect transversally.
\end{itemize}

Let $\mu\in \mathcal{D}^\prime(\mathbb{R}_{>0})$ be defined as
\begin{equation}
\mu(t)=\sum_{0\leqslant \tau\leqslant T: N_1\cap\varphi_Y^{\tau}(N_2)\neq\emptyset}\left(\sum_{x\in N_1\cap\varphi_Y^{\tau}(N_2)}\epsilon_\tau(x)\right)\delta(t-\tau),
\end{equation}
where $\epsilon_\tau(x)$ is equal to $1$ if 
 
$$T_xN_1\oplus \IR Y(x)\oplus  d_{\varphi_Y^{-\tau}(x)}\varphi_Y^{\tau}T_{\varphi^{-\tau}_Y(x)}N_2$$ 
has the same orientation as $ T_{x}N,$ and to $-1$ otherwise.
Then, we have
\begin{equation}\label{e:keyformula}
\boxed{ \mu(t)=(-1)^{\dim(N_1)}\int_M  [N_1]\wedge\left(\iota_Y \varphi_Y^{-t*}[N_2]\right) ,}
\end{equation}
where both sides should be understood as \textbf{distributions} of $t$ and where $[N_1]$ and $[N_2]$ are the currents of integration over $N_1$ and $N_2$ respectively.
\end{lemm}
Note that our assumptions on $N_1$ and $N_2$ ensure that the intersection
of the two submanifolds $\{\varphi_Y^t(x): t\in [0,T], x\in N_2\}$ and $N_1$ consists 
of a finite number of points.
\begin{proof}
By a partition of unity argument, we only need to work locally near some point $x\in N_1$ such that $\varphi_Y^{-\tau}(x)\in N_2$ for some $\tau>0$. Up to replacing $N_2$ by $\varphi_Y^{\tau}(N_2)$, we may also assume that $\tau=0$ and that we work on a small interval of time centered around $0$ rather than in $\IR_{>0}$.
Using our transversality assumptions on $N_1$, $N_2$ and $Y$, we may assume without loss of generality that there are 
local coordinates $(x_1,\dots,x_k,x_{k+1},\dots,x_n)$ near $x$ such that
\begin{itemize}
 \item $N_1$ (resp. $N_2$) is given by the equations $x_{k+1}=\dots=x_n=0 $ (resp. $x_{1}=\ldots=x_{k+1}=0$);
 \item $x$ is given by $x_1=\dots=x_n=0$;
 \item the vector field $Y$ reads $\frac{\partial}{\partial x_{k+1}}$. 
\end{itemize}
Here, one has $\text{dim}(N_1)=k$ and $\text{dim}(N_2)=n-(k+1)$. In these local coordinates, the currents of integration on $N_1,N_2$ read:
$$[N_1]=\delta_0(x_{k+1},\dots,x_n)dx_{k+1}\wedge\dots\wedge dx_n,$$
and
$$\quad[N_2]=\delta_0(x_{1},\dots,x_{k+1})dx_1\wedge\ldots \wedge dx_{k+1}.$$
In this representation,
$N_1$ is oriented by $(-1)^{(n-k)k}dx_{1}\wedge\dots\wedge dx_k$ and $N_2$ by $dx_{k+2}\wedge\dots\wedge dx_n$, where we assume that $N$ is oriented\footnote{Recall that the integration current on any submanifold depends on some choice of orientation of the submanifold and the ambient manifold.} by $\text{Or}_N=dx_1\wedge \dots\wedge dx_n$. In particular, at $x=0$, the tangent space $T_xN_1\oplus \mathbb{R}Y(x)\oplus T_xN_2$ is oriented by the volume form $(-1)^{(n-k)k}\text{Or}_M$.
Let now $\chi_1$ be a smooth function compactly supported near $t=0$ and $x=0$. In order to conclude, we need to compute
$$\int_{\mathbb{R}\times M}\chi_1 [N_1]\wedge\iota_Y\varphi_Y^{-t*}([N_2])|dt|.$$
Using the above explicit formulas, this is in fact equal to
$$(-1)^{k(n-k+1)}\int_{\mathbb{R}\times M}\chi_1\delta_0(x_{k+1},\ldots x_n)\delta_0(x_1,\ldots,x_k,x_{k+1}-t) dx_1\wedge\ldots\wedge dx_n|dt|.$$
This can be rewritten as
\begin{multline*}
\int_{\mathbb{R}\times M}\chi_1 [N_1]\wedge\iota_Y\varphi_Y^{-t*}([N_2])|dt|
=(-1)^k(-1)^{(n-k)k}\chi_1(0,0)\\
=(-1)^{k+(n-k)k}\int_{M\times\mathbb{R}}\chi(t,x)\delta_0(x_1,\ldots,x_n,t)dx_1\wedge\ldots \wedge dx_n|dt|,
\end{multline*}
which implies the expected result (by partition of unity). Working in these local coordinates, one could also verify that $\int_{0}^{T}\iota_Y(\varphi_Y^{-t*}[N_2])|dt|$ is the current of integration on the submanifold $(\varphi_Y^{t}(N_2))_{0\leqslant t\leqslant T}$. In fact, with the above conventions for local coordinates, one has, locally near $x=0$ and for some small enough $t_0>0$, 
$$\int_{-t_0}^{t_0}\iota_Y(\varphi_Y^{-t*}[N_2])|dt|=(-1)^k\delta_0(x_{1},\ldots,x_k)dx_1\wedge\ldots\wedge dx_k.$$
In other words, $\int_{0}^{T}\iota_Y(\varphi_Y^{-t*}[N_2])|dt|$ is the current of integration on the submanifold $(\varphi_Y^{t}(N_2))_{0< t<T}$.

\end{proof}

\section{Background on Riemannian geometry}\label{a:geometry}

In this section, we collect some classical results on Riemannian and symplectic geometry -- see~\cite{Be78, Pa99, Rug07} for a more detailed account. Along the way, we recall classical notations that are used all along this article.

\subsection{Differential geometry of surfaces with negative Gauss curvature}\label{ss:preliminary} 

We recall some basic geometry of surfaces with special emphasis on the natural frame on the unit cotangent bundle following~\cite[7.2]{ST76} and \cite[section 3.5]{PaSaU}. Recall that we denoted by $M$ the unit cotangent bundle $S^*X=\{(q;p)\in T^*X; \vert p\vert_{g^*(q)}=1\}$ of our surface $X$, and that $\pi:M=S^*X\mapsto X$ is a $\mathbb{S}^1$--bundle. The geodesic vector field $V\in C^\infty(TM)$ on $M$ is the infinitesimal generator of the geodesic flow $\varphi^t:M\mapsto M$. The Liouville $1$--form $\alpha \in C^\infty(T^*M)$ is defined by 
$$\alpha(\eta)=\left\langle d_{(q,p)}\pi(\eta),p \right\rangle, \forall \eta\in T_{(q,p)}M. $$  
Then $\alpha$ is a contact form meaning that $\alpha\wedge d\alpha\in C^\infty\left(\Lambda^3T^*M \right)$ is a volume form on $M$ and the relation between $V$ and $\alpha$ reads
\begin{eqnarray}
\iota_Vd\alpha=0, \iota_V\alpha=1
\end{eqnarray} 
where $\iota$ denotes the interior product. We shall denote by $R\in C^\infty(TM)$ the infinitesimal generator of the $\mathbb{S}^1$ action on the fibers of $M$ which is also the vertical vector field. 
Let now $H$ be the vector field obtained from $V$ by a rotation of angle $\frac{\pi}{2}$ in the direct sense using the flow generated by $R$. The triple $\left(V,R,H \right)$ is a direct orthonormal frame for the orientation induced by $\alpha\wedge d\alpha$. It verifies the structure equations~\cite[Lemma 3.5.5 p.~79]{PaSaU}:
\begin{eqnarray}\label{e:commutation}
[V,R]=-H,\quad [H,R]=V,\quad [V,H]=\left(\pi^*K\right) R
\end{eqnarray}
where $\pi^*K$ is the pull--back of the Gaussian curvature $K$ on $X$.
In this picture, the tangent bundle $TM$ splits as the direct sum $TM=\mathcal{H}\oplus\text{ker} d\pi$ where $\mathcal{H}=\mathbb{R} V\oplus \mathbb{R} H $ and $\text{ker} d\pi=\mathbb{R}R$ is the line bundle spanned by $R$. Moreover, one has $\text{ker}\alpha=\IR H\oplus\IR R$. We also note that $d_{(q,p)}\pi(V)=p$ and that  $d_{(q,p)}\pi( H)=p^\perp$, where $p^\perp$ is the covector that is directly orthogonal to $p$ for the orientation that we have fixed on $X$.

From the dynamical point of view, recall that since $(X,g)$ has negative curvature, the flow $\varphi^t$ is Anosov~\cite[\S 1]{Mar}. Namely, there exists a continuous splitting:
\begin{equation}\label{e:Anosov}\forall x\in M,\quad T_x M=\IR V(x)\oplus E_u(x)\oplus E_s(x), \,\ \dim(E_u)=\dim(E_s)=1
\end{equation}
where $E_u(x)\neq\{0\}$ (resp. $E_s(x)\neq\{0\}$) is the unstable (resp. stable) direction. 
Moreover the unstable and stable directions are preserved by the tangent map $d_x\varphi^t$ and there exist some constants $C>0$ and $\lambda_0>0$ such that, for every $t\geqslant 0$,
$$\forall v\in E_u(x),\quad \|d_x\varphi^{-t}v\|_{g_S(\varphi^{-t}(x))}\leqslant Ce^{-\lambda_0 t}\|v\|_{g_S(x)},$$
and
$$\forall v\in E_s(x),\quad \|d_x\varphi^{t}v\|_{g_S(\varphi^{t}(x))}\leqslant Ce^{-\lambda_0 t}\|v\|_{g_S(x)}.$$
For every $x\in M$, one can define the weakly unstable (resp. stable) manifold $W^{u0}(x)$ (resp. $W^{s0}(x)$) which are smooth immersed submanifolds inside $M$ such that, for every $x\in M$,
\begin{equation}\label{e:un-stable-manifold}T_x W^{u0}(x)=E_u(x)\oplus\IR V(x),\quad\text{and}\quad T_x W^{s0}(x)=E_s(x)\oplus\IR V(x).
\end{equation}
For later purpose, we also define the dual spaces $E_u^*(x)$, $E_s^*(x)$ and $E_0^*(x)$ as the annihilators of $E_u(x)\oplus \IR V(x)$, $E_s(x)\oplus\IR V(x)$ and $E_u(x)\oplus E_s(x)$. We have in some sense an explicit description of the stable and unstable bundles in the vertical/horizontal decomposition of $TM$ via the stable and unstable Ricatti solutions~\cite[\S 3.1.2]{Rug07}. More precisely, there exist two continuous functions $r_\pm:M\rightarrow\mathbb{R}_\pm^* $ such that
\begin{eqnarray}\label{eq:EsEuversusframe}
E_s(x)= \mathbb{R}\left( H(x)+r_-R(x) \right), E_u(x)=\mathbb{R}\left( H(x)+r_+R (x)\right).
\end{eqnarray}
The above equations immediately imply that both $E_{s/u}$ are transverse to the vertical bundle $\mathcal{V}$ and to $\IR H$ inside $\text{ker}\ \alpha$. Recall that the functions $r_\pm$ are $C^1$ in the direction of the flow and that they are solutions to the Ricatti equations
\begin{eqnarray*}
\pm Vr_\pm+r_\pm^2+K\circ \pi=0.
\end{eqnarray*}

\begin{rema}\label{r:anosov-ricatti} More generally, unstable and stable Ricatti solutions are well defined for more general Anosov surfaces (in fact for surfaces without conjugate points). Our assumption on the curvature ensures that $r_\pm$ is nowhere vanishing. 
\end{rema}

Ricatti solutions are naturally related to perpendicular Jacobi fields $J$ along geodesics, through the relation $r=J'(t)J(t)^{-1}$ when it makes sense. Recall that, along a geodesic $t\in\mathbb{R}\mapsto\gamma(t)\in X$ verifying $(\gamma(0),\gamma'(0))=x$, a Jacobi field $J(t)$ is a solution of the equation
$$J''(t)+K(\gamma(t))J(t)=0.$$
Moreover, one has the natural identification
$$\forall t\in\IR,\ d_x\varphi^t(J(0)H(x)+J'(0)R(x))=J(t)H\circ\varphi^t(x)+J'(t)R\circ\varphi^t(x).$$
In the case of $r_\pm$, one has by construction that $J_\pm$ are nowhere vanishing~\cite[Ch.3]{Rug07}. 

Finally, a surface $(X,g)$ is said to have \emph{no focal point} if, for every Jacobi field, $J(0)=0$ implies that $J'(t)>0$ for every $t>0$. In particular, if we denote by $J_R(t)$ the Jacobi field with initial condition $(J(0),J'(0))=(0,1)$, then one has $J_R'(t)\neq 0$ for every $t>0$. We refer to~\cite[\S3]{Gul} for the construction of Anosov surfaces without focal points and with curvatures of both signs.

\subsection{Transversal submanifolds}
\label{ss:transversesubm}

We now fix two smooth embedded curves $\Sigma_1$ and $\Sigma_2$ in $M=S^*X$ that we suppose to be oriented  and boundaryless.
We make the following transversality assumptions which already appeared in the seminal work of Margulis~\cite[p.~49]{Mar}:
\begin{equation}\label{e:transversality-unstable}\forall x\in\Sigma_1,\quad T_x M=T_x\Sigma_1\oplus T_{x} W^{u0}(x),
\end{equation}
and
\begin{equation}\label{e:transversality-stable}\forall x\in\Sigma_2,\quad T_x M=T_x\Sigma_2\oplus T_{x} W^{s0}(x).
\end{equation}

The simplest example of curve $\Sigma$ verifying either~\eqref{e:transversality-unstable} or~\eqref{e:transversality-stable} is
just the fiber $\Sigma(q):=S_q^*X\subset M$. 
This follows from the fact that $T_x(S_q^*X)$ is the vertical space at $x$ which is transversal to the weakly unstable/stable manifold at $x$ thanks to~\eqref{eq:EsEuversusframe}.

\begin{rema} \label{r:fiber}
We note that each submanifold $S_q^*X$ is oriented
by the vertical vector field $R$.  
It means that we orient it in the trigonometric sense 
relative to the orientation on $X$.
\end{rema}

\subsubsection{Examples}\label{sss:tangent-space}

Besides the fiber $S_q^*X$ (and small perturbations of it), we can consider $c:t\in\IR/\ell
\IZ\mapsto q(t)\in X$, $\ell>0$ to be a \emph{smooth} curve such that $q'(t)\neq 0$ for every $t\in\IR/\ell
\IZ$. Up to reparametrization, we can choose $q'(t)$ to be of norm $1$ for every $t\in\IR/\ell
\IZ$. 
Then, we define the (unit) conormal bundle to $c$:
$$N_1^*(c):=\{(q(t),p)\in S^*X:\ t\in\IR/\ell
\IZ,\ p(q^\prime(t))=0\}.$$
This defines a smooth curve inside $S^*X$ and we can verify that, for every $x\in\Sigma$, $T_x\Sigma$ is contained inside the kernel of the canonical Liouville one form $\alpha$. Such a submanifold $N_1^*(c)$ is called \emph{Legendrian}.

Since $\text{dim}\ X=2$, $N_1^*(c)$ consists of two connected components. Given $t\in\IR/\ell
\IZ$, we denote these two covectors by $p^\perp(t)$ and $-p^{\perp}(t)$, with $p^\perp$ being the covector directly orthogonal to $q'(t)^\flat$, i.e. $p^\perp(t):=\star_gq'(t)^\flat$, where $\flat$ means that we take the covector associated with $q'(t)$ via the natural isometry induced by $g$ and where $\star_g$ is the Hodge star map with the convention that $g^*(p,\star_g p)>0$. This defines a natural parametrization of each component of $N_1^*(c)$ as:
$$\gamma_\pm(t)=(q(t), \pm (p^\perp)(t))\in M.$$

 There is a way to check the transversality assumption~\eqref{e:transversality-unstable} in the natural moving frame $(V,H,R)$ of $TM$. The velocity $\gamma_\pm^\prime(t)\in T_{\gamma(t)}M$ can be decomposed in the frame $(V,H,R)$ as 
\begin{eqnarray}
\gamma_\pm^\prime(t)=a_1(t)H(\gamma_\pm (t))+a_{ 2}(t)R(\gamma_\pm(t))
\end{eqnarray} 
where $a_i(t)_{i=1,2}$ are smooth in the variable $t$.
Then the transversality condition means that $(\gamma_\pm^\prime(t),V(\gamma_\pm(t)),U(\gamma(t)) )$ forms a basis of $T_{\gamma(t)}M$ where $U(x)=H(x)+r_+R(x)$ is the unstable vector field. This condition also reads
\begin{eqnarray}\label{eq:transversalitymovingframe}
\boxed{\forall t\in\IR/\ell
\IZ,\quad a_{1}(t)r_+(\gamma_\pm(t))-a_{2}(t)\neq 0.}
\end{eqnarray}
\begin{rema} One can verify that this can be expressed more concretely as follows
$$\forall t\in\IR/\ell
\IZ,\quad g^*_{q(t)}\left(\nabla_{q'(t)}p^\perp(t),p(t)\right)\neq r_+(q(t),p^\perp(t)),$$
where $\nabla$ is the covariant derivative induced by $g$. 
\end{rema}

\begin{rema}\label{r:geodesic} In Theorem~\ref{t:zero}, besides the case of points, we will consider the case where $c(t)$ is a geodesic representative of a nontrivial homotopy class $\mathbf{c}\in\pi_1(X)$. In that case, we will take $\Sigma=\Sigma(c)$ to be the connected component of $N_1^*(c)$ consisting of the covector directly orthogonal to $p(t)$. This satisfies the expected transversality assumption as $a_{ 2}=0$ in that case and as $r_+\neq 0$ thanks to negative curvature. For more general Anosov surfaces, this transversality assumption may not be satisfied for certain closed geodesics as one may have $r_\pm =0$ at some points.

We choose to orient $\Sigma(c)$ via the orientation of the geodesic, i.e. with the vector $-H$. 
\end{rema}

\subsubsection{Orientations in a toy model}\label{r:ex-orientation}

Let us illustrate our choices of orientation in a toy model on $\mathbb{R}^2\times \mathbb{S}^1$, oriented by $dq_1\wedge dq_2\wedge d\phi$, that will be useful for our computations in Section~\ref{s:morse}.
  \begin{itemize}
   \item  \textbf{Example 1}. We consider the horizontal line $c_1=\{(q_1,0), q_1\in \mathbb{R}\}$ oriented by $dq_1$ in $\mathbb{R}^2$. Recalling now that $\mathbf{1}_{[0,+\infty)}'(q)=\delta_0(q)$, therefore $\partial[\IR\times\IR_+]=\partial\mathbf{1}_{\IR_+}(q_2)= -d\mathbf{1}_{\IR_+}(q_2)=-\delta_0(q_2)dq_2=[c_1]$. In other words, $c_1$ is the oriented boundary of $\IR\times\IR_+$. Then, $\Sigma(c_1):=\{(q_1,0,\pi/2):q_1\in\IR\}$ and we oriented it using $dq_1$. This yields the following representation of its current of integration
$$[\Sigma(c_1)]=\delta_0(q_2)\delta_0\left(\phi-\frac{\pi}{2}\right)dq_2\wedge d\phi.$$
Introduce now the surface $S:=\{(q_1,q_2,\pi/2):q_2\geqslant 0\}$ whose (topological) boundary is $\Sigma(c_1)$. This surface is naturally oriented by $dq_1\wedge dq_2$ and thus it can be represented as
$$[S]=\mathbf{1}_{q_2\geqslant 0}(q_1,q_2)\delta_0\left(\phi-\frac{\pi}{2}\right)d\phi.$$
One finds that $[\Sigma(c_1)]$ is a coboundary:
$$d[S]=\delta_0(q_2)\delta_0\left(\phi-\frac{\pi}{2}\right)dq_2\wedge d\phi=[\Sigma(c_1)].$$
\item \textbf{Example 2}. Consider now the point $c_0=(0,0)$. One has
$$S_{c_0}^*\mathbb{R}^2=\{c_0\}\times\mathbb{S}^1:=\{(0,0,\phi):0\leqslant\phi\leqslant 2\pi\}.$$
Our choice of orientation on this curve is to take $d\phi$. Hence, the current of integration on the
fiber $S_{c_0}^*\mathbb{R}^2$ reads $[S_{c_0}^*\mathbb{R}^2]=\delta_0(q_1,q_2)dq_1\wedge dq_2$. As in the above example, introduce the following submanifold $S$ in $\mathbb{R}^2\times\mathbb{S}^1$:
$$S:=\left\{(q_1,q_2,\phi):\ (q_1,q_2)\in\IR^2\setminus\{c_0\},\ \cos\phi=\frac{q_1}{\sqrt{q_1^2+q_2^2}},\ \sin\phi=\frac{q_2}{\sqrt{q_1^2+q_2^2}}\right\},$$
whose topological boundary is $S_{c_0}^*\mathbb{R}^2$. 
Endowing $S$ with the orientation $dq_1\wedge dq_2$ yields the following representation of $[S]$ in $\mathbb{R}^2\setminus\{0\}\times(-\pi/2,\pi/2)$:
\begin{eqnarray*}[S]&=&\mathbf{1}_{\IR_+}(q_1)\delta_0\left(q_2-q_1\tan(\phi)\right)d(q_1\tan\phi-q_2)\\
&=&\mathbf{1}_{\IR_+}(q_1)\delta_0\left(q_2-q_1\tan(\phi)\right)\left(\frac{q_1d\phi}{1+\phi^2}+\tan(\phi) dq_1-dq_2\right).
\end{eqnarray*}
To see that this is a well defined current, one can take a smooth test form $\psi(q,\phi,dq,d\phi)\in \Omega^2(\IR^2\times(-\pi/2,\pi/2))$ and observe that $\int_{\IR^2\times(-\pi/2,\pi/2)} [S]\wedge \psi$ is well defined. 
This current can be extended into a well-defined current on $\IR^2\times(-\pi/2,\pi/2)$. Hence, by a partition of unity in the $\phi$ variable, $[S]$ defines a current on $\IR^2\times\IS^1$. Finally, in $\mathbb{R}^2\times(-\pi/2,\pi/2)$, one has
$$d[S]=-\delta_0(q_1)\delta_0\left(q_2-q_1\tan\phi\right)dq_1\wedge dq_2=-[S_{c_0}^*\mathbb{R}^2].$$
Performing the same argument in every half plane, one finds that $[S_{c_0}^*\mathbb{R}^2]$ is a coboundary: $d[S]=-[S_{c_0}^*\mathbb{R}^2]$. Using polar coordinates $(r,\theta,\phi)$, this manifold can also be viewed as the boundary of the manifold $\mathbb{R}_{>0}\times\mathbb{T}^2$
oriented by $rdr\wedge d\theta\wedge d\phi$. In these coordinates, the current of integration on the fiber $S_{c_0}^*\mathbb{R}^2$, oriented by $d\phi$, reads
$$[S_{c_0}^*\mathbb{R}^2]= \delta_0(r)\delta_0(\phi-\theta)dr\wedge (d\theta-d\phi).$$
If we now form the surface 
$$S:=\left\{\left(r,\theta,\theta\right):r> 0,\ 0\leqslant\theta\leqslant 2\pi\right\},$$
endowed with the orientation $rdr\wedge d\theta$, then
$$[S]=\mathbf{1}_{\mathbb{R}_{>0}}(r)\delta_0(\phi-\theta)(d\phi-d\theta).$$
In particular,
$$d[S]=\delta_0(r)\delta_0(\phi-\theta)dr\wedge (d\phi-d\theta)=-[S_{c_0}^*\mathbb{R}^2].$$
  \end{itemize}

\subsection{First properties}\label{ss:first-prop}

Using the Anosov property and our transversality assumptions, one can show the following:
\begin{lemm}[Transversality for long enough times]\label{l:transversality} 
Suppose that~\eqref{e:transversality-unstable} and~\eqref{e:transversality-stable} hold.
Then, there exists some $T_0 >0$ such that, for every $t\geqslant T_0$, $\varphi^{-t}(\Sigma_1)$ and $\Sigma_2$ intersect transversally with respect to the flow in the sense that, for every $x\in \varphi^{-t}(\Sigma_1)\cap\Sigma_2$,
$$T_x M=T_x \varphi^{-t}(\Sigma_1)\oplus T_x \Sigma_2\oplus \IR V(x).$$
In particular, for every fixed $t\geqslant T_0$, the set of points lying in $\varphi^{-t}(\Sigma_1)\cap\Sigma_2$ is finite. Moreover, the times $t\geqslant T_0$ for which this intersection is not empty are discrete and do not accumulate. 
\end{lemm}
In particular, our preliminary geometric lemma~\ref{l:geomlemm} applies with $N_1=\varphi^{-t}(\Sigma_1)$ and $N_2=\Sigma_2$ for every $t\geq T_0$.
\begin{proof}
We fix $v(x,t)$ generating $T_{\varphi^{t}(x)}\Sigma_1$. Thanks to~\eqref{e:transversality-unstable}, one has then
$$T_{\varphi^t(x)}M=\mathbb{R}(v(x,t))\oplus E_u(\varphi^ t(x))\oplus \IR V(\varphi^t(x)).$$
Using the Anosov property and the fact that $\Sigma_1$ is a closed embedded curve satisfying~\eqref{e:transversality-unstable}, one can find some $T_0$ such that, for $t\geqslant T_0$ and for every $x\in\varphi^{-t}(\Sigma_1)$, $d_{\varphi^{t}(x)}\varphi^{-t} (v(x,t))$ lies in a fixed (small) conical neighborhood of $E_s(x)$. Moreover, one has
$$T_{x}M=d_{\varphi^{t}(x)}\varphi^{-t} \mathbb{R}(v(x,t))\oplus E_u(x)\oplus \IR V(x).
$$

Hence, we have shown that the family $(V(x),(d_{\varphi^{t}(x)}\varphi^{-t} (v(x,t))))$ generates a vector space of dimension $2$ lying in a conical neighborhood of $E_s(x)\oplus \IR V(x)$ that does not intersect $T_x\Sigma_2$ thanks to~\eqref{e:transversality-stable}. This concludes the first part of the lemma 

Now, we fix $t_1<t_2$ both greater than $T_0$ and we consider the two submanifolds $\Sigma_2$ and $(\varphi^{-t}(\Sigma_1))_{t_1<t<t_2}$ of $M$. They intersect transversally thanks to the first part of the lemma. Hence, as they are respectively of dimension $1$ and $2$, one finds that the intersection between these two submanifolds consists in a finite number of points which concludes the proof. 
\end{proof}

Recall that, in Theorems~\ref{t:meromorphic} and~\ref{t:zero}, we are primarly interested with the case where $\Sigma(c)$ is the direct conormal to an oriented closed geodesic (including the case where $c$ is reduced to a point). See remarks~\ref{r:fiber} and~\ref{r:geodesic} for definitions of $\Sigma(c)$. In that case, we can be slightly more precise about transversality:

\begin{lemm}[Immediate transversality]\label{l:transversality-geodesic} Suppose that $M=S^*X$ where $(X,g)$ is a Riemannian surface and $V$ is the geodesic vector field. Let $c_1$ and $c_2$ be two closed geodesics (including the case of points). Then,
the transversality conditions~(\ref{e:transversality-stable}), (\ref{e:transversality-unstable}) and the conclusion of Lemma~\ref{l:transversality} hold for $\Sigma(c_1)$ and $\Sigma(c_2)$ with \begin{itemize}                                                                                                                                                                                                                                                    \item either $T_0=0$ if $\Sigma(c_1)\cap\Sigma(c_2)=\emptyset$,
\item or any $T_0>0$ otherwise.                                                                                                                                                                                                                                                                                                                                                                                         \end{itemize}

\end{lemm}

The point of taking $T_0=0$ here is that the set of geodesic arcs $\mathcal{P}_{c_1,c_2}$ directly orthogonal to $c_1$ and $c_2$ verifies $\sharp\{\gamma\in\mathcal{P}_{c_1,c_2}:0<\ell(\gamma)\leq T\}<\infty$ for every $T> 0$. For more general curves $c_1$ and $c_2$ that only satisfy~\eqref{eq:transversalitymovingframe}, this may not be the case. Our main results would remain valid but we would start our Poincar\'e series from a certain $T_0$ given by Lemma~\ref{l:transversality}. Moreover, the value at $0$ (when it makes sense) computed in Theorem~\ref{t:zero} will depend on this choice of $T_0$ up to an integer term.

\begin{proof}
We already said that, in the case of geodesic curves, the transversality conditions~\eqref{e:transversality-unstable} and \eqref{e:transversality-stable} hold true for both Legendrians $\Sigma(c_1)$ and $\Sigma(c_2)$. 
Hence it remains to explain why we can choose $T_0=0$ when $\Sigma(c_1)\cap\Sigma(c_2)=\emptyset$ or any $T_0>0$ otherwise.

To that aim, we need to compute the differential of the geodesic flow acting on a vector field $W$ which is either $H$ (when $c_1$ is not a point) or $R$ (when $c_1$ is a point) and to show that, for every $t>0$, $(\tilde{W},d\varphi^{-t}(W))$ forms a moving frame of the plane bundle $\ker(\alpha)\subset TM$ with $\tilde{W}=R$ (resp. $H$) if $c_2$ is (resp. not) a point. This calculation can be easily handled using the formalism of Jacobi fields~\cite[p.18]{Rug07} that was briefly recalled at the end of \S\ref{ss:preliminary}. We start with the case where $W=H$. One has
$$d_x\varphi^{-t}(H)=J_H(-t) H(\varphi^{-t}(x))+J_H'(-t)R(\varphi^{-t}(x)),$$
where $J_H(t)$ solves the Jacobi equation $J''(t)+K\circ\pi(\varphi^t(x))J(t)=0$ with initial conditions $J(0)=1$ and $J'(0)=0$. In particular, as $K<0$ everywhere on $X$, one finds that $J_H'(-t)<0$ and $J_H(-t)\geq 1$ for every $t>0$. Hence, $d_x\varphi^{-t}(H)$ has nontrivial components along $H$ and $R$. It implies that $(H,d\varphi^{-t}H)$ and $(H,d\varphi^{-t}R)$ are basis of $\text{ker}(\alpha)$ for every $t>0$ as expected. 

In the case where $W=R$, one has
$$d_x\varphi^{-t}(R)=J_R(-t) H(\varphi^{-t}(x))+J_R'(-t)R(\varphi^{-t}(x)),$$
where $J_R(t)$ solves the Jacobi equation $J''(t)+K\circ\pi(\varphi^t(x))J(t)=0$ with initial conditions $J(0)=0$ and $J'(0)=1$. We can then conclude similarly. In fact, one can also pick $T_0=0$ when $\textbf{c}_1$ is trivial and $\textbf{c}_2$ is not (even if $c_1$ lies on $c_2$).
\end{proof}

\begin{rema}\label{r:sign-anosov}For surfaces with an Anosov geodesic flow (but not necessarily negatively curved), there is no conjugate points so that $J_R(t)\neq 0$ as soon as $t\neq 0$. If we make the extra assumption that $(X,g)$ has no focal points (see~\S\ref{ss:preliminary}), then one has $J_R'(-t)\neq 0$ for every $t>0$ so that the conclusion of the lemma remains true in the case where $c_1$ and $c_2$ are points. In particular, in the case of two points, our main Theorem remains valid for Anosov surfaces without focal points.
\end{rema}

\subsection{Applying Lemma~\ref{l:geomlemm}}
We are now in position to apply our preliminary lemma to the geodesic vector field on $M$. More precisely, we introduce a weighted version of Poincar\'e series in such a way that the sum runs over a finite number of elements:

\begin{prop}\label{p:WF-finite-time-geodesic} Let $c_1,c_2$ be two closed geodesics (including the case of points) and denote by $\Sigma_1=\Sigma(c_1)$ and $\Sigma_2=\Sigma(c_2)$ their corresponding Legendrian lifts.  
 
 Then, for every $T_0\geqslant0$ satisfying
 $$\varphi^{-T_0}(\Sigma_1)\cap \varphi^{T_0}(\Sigma_2)=\emptyset,$$
 one can find some $t_0>0$ such that, for every $\chi\in\mathcal{C}^{\infty}_c((2T_0-t_0,+\infty))$, and for every $\tau\geqslant 0$:
$$I_\tau(\chi):=-\int_{M}\varphi^{T_0*}([\Sigma_1])(x,dx)\wedge \left(\int_{\IR}\chi(t-2T_0)\varphi^{-(t+T_0+\tau)*}\iota_V[\Sigma_2](x,dx)  \vert dt\vert \right)$$
is well defined and it is equal to
$$\varepsilon(\mathbf{c}_2)\sum_{t\geqslant 2T_0-t_0: \varphi^{-T_0}(\Sigma_1)\cap\varphi^{t+\tau+T_0}(\Sigma_2)\neq\emptyset}\left(\sum_{x\in  \varphi^{-T_0}(\Sigma_1)\cap\varphi^{t+\tau+T_0}(\Sigma_2)}\chi( t -2T_0)\right),$$
where
$$\varepsilon(\mathbf{c}_2):= 1,\ \text{if $\mathbf{c}_2$ is trivial},$$
and
$$\varepsilon(\mathbf{c}_2):= -1,\ \text{otherwise}.$$
\end{prop}

\begin{proof}[Proof of Proposition \ref{p:WF-finite-time-geodesic}] Thanks to Lemma~\ref{l:transversality-geodesic}, we are in position to apply the geometric lemma~\ref{l:geomlemm} with $N_1=\Sigma_1$, $N_2=\Sigma_2$, $N=M$ and $Y=V$. The conclusion of this Lemma exactly tells us the expected equality up to the computation of the orientation parameter $\epsilon_\tau(x)$. To compute this parameter, we proceed as in the proof of Lemma~\ref{l:transversality-geodesic} when we computed the image of $H$ and $R$ under the differential of the geodesic flow. More precisely, the orientation on $$T_x\Sigma(c_1)\oplus\IR V(x)\oplus d_{\varphi^{-t}(x)}\varphi^{t}\left(T_{\varphi^{-t}(x)}\Sigma(c_2)\right)$$
 induced by the orientation on each subspace is given by
 $$\tilde{W}(x)\wedge V(x)\wedge\left(d_{\varphi^{-t}(x)}\varphi^{t}(W(\varphi^{-t}(x)))\right),$$ 
 where $W=R$ (resp. $-H$) if $c_2$ is (resp. not) reduced to a point and the same for $\tilde{W}$. This is proportional to the polyvector $V\wedge H\wedge R$ with a coefficient in front that can be expressed in terms of $J_{R}(t)$, $J_R'(t)$, $J_H(t)$ and $J_H'(t)$ with the conventions of the proof of Lemma~\ref{l:transversality-geodesic}. The same sign discussion as the one discussed there allows to conclude on the value of $\epsilon_\tau(x)$.

\end{proof}

\subsection{A priori bounds on the growth of intersection points}

We would now like to replace the smooth compactly supported function $\chi$ in Proposition~\ref{p:WF-finite-time-geodesic} by $e^{-st}\mathbf{1}_{\IR_+}(t)$. This is where we will extensively use the theory of Pollicott-Ruelle resonnances in the next paragraph. Before doing that, we need some a priori estimate on the growth of points in $\mathcal{P}_{c_1,c_2}$.

According to Lemma~\ref{l:transversality-geodesic}, one knows that 
$$\tilde{\mathcal{P}}_{c_1,c_2}:=\left\{t>0: \varphi^{-t}(\Sigma(c_1))\cap\Sigma(c_2)\neq\emptyset\right\}$$
defines a discrete subset of $\IR_+^*$ with no accumulation points. Moreover, for every $t>0$, we can define
\begin{equation}\label{e:multiplicity}m_{c_1,c_2}(t):=\left|\left\{x\in \varphi^{-t}(\Sigma(c_1))\cap\Sigma(c_2)\right\}\right|<+\infty,\end{equation}
which is thus equal to $0$ outside a discrete subset of $(0,+\infty)$. We begin with the following a priori upper bound on these quantities:
\begin{lemm}\label{l:exp-growth} Let $c_1$ and $c_2$ be two closed geodesics. Then, for every $h>h_{\operatorname{top}}$, one can find some constant $C_h>0$ such that, for every $T>0$,
$$\sum_{t\in[T,T+1)}m_{c_1,c_2}(t)\leqslant C_h e^{hT}.$$
\end{lemm}

The proof of this Lemma could be extracted from Margulis' arguments in~\cite[\S 7]{Mar}. Yet, for the sake of completeness, we give a short proof of it.

\begin{proof} We fix some $h>h_{\text{top}}$. For every $x\in M$ and for every $\epsilon, T>0$, we define the Bowen ball centered at $x$:
$$B(x,\epsilon,T):=\{y\in M:\forall 0\leqslant t\leqslant T,\ d_{g_S}(\varphi^t(x),\varphi^t(y))<\epsilon\},$$
where $d_{g_S}$ is the distance induced by the Riemannian metric. From the definition of the topological entropy, one can find some $\epsilon_0>0$ such that, for every $0<\epsilon<\epsilon_0$, one can find some constant $C_\epsilon>0$ so that 
$$\forall T>0,\quad\inf\left\{|F|:F\subset M\ \text{and}\ \bigcup_{x\in F}B(x,\epsilon,T)=M\right\}\leqslant C_\epsilon e^{hT}.$$
Fix now some $T>0$ and some $0<\epsilon<\epsilon_0$. We let $F\subset M$ so that the infimum is attained in the previous inequality. We decompose $\Sigma(c_2)$ as follows
$$\Sigma(c_2)=\bigcup_{x\in F}\Sigma_2(x,\epsilon,T),$$
where
$$\Sigma_2(x,\epsilon,T):=\Sigma(c_2)\cap B(x,\epsilon,T).$$
We fix some conic neighborhood $\mathbf{C}_u$ of $E_u\setminus \underline{0}$ so that it does not intersect $T\Sigma(c_1)$. From the Anosov assumption and from the transversality assumption~\eqref{e:transversality-stable}, we know that there exists some $T_1>0$ such that for every $t\geqslant T_1$, $d\varphi^t(T\Sigma(c_2))\subset\mathbf{C}_u$. Observe that for $\varepsilon$ small enough, for every \emph{small} piece of curve $\tilde{\Sigma}$ so that $T(\tilde{\Sigma})$ is contained in $\mathbf{C}_u$ and $\tilde{\Sigma}$ is contained in a ball of radius $\varepsilon$, then $\tilde{\Sigma}$ intersects $\Sigma(c_1)$ at most at one point thanks to~\eqref{e:transversality-unstable}. Still thanks to our transversality assumptions, one can verify that there exists some integer $p_0$ (depending only on the cone and on $\Sigma(c_1)$) so that, for every $\tilde{\Sigma}$ such that $T(\tilde{\Sigma})$ is contained in $\mathbf{C}_u$
\begin{equation}\label{e:return}
\tilde{\Sigma}\cap\Sigma(c_1)\neq\emptyset\Longrightarrow \forall 0<t\leqslant p_0^{-1},\ \varphi^t(\tilde{\Sigma})\cap\Sigma(c_1)=\emptyset.
\end{equation}
In particular, if $\Sigma_2(x,\epsilon,T)\cap\varphi^{-t}(\Sigma(c_1))\neq\emptyset$ for some $t\geqslant T_1$, then
$$\left|\left(\Sigma_2(x,\epsilon,T)\cap\varphi^{-t}(\Sigma(c_1))\right)_{T\leqslant t<T+1}\right|\leqslant p_0.$$ As the cardinal of $F$ is $\leqslant C_{\epsilon} e^{hT}$, we finally find that, for every $T\geqslant T_1$,
$$\sum_{t\in[T,T+1)}m_{c_1,c_2}(t)\leqslant C_{\epsilon}p_0 e^{hT},$$
which concludes the proof of the Lemma.
\end{proof}

\section{Proof of Theorem~\ref{t:meromorphic}}\label{s:zeta}

Let $c_1$ and $c_2$ be two closed geodesics (including the case of points) and denote by $\Sigma_1:=\Sigma(c_1)$ and $\Sigma_2:=\Sigma(c_2)$ their Legendrian lifts as in Section~\ref{a:geometry}. See Remarks~\ref{r:fiber} and~\ref{r:geodesic}. With the convention of the previous section, we set, for $T_0\geq 0$,
$$\zeta_{\Sigma_1,\Sigma_2}(z):=\varepsilon(\mathbf{c}_2)\sum_{t>0: \varphi^{-T_0}(\Sigma_1)\cap\varphi^{t}(\Sigma_2)\neq\emptyset}\sharp\left\{x\in \varphi^{-T_0}(\Sigma_1)\cap\varphi^{t}(\Sigma_2)\right\}e^{-zt}.$$
If $T_0\geq 0$ is chosen small enough, then the relation with the function from the introduction is
$\mathcal{N}_\infty(c_2,c_1,z)=\varepsilon(\mathbf{c}_2)e^{zT_0}\zeta_{\Sigma_1,\Sigma_2}(z)$. More precisely, if $\Sigma_1\cap\Sigma_2=\emptyset$, we can take $T_0=0$. Otherwise, we take $T_0>0$ some small enough to ensure that $\varphi^{-T_0}(\Sigma_1)\cap \Sigma_2$ is empty. Recall that this sum is holomorphic for $\text{Re}(z)>h_{\text{top}}$ thanks to Lemma~\ref{l:exp-growth}.

We are now ready to prove Theorem~\ref{t:meromorphic} and the only missing ingredient is the input given by the theory of Pollicott-Ruelle resonances that we will  briefly review. More precisely, we will rely on the microlocal methods that were initiated in~\cite{FRS, FaSj, DyZw13}. Yet, it is plausible that a similar result could be derived using the geometric methods from~\cite{BL07, GLP} or the coherent states approach of~\cite{Ts12, FaTs}. We also note that Chaubet recently adapted our argument to surfaces with boundary and that he provided a proof of this meromorphic extension relying only on wavefront sets arguments~\cite{Ch21} without introducing the formalism of anisotropic Sobolev spaces as we are doing here.

\begin{rema}\label{r:margulis} As we shall see in our proof, the poles of this meromorphic function are included in the set of Pollicott-Ruelle resonances for currents of degree $2$~\cite{BL07, BL13, FaSj, GLP, DyZw13}. Recall from~\cite[Prop.~4.9]{GLP} that the real parts of the resonances are in that case $\leqslant h_{\text{top}}$. Moreover, as the geodesic flow on a negatively curved surface is topologically mixing, then $h_{\text{top}}$ is a simple eigenvalue and it is the only resonance on the axis $h_{\text{top}}+i\IR$. Using the inverse Laplace transform, this would allow to recover Margulis' asymptotic formula~\eqref{e:margulis} in that framework.
\end{rema}

\begin{rema} As was pointed to us by Guedes-Bonthonneau, Theorem~\ref{t:meromorphic} can be thought of as an analogue in the context of Pollicott-Ruelle resonances of the Kuznetsov trace formulas for the Laplace-Beltrami operator~\cite{Ze92}. In that context, Zelditch considered the spectral projector of the Laplacian on the eigenvalues $\leqslant\lambda$ and he integrated the kernel of the operator against singular distributions carried by smooth submanifolds. Here, we will do the exact same thing with the resolvent of our operator. Yet, compared with that reference, we need to restrict ourselves to certain families of submanifolds verifying our transversality assumptions~\eqref{e:transversality-unstable} and~\eqref{e:transversality-stable}, to ensure that they satisfy the appropriate wavefront set conditions so that they can be integrated against the Schwartz kernel of our resolvent. 
\end{rema}

\subsection{Anisotropic Sobolev spaces}

 Let us denote by $\mathcal{L}_V=d\iota_V+\iota_Vd$ the Lie derivative along the geodesic vector field $V$. For every $0\leqslant k\leqslant  3$, the map
$$R_k(z):= (\mathcal{L}_{V}+z)^{-1}=\int_0^{+\infty}e^{-tz}e^{-t(\mathcal{L}_{V})}|dt|:\Omega^k(M)\rightarrow\mathcal{D}^{\prime k}(M)$$
is well defined and holomorphic in some region $\text{Re}(z)\geqslant C_0$ for some $C_0>0$ depending on $M$ and $V$. Here $|dt|$ is understood as the Lebesgue measure on $\IR$ in order to distinguish with currents of integration. It follows from the works of Butterley-Liverani~\cite{BL07, BL13}, Faure-Sj\"ostrand~\cite{FaSj}, Giulietti-Liverani-Pollicott~\cite{GLP} and Dyatlov-Zworski~\cite{DyZw13} that this resolvent admits a meromorphic extension to the whole complex plane. The poles are the so-called Pollicott-Ruelle resonances and the residues are given by spectral projectors. Such a property was obtained by defining appropriate Banach spaces with anisotropic regularity properties and we briefly describe in this paragraph the anisotropic Sobolev spaces introduced by Faure--Sj\"ostrand~\cite{FaSj} via microlocal methods -- see also~\cite{FRS} for an earlier construction of Faure-Roy-Sj\"ostrand in the case of diffeomorphisms and~\cite{DyZw13} for the extension to the case of currents by Dyatlov-Zworski as we need here.

\subsubsection{Anisotropic Sobolev spaces} Let $0\leqslant k\leqslant n$. Recall that we have a scalar product on $\Omega^k(M)$ by setting, for every $(\psi_1,\psi_2)\in \Omega^k(\ml{M})$,
\[\langle \psi_1,\psi_2\rangle_{L^2}:=\int_{M}\langle \psi_1,\overline{\psi_2}\rangle_{g^*}{\rm dvol}_g,\]
where $g^*$ is the metric induced by $g$ on $k$-forms. We set $L^2(M,\Lambda^k(T^*M))$ (or $L^2(M)$ if there is no ambiguity) to be the completion of the (complex-valued) $k$-forms $\Omega^k(M)$ for this scalar product. Recall that the set of De Rham currents of degree $k$ (the topological dual to $\Omega^{n-k}(M)$) is denoted by $\ml{D}^{\prime k}(M)$. It was shown in~\cite{BL07, BL13, FaSj, GLP, DyZw13} that $\mathcal{L}_V$ has a discrete spectrum when acting on convenient Banach spaces of currents of degree $k$. 

Let us recall the definition of these spaces in the microlocal framework of Faure and Sj\"ostrand.
Following these authors, we can define anisotropic Sobolev spaces $\mathcal{H}_k^{m_{N_0,N_1}}$ of degree $k$ currents, $N_1\gg N_0>0$ large enough, which satisfy the continuous inclusion properties 
$$\Omega^k(M)\subset \mathcal{H}_k^{m_{N_0,N_1}}\subset \mathcal{D}^{\prime,k}(M) ,$$ 
where elements in $\mathcal{H}_k^{m_{N_0,N_1}}$ are microlocally of Sobolev regularity $\leqslant-N_0$ in some conical neighborhood of $E_u^*$, of Sobolev regularity $\geqslant \frac{N_1}{8}$ outside some larger neighborhood of $E_u^*$ and of positive Sobolev regularity $\geqslant N_1$ in some small conical neighborhood of $E_s^*$. We refer to~\cite[\S3]{DGRS18} for a more detailed exposition with conventions similar to ours. In fact, a key point for our analysis is that the currents of interest for our analysis belongs to these anisotropic spaces or to its dual. See~\eqref{e:stable-current-anisotropic} and~\eqref{e:unstable-current-anisotropic} below for precise statements.

\begin{rema}\label{r:semigroup} 
When proving Theorem~\ref{t:meromorphic}, we will need the following a priori estimate: there exists $C_0>0$ such that, for every $t\geqslant 0$,
\begin{equation}\label{e:semigroup}\left\|\varphi^{-t*}\right\|_{\ml{H}^{m_{N_0,N_1}}_k\rightarrow \ml{H}^{m_{N_0,N_1}}_k}\leqslant e^{tC_0}.
\end{equation}
Using classical results from semigroup theory on Banach spaces~\cite[Cor.~3.6, p.~76]{EN}, such a bound can be obtained from the fact that
$$-\ml{L}_V:\ml{D}(\ml{L}_V)\subset\ml{H}^{m_{N_0,N_1}}_k(M)\rightarrow\ml{H}^{m_{N_0,N_1}}_k(M)$$
is closed, densely defined and verifies the resolvent estimate: for every $\text{Re}(z)>C_0$, one has
$$\left\|(\ml{L}_V+z)^{-1}\right\|_{\ml{H}^{m_{N_0,N_1}}_k\rightarrow \ml{H}^{m_{N_0,N_1}}_k}\leqslant\frac{1}{\text{Re}(z)-C_0}.$$
The closedness property follows from~\cite[Lemma A.1]{FaSj} while the density property follows from the density of $\Omega^k(M)$ in our anisotropic spaces and the resolvent estimate was proved in~\cite[Proof of Lemma~3.3]{FaSj} (adapted to the case of currents).

\end{rema}

 In~\cite[Th.~1.4]{FaSj} (see~\cite[\S 3.2]{DyZw13} for the case of currents), it is shown that $(\ml{L}_V+z):\ml{D}(\ml{L}_V)\rightarrow \ml{H}^{m_{N_0,N_1}}_k(M)$ is a family of Fredholm operators of index $0$ depending analytically on $z$ in the region $\{\text{Re}(z)>C_0-c_0N_0\}$ for some constants $C_0,c_0>0$ that are independent of $N_0$ and $N_1$. Then, the poles of the meromorphic extension are the eigenvalues of $-\ml{L}_V$ on $\ml{H}_k^{m_{N_0,N_1}}(M)$, the so-called \emph{Pollicott-Ruelle resonances}. The residues at each pole are the corresponding spectral projectors, and the range of each spectral projector generates the \emph{Pollicott-Ruelle} resonant states.

\subsubsection{Dual spaces} The dual space $(\ml{H}^{m_{N_0,N_1}}_k(M))^\prime$ to $\ml{H}^{m_{N_0,N_1}}_k(M)$ is denoted by $\ml{H}^{-m_{N_0,N_1}}_{3-k}(M)$
via some slight abuse of notations. It is a Sobolev space of currents of degree $3-k$ since we use the wedge product to pair it with elements from $\ml{H}^{m_{N_0,N_1}}_k(M)$. 
Elements in the dual have positive Sobolev regularity in a small conic neighborhood of $E_u^*$ and negative Sobolev regularity outside a slightly bigger neighborhood. The duality pairing is then given, for every $(\psi_1,\psi_2)\in\ml{H}^{-m_{N_0,N_1}}_{3-k}(M)\times\ml{H}^{m_{N_0,N_1}}_{k}(M),$
$$\langle\psi_1,\psi_2\rangle_{\ml{H}^{-m_{N_0,N_1}}_{3-k}(M)\times\ml{H}^{m_{N_0,N_1}}_{k}(M)}=\int_M\psi_1\wedge\overline{\psi_2}.$$
The operator dual to $-\ml{L}_V$ is given by $-\mathcal{L}_{-V}.$ Finally, let, for $T\geq 0$, 
$$\Sigma^T_1:=\varphi^{-T-T_0}(\Sigma_1),\quad\text{and}\quad \Sigma^{-T}_2:=\varphi^{T}(\Sigma_2),$$
where $T_0$ was fixed in the definition of the zeta function.
For $T>0$ large enough, the wave front set of $[\Sigma^T_1]=\varphi^{(T+T_0)*}([\Sigma_1])$ is contained in $\Gamma_s\subset T^*M$ which is a small conical neighborhood of $E_s^*\oplus E_0^*$. See Appendix~\ref{a:WF} for a brief reminder on wavefront sets.
We also note that the curve $\varphi^{-T-T_0}\left( \Sigma_1\right)$ in $M$ has codimension $2$. Therefore the corresponding current of integration $[\Sigma^T_1]$ has negative Sobolev regularity of order $<-1 $. More precisely, it has microlocal Sobolev regularity $<-1$ in $\Gamma_s$.     
If we choose $N_1$ large enough, elements of $\mathcal{H}_2^{-m_{N_0,N_1}}(M)$ will have Sobolev regularity of order $\leqslant-\frac{N_1}{8}<-1$ (by duality with $\mathcal{H}_1^{m_{N_0,N_1}}(M)$) outside some neighborhood of $E_u^*$ that contains $\Gamma_s$. Therefore, there exists $T_1\geq 0$ (depending on the choice of the anisotropic Sobolev spaces) such that, 
\begin{equation}\label{e:stable-current-anisotropic}
 \forall T\geqslant T_1,\quad [\Sigma^T_1]\in\ml{H}^{-m_{N_0,N_1}}_{ 2}(M)=(\ml{H}^{m_{N_0,N_1}}_1(M))^\prime.
\end{equation}

\subsection{Proof of Theorem~\ref{t:meromorphic}}\label{ss:proof-theo1}

First by transversality properties of $\Sigma_1,\Sigma_2$ proved in Lemma~\ref{l:transversality-geodesic}, the set of times $\{t>0, \Sigma_1\cap \varphi^{t}\left(\Sigma_2\right)\neq \emptyset\}$ is isolated without accumulation points. Hence
it is sufficient to prove the analytic continuation of 
\begin{multline*}\zeta^T_{\Sigma_1,\Sigma_2}(z):=\varepsilon(\mathbf{c}_2)\sum_{t\geq 2T:\Sigma_1\cap \varphi^{t}\left(\Sigma_2\right)\neq \emptyset} \sum_{x\in \Sigma_1\cap \varphi^{t}\left(\Sigma_2\right)}e^{-tz}\\
=
\varepsilon(\mathbf{c}_2)\sum_{t\geq 0:\Sigma_1^T\cap \varphi^{t}\left(\Sigma_2^{-T}\right)\neq \emptyset} \sum_{x\in \Sigma_1^T\cap \varphi^{t}\left(\Sigma_2^{-T}\right)}e^{-tz}
\end{multline*}
for any $T>0$ large enough. Assume now that $T>0$ is chosen in such a way that $\Sigma^T_1\cap \Sigma^{-T}_2=\emptyset$ and choose some $\delta>0$ such that, for all $t\in [0,2\delta]$, the intersection $\Sigma_1^T\cap \varphi^{t}\left(\Sigma_2^{-T}\right)$ remains empty. Let $\chi\in\ml{C}^{\infty}(\IR,[0,1])$ be a nondecreasing function which is equal to $1$ on $[2\delta,\infty)$ and to $0$ on $(-\infty,\delta]$. In particular, writing $\chi(t)=\int_{0}^\infty\chi'(t-\tau)|d\tau|,$ one has obviously
$$\zeta^T_{\Sigma_1,\Sigma_2}(z)=\varepsilon(\mathbf{c}_2)
\sum_{t\geq 0:\Sigma_1^T\cap \varphi^{t}\left(\Sigma_2^{-T}\right)\neq \emptyset} \sum_{x\in \Sigma_1^T\cap \varphi^{t}\left(\Sigma_2^{-T}\right)}\left(\int_{0}^\infty\chi'(t-\tau)|d\tau|\right)e^{-tz}.$$
Thanks to Lemma~\ref{l:exp-growth}, we know that, for $\text{Re}(z)$ large enough, we can intertwine the sum and the integral over time. Hence,
$$\zeta^T_{\Sigma_1,\Sigma_2}(z)=\varepsilon(\mathbf{c}_2)\int_{0}^\infty\left(
\sum_{t\geq 0:\Sigma_1^T\cap \varphi^{t}\left(\Sigma_2^{-T}\right)\neq \emptyset} \sum_{x\in \Sigma_1^T\cap \varphi^{t}\left(\Sigma_2^{-T}\right)}\chi'(t-\tau) e^{-tz}\right)|d\tau|,$$
or equivalently
$$\zeta^T_{\Sigma_1,\Sigma_2}(z)=\varepsilon(\mathbf{c}_2)\int_{0}^\infty\left(
\sum_{t\geq 0:\Sigma_1^T\cap \varphi^{t+\tau}\left(\Sigma_2^{-T}\right)\neq \emptyset} \sum_{x\in \Sigma_1^T\cap \varphi^{t+\tau}\left(\Sigma_2^{-T}\right)}\chi'(t) e^{-tz}\right)e^{-\tau z}|d\tau|,$$
where we used that $\chi'$ is compactly supported in $(0,\infty)$. We recognize the quantity from Proposition~\ref{p:WF-finite-time-geodesic} so that
$$\zeta^T_{\Sigma_1,\Sigma_2}(z)=-\int_{0}^\infty e^{-\tau z}\left(\int_M[\Sigma_1^T]\wedge \varphi^{-\tau*}\left(\int_{\IR}e^{-tz}\chi'(t)\varphi^{-t*}\iota_V[\Sigma_2^{-T}]|dt|\right)\right)|d\tau|.$$
For a smooth compactly supported function $\chi_1$ on $\IR$, we set 
$$A_{\chi_1}(z):=\int_{0}^\infty\chi_1'(t)e^{-tz}\varphi^{-t*}|dt|.$$
Recall that the integration current on the submanifold $(\varphi^{T+t}(\Sigma_2))_{0<t<2\delta}$ reads $\int_{0}^{2\delta}\varphi^{-t*}\iota_V([\Sigma_2^{-T}])|dt|$ (up to a sign) as we explained at the end of the proof of Lemma~\ref{l:geomlemm}. One can thus remark that $A_{\chi'}(z)\iota_V([\Sigma_2^{-T}])$ is just a truncated (and weighted) version of this current of integration. In particular, it is a current of order $0$ (in the sense that its action on continuous form is bounded) whose wavefront is carried by the conormal to this submanifold. In particular, if we fix $N_0,N_1$ large enough and if we take $T>0$ large enough to ensure that the wavefront set of $(\varphi^{T+t}(\Sigma_2))_{0<t<2\delta}$ lies inside\footnote{This follows from the hyperbolicity of the flow and from the transversality property~\eqref{e:transversality-unstable}.} a small neighborhood of $E_u^*$, then \begin{equation}\label{e:unstable-current-anisotropic}z\in\IC\mapsto A_{\chi'}(z)\iota_V([\Sigma_2^{-T}])\in\mathcal{H}_1^{m_{N_0,N_1}}(M) 
\end{equation}
 defines an holomorphic function. Hence, by remark~\ref{r:semigroup}, we can rewrite for $\text{Re}(z)$ large enough,
\begin{equation}\label{e:referee}\zeta^T_{\Sigma_1,\Sigma_2}(z)=-\left\langle[\Sigma_1^T],(\ml{L}_V+z)^{-1}A_{\chi'}(z)\iota_V[\Sigma_2^{-T}]\right\rangle_{\left(\mathcal{H}_1^{m_{N_0,N_1}}\right)'\times\mathcal{H}_1^{m_{N_0,N_1}}} .
 \end{equation}
We can then conclude the meromorphic continuation of our Poincar\'e series in the region $\{\text{Re}(z)\geq C_0-c_0N_0\}$ thanks to the meromorphic properties of the resolvent that were recalled above. As this is valid for any $N_0$ large enough, this yields the expected meromorphic continuation to the whole complex plane and ends the proof of Theorem~\ref{t:meromorphic}. 

Relation~\eqref{e:referee} is the main formula that allows to show the meromorphic continuation of Poincar\'e series as it expresses it in terms of a resolvent. However, in view of analyzing the value at $0$, it has the drawback that the geodesics of length $\leq 2T$ are missing in the sum. To handle this problem, let us record the following useful rewriting of our initial zeta function. Using one more time Proposition~\ref{p:WF-finite-time-geodesic} to handle the length between $0$ and $2T$ and recalling that $\varphi^{-T_0}(\Sigma_1)\cap\Sigma_2=\emptyset$, we find that, up to decreasing a little bit the value of $\delta$ from the beginning, there exists $\tilde{\chi}\in\ml{C}^\infty_c((-\delta,2T+\delta),[0,1])$ such that $\tilde{\chi}(t)+\chi(t-2T)=1$ on $\IR_+$ and such that 
\begin{multline}\label{e:zeta-equal-resolvent}\zeta_{\Sigma_1,\Sigma_2}(z)=-\int_M\varphi^{T_0*}[\Sigma_1]\wedge A_{\tilde{\chi}}(z)\iota_V[\Sigma_2]\\
-\int_M[\Sigma_1^T]\wedge (\ml{L}_V+z)^{-1}A_{\chi'}(z)\iota_V[\Sigma_2^{-T}].\end{multline}
In this expression, the first term makes sense as a product between currents on submanifolds that are transversal (in particular their wavefront sets are disjoint) while the second term is a pairing in our anisotropic spaces.

\section{ Behaviour of $\zeta_{\Sigma_1,\Sigma_2}$ at $0$}\label{s:zero-contact}
In this section, we use the same conventions as in Section~\ref{s:zeta}, i.e. we are in the set-up of Theorem~\ref{t:meromorphic}. Recall that the currents $[\Sigma_1]=[\Sigma(c_1)]$ and $[\Sigma_2]=[\Sigma(c_2)]$ are elements of $\mathcal{D}^{\prime 2}(M)$. Given the fact that they are currents of integration over a smooth closed curve, we also note that, for $i=1,2$, 
\begin{equation}\label{e:closed}d[\Sigma_i]=0.\end{equation}

\subsection{Description of the spectral projector at $0$} In order to describe the behaviour of our zeta function at $0$, we need to describe the spectral projector $\pi_0^{(1)}$ at $z=0$. This can be achieved following the recent results of Dyatlov and Zworski~\cite{DyZw} on the behaviour of the Ruelle zeta function at $0$. In particular, this will crucially use the contact structure -- see~\cite{Ha18} for extensions to manifolds with boundary and~\cite{CePa} for extensions to the volume preserving case. For $0\leqslant k\leqslant 3$, we set
$$C^{k}:=\text{Ran}(\pi_0^{(k)})\quad\text{and}\quad C_0^k:=C^k\cap\text{Ker}(\iota_V).$$
According to~\cite[Lemma~7.1]{DGRS18}, one has
$$\forall 0\leqslant k\leqslant3,\quad C^k=C_0^k\oplus (\alpha\wedge C_0^{k-1}),$$
with the convention that $C_0^{-1}=\{0\}$. In~\cite[Lemma~3.2]{DyZw}, it is shown that
$$C^{0}:=\mathbb{C} 1,\quad\text{and}\quad C^2_0:=\IC d\alpha.$$
Still in~\cite[Lemma~3.4]{DyZw}, the authors proved that
\begin{equation}\label{e:cohomology}C_0^1=C^1\cap\text{Ker}(d)\simeq H^1(M,\IC).\end{equation}
In particular, the spectral projector can be written as
$$\forall\psi\in\Omega^1(M),\quad \pi_0^{(1)}(\psi)=\left(\int_M\tilde{\mathbf{S}}_0\wedge\psi\right)\alpha+\sum_{j=1}^{b_1(M)}\left(\int_M \tilde{\mathbf{S}}_j\wedge \psi\right) \mathbf{U}_j,$$
with $d\mathbf{U}_j=\iota_V(\mathbf{U}_j)=0$ for every $1\leqslant j\leqslant b_1(M):=\text{dim}\ H^1(M,\IC)$. Note from~\cite[Thm 1.7 p.~334]{FaSj}
that $\mathbf{U}_j$ (resp. $\tilde{\mathbf{S}}_j$) belongs to $\mathcal{D}^{\prime 1}_{\mathring{E}_u^*}(M)$ (resp. $\mathcal{D}^{\prime 2}_{\mathring{E}_s^*}(M)$) for every $1\leqslant j\leqslant b_1(M)$.

By Poincar\'e duality~\cite[\S 4.6]{DaRi17c}, we can observe that $(\tilde{\mathbf{S}}_0, \tilde{\mathbf{S}}_1,\ldots \tilde{\mathbf{S}}_{b_1(M)})$ forms a basis for the adjoint     operator to $-\ml{L}_{V}$ which is $-\ml{L}_{-V}$ acting on some dual anisotropic space of $2$-forms. Thanks to the fact $\int_M\alpha\wedge d\alpha\neq 0$ and to the fact that $\mathbf{U}_j$ is closed, one already knows that $\tilde{\mathbf{S}}_0=\beta d\alpha$ for some $\beta\neq 0$. Observing now that
$$C^2=C^2_0\oplus \left(\oplus_{j=1}^{b_1(M)}\IC(\alpha\wedge \mathbf{U}_j)\right),$$
and applying~\eqref{e:cohomology} to $-V$ instead of $V$, we find that there exists a family\footnote{It is given by the family of ``dual'' eigenvectors.} $(\mathbf{S}_j)_{j=1,\ldots b_1(M)}$ in $\mathcal{D}^{\prime 1}_{\mathring{E}_s^*}(M)$ such that, for every $1\leqslant j\leqslant b_1(M)$, $\iota_V(\mathbf{S}_j)=0$, $d\mathbf{S}_j=0$ and $\tilde{\mathbf{S}}_j=\alpha\wedge \mathbf{S}_j$. Hence, to summarize, one has
\begin{lemm} For every $\psi\in\Omega^1(M)$,
\begin{equation}\label{e:spectral-projector}
 \pi_0^{(1)}(\psi)=\left(\frac{\int_Md\alpha\wedge\psi}{\int_Md\alpha\wedge\alpha}\right)\alpha+\sum_{j=1}^{b_1(M)}\left(\int_M \alpha \wedge \mathbf{S}_j\wedge \psi\right) \mathbf{U}_j,
\end{equation}
where, for every $1\leqslant j\leqslant b_1(M)$,
\begin{enumerate}
 \item $\mathbf{U}_j\in\mathcal{D}^{\prime 1}_{\mathring{E}_u^*}(M)$, $\mathbf{S}_j\in\mathcal{D}^{\prime 1}_{\mathring{E}_s^*}(M)$,
 \item $d\mathbf{U}_j=d\mathbf{S}_j=0$,
 \item $\iota_V(\mathbf{U}_j)=\iota_V(\mathbf{S}_j)=0$.
\end{enumerate}
\end{lemm}

\subsection{Behaviour at $0$ of the Poincar\'e series} We are now in position to study the behaviour at $0$ of the zeta function $\zeta_{\Sigma_1,\Sigma_2}$ from Section~\ref{s:zeta}. Recall from~\cite[Lemma~3.5]{DyZw} that there is no Jordan blocks in the kernel of the operator. In particular, according to~\eqref{e:zeta-equal-resolvent}, we find that
$$\zeta_{\Sigma_1,\Sigma_2}(z)=-\frac{\left\langle\varphi^{(T+T_0)*}[\Sigma_1],\pi_0^{(1)}A_{\chi'}(0)\iota_V\varphi^{-T*}[\Sigma_2]\right\rangle}{z}+h(z),$$
where $h$ is some holomorphic function near $0$. Recall that one can choose $T_0=0$ as soon as $\Sigma_1\cap\Sigma_2=\emptyset$. Note that, from the definition of $A_{\chi'}$, it commutes with $\varphi^{-T*}$ and $\iota_V$ so that

$$\zeta_{\Sigma_1,\Sigma_2}(z)=-\frac{\left\langle\varphi^{T*}([\Sigma_1]),\pi_0^{(1)}\left(\iota_V \varphi^{-T*}A_{\chi'}(0)([\Sigma_2])\right)\right\rangle}{z}+h(z),$$
where $h(z)$ is a holomorphic function. 
Using now the explicit expression given by~\eqref{e:spectral-projector},~\eqref{e:closed} and the fact that $d\alpha\wedge\iota_V(u)=0$ (for every $2$-form $u$), one finds
\begin{eqnarray*}
 z(h(z)-\zeta_{\Sigma_1,\Sigma_2}(z))&=&\left(\int_Md\alpha\wedge \iota_V \varphi^{-T*}(A_{\chi'}(0)[\Sigma_2])\right)\left(\int_M \varphi^{(T+T_0)*}([\Sigma_1])\wedge\alpha\right)\\
 &+ &\sum_{j=1}^{b_1(M)}\left(\int_M \alpha \wedge \mathbf{S}_j\wedge \iota_V \varphi^{-T*}(A_{\chi'}(0)[\Sigma_2])\right) \left(\int_M\varphi^{(T+T_0)*}([\Sigma_1])\wedge \mathbf{U}_j\right)\\
 &= &-\sum_{j=1}^{b_1(M)}\left(\int_M \mathbf{S}_j\wedge \varphi^{-T*}(A_{\chi'}(0)[\Sigma_2])\right) \left(\int_M\varphi^{(T+T_0)*}([\Sigma_1])\wedge \mathbf{U}_j\right)\\
 &= &-\sum_{j=1}^{b_1(M)}\left(\int_M \mathbf{S}_j\wedge A_{\chi'}(0)[\Sigma_2]\right) \left(\int_M[\Sigma_1]\wedge \mathbf{U}_j\right)\\
 &= &-\sum_{j=1}^{b_1(M)}\left(\int_M \mathbf{S}_j\wedge[\Sigma_2]\right) \left(\int_M[\Sigma_1]\wedge \mathbf{U}_j\right),
\end{eqnarray*}
where we used that $\int_{\IR}\chi'(t)dt=1$ and $\varphi^{t*}\mathbf{S}_j=\mathbf{S}_j$ to derive the last line. To summarize, we have shown:
\begin{prop}\label{p:residue} 
There exist a holomorphic function $h$ (in a neighborhood of $0$) and two families of linearly independent closed currents $(\mathbf{U}_j)_{j=1,\ldots, b_1(M)}$ in $\ml{D}^{\prime 1}_{\mathring{E}_u^*}(M)$ and $(\mathbf{S}_j)_{j=1,\ldots, b_1(M)}$ in $\ml{D}^{\prime 1}_{\mathring{E}_s^*}(M)$ such that
$$\forall 1\leqslant i,j\leqslant b_1(M),\ \int_M\alpha\wedge \mathbf{S}_i\wedge \mathbf{U}_j=\delta_{ij},$$
 and, near $z=0$,
 $$\zeta_{\Sigma_1,\Sigma_2}(z)=\frac{1}{z}\sum_{j=1}^{b_1(M)}\left(\int_M \mathbf{S}_j\wedge [\Sigma_2]\right) \left(\int_M[\Sigma_1]\wedge \mathbf{U}_j\right)+h(z).$$
\end{prop}
Recall that Hodge-De Rham theory shows that the ellipticity of $d$ implies that the cohomology is independent of the choice of the spaces we are working with -- see e.g.~\cite[Lemma~2.1]{DyZw}. In particular, the currents $(\mathbf{U}_j)_{j=1,\ldots, b_1(M)}$ in $\ml{D}^{\prime 1}_{E_u^*}(M)$ and $(\mathbf{S}_j)_{j=1,\ldots, b_1(M)}$ form a basis of $H^1(M,\IC)$. We would like to note that, so far, the only property that was used is the fact that $[\Sigma_1]$ and $[\Sigma_2]$ lie in the appropriate functional spaces. We did not even use that $d[\Sigma_i]$ was equal to $0$. We emphasize that this formula is only valid in dimension $2$ and that its extension to higher dimensions would require a good knowledge of the spectral projector at $0$ in higher dimensions. See e.g.~\cite{CDDP22} for recent progress in that direction.

As a direct Corollary of this Proposition, we find that
\begin{coro}\label{c:exact}  
If either $[\Sigma_1]$ or $[\Sigma_2]$ is exact, then $\zeta_{\Sigma_1,\Sigma_2}(z)$ is holomorphic in a neighborhood of $z=0$. 
\end{coro}

We note that we just needed one of the two currents $[\Sigma_i]$ to be exact, the other current does not even need to be closed.

\section{$\zeta_{\Sigma_1,\Sigma_2}(0)$ and linking numbers} 

\label{s:morse}

Our goal is now to prove Theorem~\ref{t:zero}. The set-up is the same as in Section~\ref{s:zeta} except that we now suppose that the geodesic curves $c_1$ and $c_2$ defining $\Sigma_1$ and $\Sigma_2$ are \emph{homologically trivial}. We begin by collecting a few facts from topology that will be used in our proof. Then, in~\S\ref{ss:proofvalue0}, we relate the value at $0$ with the linking number of $\Sigma_1$ and $\Sigma_2$. Finally, in~\S\ref{ss:proof-euler-zero}, we gather these facts to compute the value at $0$ in terms of Euler characteristics.

\subsection{Three topological ingredients}

\subsubsection{Writing $\Sigma_i$ as the boundary of an explicit surface} 

We assume that $c_i$ is homologically trivial in $X$. Therefore an application of a Gysin type argument~\cite[Prop.~14.33]{BoTu82} implies that the multiple $\chi(X)\Sigma_i$ of the lifted curve $\Sigma_i$ is homologically trivial in $M$. Since the value of $\zeta_{\Sigma_1,\Sigma_2}(0)$ will be interpreted in terms of linking (see~\S\ref{ss:proofvalue0}) number between $\Sigma_1,\Sigma_2$, this implies that $\chi(X)\zeta_{\Sigma_1,\Sigma_2}(0)$ is an integer. 
But our goal is to compute more ``concretely'' the linking between $\Sigma_1,\Sigma_2$. This involves constructing an explicit de Rham primitive $R_2$ of $\Sigma_2$ and then computation the exact intersection number of $R_2$ with $\Sigma_1$.

 In the sequel, we will say that
a smooth vector field $Y\in \Gamma(TX)$ has real hyperbolic zeroes if for every $q$ such that $Y(q)=0$, the linearization of the vector field $Y$ at $q$ is a matrix with \emph{real, non vanishing eigenvalues}. The index $\text{Ind}(q)$ of a critical point is the number of negative eigenvalues.
This hyperbolicity condition immediately implies that the set $\text{Crit}(Y)$ of critical points is isolated and finite. The main example of such vector fields is given by the gradient vector field of a Morse function which is the case we will mostly use in the following. 
 
In the case where $c_i$ is reduced to a point, this reads as follows:
\begin{lemm}\label{l:boundary} Let $Y$ be a smooth vector field on $X$ with real hyperbolic zeroes.
Set 
$$S:=\left\{\left(q,\frac{Y(q)^{\flat}}{\|Y(q)^\flat\|}\right): q\in X\setminus\operatorname{Crit}(Y)\right\}.$$
Then the current of integration $[S]$ on $S^*(X\setminus\operatorname{Crit}(Y))$ extends as a current on $S^*X$ and it satisfies the equation:
\begin{equation}
d [S]=-\sum_{a\in \operatorname{Crit}(Y)}(-1)^{\operatorname{Ind}(a)}[S^*_aX]. 
\end{equation}
\end{lemm}

Before getting to the proof of the Lemma, let us quickly show how to deduce an explicit de Rham primitive of the current of integration on a cotangent fiber $[S^*_aX]$. As $X$ is path-connected, for any pair of points $(a,b)$, consider some smooth oriented path $\gamma_{ab}$ from $a$ to $b$ and denote by $\theta_{ab}=\pi^*([\gamma_{ab}])\in \mathcal{D}^{\prime,1}(S^*X)$. One can verify that $[S_a^*X]=[S_b^*X]+d\theta_{ab}$ which just says cotangent fibers are homotopic. By the Poincar\'e--Hopf Theorem, one has
$\sum_{a\in\text{Crit}(f)}(-1)^{\text{ind}(a)}=\chi(X)\neq 0$ so that
\begin{coro}\label{r:trivial-homology2}  
For every $a\in \operatorname{Crit}(Y)$:
\begin{eqnarray*}
[S^*_aX]=-\frac{1}{\chi(X)}d\left([S]+ \sum_{b\in\operatorname{Crit}(Y)}(-1)^{\operatorname{ind}(b)}\theta_{ab}\right).
\end{eqnarray*}              
\end{coro}

\begin{rema}\label{r:tangent-surface} Given a smooth vector field $q\mapsto Y(q)$ as in the lemma, we can compute the tangent space to the corresponding surface $S$
in the horizontal/vertical bundles of Section~\ref{a:geometry}. Given a curve $x:t\mapsto (q(t),p(t))\in S$ such that $x'(0)\neq 0$, we find that $d_{q(0),p(0)}\Pi(x'(0))=q'(0)\neq 0.$ As this is valid for any curve, this shows that the tangent space to $S$ is transversal to the vertical bundle. 
\end{rema}

\begin{proof} We only need to prove this formula near a fixed critical point $a$ of $Y$. The argument is just a variation of the proof of Stokes formula except that we do not know a priori that $\overline{S}$ is a \emph{smooth} manifold with boundary. We let $\kappa: U\rightarrow \mathbb{R}^2$ be a local chart centered at $a$ (i.e. $\kappa(a)=0$). Using the symplectic lift of $\kappa$, this chart lifts into a chart $\tilde{\kappa}:S^*U\rightarrow \mathbb{R}^2\times\mathbb{S}^1$. In this chart, the boundary $\partial S$ of $S$ is exactly given by $\{0\}\times\mathbb{S}^1$. Similarly, $S$ reads $\{(\tilde{q},\phi(\tilde{q})):\tilde{q}\neq 0\}$ where $\phi:\mathbb{R}^2\setminus \{0\}\rightarrow \mathbb{S}^1$ is a smooth map obtained via the local chart and the initial vector field $Y$. Without loss of generality, we may assume that the image of the submanifold $S$ in $\mathbb{R}^2$ is oriented by the canonical orientation of $\mathbb{R}^2$. As the critical points are of real hyperbolic type, we can always choose the local chart centered at $a$ in such a way that, in the (induced) local coordinates $$Y(\tilde{q}_1,\tilde{q}_2)=\left(\chi_1 \tilde{q}_1+h_1(\tilde{q})\right)\partial_{\tilde{q}_1}+\left(\chi_2 \tilde{q}_2+h_2(\tilde{q})\right)\partial_{\tilde{q}_2},$$ 
with $\chi_1\chi_2\neq 0$ and with $h_1,h_2=\mathcal{O}(\Vert\tilde{q}\Vert^2)$ which are smooth functions defined near $0$. As in the examples of paragraph~\ref{r:ex-orientation}, we define the submanifold of $\mathbb{R}^2\times\mathbb{S}^1$:
$$S:=\left\{(\tilde{q},\phi):\tilde{q}\neq 0,\ \cos\phi=\frac{\chi_1 \tilde{q}_1+h_1(\tilde{q})}{\|Y(\tilde{q})\|},\ \sin\phi=\frac{\chi_2 \tilde{q}_2+h_2(\tilde{q})}{\|Y(\tilde{q})\|}\right\},$$
which is oriented with $d\tilde{q}_1\wedge d\tilde{q}_2$. We set
$$F(\tilde{q},\phi):=\chi_2 \tilde{q}_2+h_2(\tilde{q})-\tan\phi\left(\chi_1 \tilde{q}_1+h_1(\tilde{q})\right).$$
Then, as in this example, one can define, in $\mathbb{R}^2\times(-\pi/2,\pi/2)$
$$[S]=-\frac{\chi_1}{|\chi_1|}\mathbf{1}_{\mathbb{R}_+}(\tilde{q}_1)\delta_0\left(F(\tilde{q},\phi)\right)dF,$$
which extends the current of integration $[S]$ defined on $\mathbb{R}^2\setminus\{0\}\times(-\pi/2,\pi/2)$. Then, taking a partition of unity on $\mathbb{S}^1$ (associated with each half plane of $\IR^2$), one can verify that $[S]$ is well defined on $\IR^2\times\IS^1$. Finally, if we differentiate this expression, we find 
$$d[S]=-\frac{\chi_1\chi_2}{|\chi_1\chi_2|}\delta_0(\tilde{q}_1,\tilde{q}_2)d\tilde{q}_1\wedge d\tilde{q}_2.$$
Recalling that $[S_a^*X]=\delta_0(\tilde{q}_1,\tilde{q}_2)d\tilde{q}_1\wedge d\tilde{q}_2$ was oriented by $d\phi$, we obtain the expected result.
\end{proof}

For general simple closed curves (with maybe several connected components), we get a similar construction

\begin{lemm}\label{l:boundary-conormal} Let $c$ be a simple closed curve in $X$ which is trivial in homology. Denote by $X(c)$ the surface whose oriented boundary is equal to $c$. 

Then, for any vector field $Y$ with real hyperbolic zeroes outside of the curve $c$ that is pointing normally inward $X(c)$, the surface
 $$S:=\left\{\left(q,\frac{Y(q)^{\flat}}{\|Y(q)^\flat\|}\right): q\in X(c)\setminus\operatorname{Crit}(Y)\right\}.$$
defines a current of integration $[S]$ on $S^*(X\setminus\operatorname{Crit}(Y))$ that extends as a current on $S^*X$ satyisfying the equation:
\begin{equation}\label{e:boundary-formula}
d [S]=[\Sigma(c)]-\sum_{a\in \operatorname{Crit}(Y)\cap X(c)}(-1)^{\operatorname{Ind}(a)}[S^*_aX].
\end{equation}
Moreover, such a vector field exists.
\end{lemm}
Again, we recover from this Lemma (and from Lemma~\ref{l:boundary}) that $[\Sigma(c)]$ is trivial in De Rham cohomology with somehow a ``concrete'' expression of a surface whose boundary is $\Sigma(c)$. We emphasize that $c$ does not need to be a closed geodesic here, that $c$ may have several connected components and that the assumption on the curvature of $X$ is useless in that statement. We record the following important geometric fact. By construction, the vector field $Y$ is normal to $c$ and points \emph{inward} $X(c)$ as in the examples of paragraph~\ref{r:ex-orientation}.
\begin{rema} In the following, we always take the conventions that the surfaces $X(c)$ are closed, i.e. they contain their topological boundary. 
\end{rema}

\begin{proof}
Take $\tilde{f}$ to be a smooth function which is constant on $c$ and whose gradient vector field is positively colinear to $c'(t)^{\perp}$. This implies that $\nabla_g \tilde{f}$ has no critical points in some neighborhood of $c$.   
By density of Morse functions in the $\mathcal{C}^\infty$-topology, we can find arbitrarily close to $\tilde{f}$ a smooth Morse function $f$. In particular, its gradient vector field has now finitely many critical points which are all of real-hyperbolic type and which are away from $c$. The gradient vector field $\nabla_gf$ may not be normal to $c$ anymore. 
Take some $\ml{C}^\infty$ cut--off function such that $\chi=1$ near $ c$ and such that $\chi$ is supported in some small tubular neighborhood of $c$. Then $h=\chi \tilde{f}+(1-\chi)f$ is arbitrarily $\ml{C}^1$ close to both $f$ and $\tilde{f}$. The function $h$ is $\mathcal{C}^1$ close to $\tilde{f}$ hence we can choose $f$ and $\chi$ in such a way that $h$ has no critical points in the support of $\chi$. So all the critical points of $h$ are in the region where $\chi=0$ which is the region where $h=f$. Hence $h$ is a Morse function. Arguing as in the proof of Lemma~\ref{l:boundary} (see also~\S\ref{r:ex-orientation}), we find that $Y=\nabla_g h$ has the expected properties. Finally, the proof of Lemma~\ref{l:boundary} did not explicitely used the fact that the vector field was a gradient one and the boundary formula remains true for any vector field verifying the assumption of the Lemma.
\end{proof}

As a consequence of the reminder of the appendix on the wavefront set of the current of integration on a submanifold, we record a complementary Lemma concerning the wave front sets of the primitives produced by both Lemmas~\ref{l:boundary} and~\ref{l:boundary-conormal}:

\begin{lemm}\label{r:WF-bounding-surface} In Lemmas~\ref{l:boundary} and~\ref{l:boundary-conormal}, the resulting current of integration $[S]$ has wavefront set contained in $N^*(S)\cap \cap_{a\in \text{Crit}(Y)}T^*(S_a^*X)\setminus\underline{0}$.
\end{lemm}

\subsubsection{Poincar\'e-Hopf formula for surfaces with boundary} A second key ingredient in our proof is the Poincar\'e-Hopf formula as it was derived by Morse in the case of manifolds with boundary~\cite[Th.~A0]{Mor29}. For simplicity, we only state the formula for surfaces:
\begin{theo}[Poincar\'e-Hopf formula for surfaces with boundary]\label{t:morse} 
Let $c$ be a simple closed and nontrivial curve in $X$ (possibly with several connected components) which is homologically trivial. Denote by $X(c)$ the surface whose oriented boundary is equal to $c$. 

Let $Y$ be a smooth vector field on $X$ with real hyperbolic zeroes and such that $Y$ meets the outgoing normal to $c$ at finitely many points. Denote by $\tilde{Y}$ the vector field induced\footnote{It is obtained by projecting $Y$ on the tangent space to $c$.} by $Y$ on $c$ and suppose that the zeroes $\operatorname{Crit}_{\operatorname{out}}(\tilde{Y})$ of $\tilde{Y}$ where $Y$ is outgoing are also of hyperbolic type. Then, one has
$$\chi(X(c)):=\sum_{a\in\operatorname{Crit}(Y)\cap X(c)}(-1)^{\operatorname{ind}(a)}-\sum_{a\in\operatorname{Crit}_{\operatorname{out}}(\tilde{Y})}(-1)^{\operatorname{ind}(a)}.$$
\end{theo}

Here, we say that $\tilde{Y}$ has hyperbolic zeroes if these are nondegenerate. At some point, we will also need a $1$-dimensional version of this result~\cite[Th.~1]{Mor29}. This reads:

\begin{lemm}\label{r:morse-1d} 
Let $\tilde{Y}$ be a vector field on the interval $[0,1]$ all of whose zeroes are of real hyperbolic type and that is pointing inward $[0,1]$ at $0$ and $1$. Then, one has
$$-1=\sum_{a\in\text{Crit}(\tilde{Y})}(-1)^{\operatorname{ind}(a)}.$$
\end{lemm}

\subsubsection{Hodge-De Rham Theorem} Our final ingredient is the following version of the Hodge-De Rham Theorem in the spaces $\ml{D}^\prime_{\Gamma}(M)$ from Appendix~\ref{a:WF}.
\begin{theo}\label{t:hodge-derham} Let $\Gamma$ be a closed cone, let $0\leq k\leq 3$ and let $u$ be an element in $\ml{D}^{\prime k}_\Gamma(M)$ verifying $du=0$. Then the following holds
\begin{itemize}
 \item there exists $\omega\in\Omega^k(M)$ and $v\in\mathcal{D}^{\prime k-1}_{\Gamma}(M)$ such that $u=\omega+dv$,
 \item if $u=dv$ for some $v\in\mathcal{D}^{\prime k-1}(M)$, then there exists $\omega\in\mathcal{D}^{\prime k-1}_{\Gamma}(M)$ such that $u=d\omega$.
\end{itemize}
\end{theo}

\begin{proof}
A proof of the first point can be found in~\cite[\S2.2]{DyZw}. The proof of the second point is also a consequence of this result. In fact, it tells us that $u=\omega+dW$ for some $\omega\in\Omega^k(M)$ and some $W\in \ml{D}^{\prime k-1}_\Gamma(M)$. Now, from the classical De Rham Theorem~\cite[p.355]{Schwartz-66} on $\mathcal{D}^\prime(M)$, one knows that $\omega=dw_1$ for some $w_1\in\Omega^k(M)$.
\end{proof}

\subsection{The value at $0$ as a linking number}
\label{ss:proofvalue0} 
Using the conventions of~\S\ref{s:zeta} and recalling that we denote by $[\Sigma_i]$ the current of integration over $\Sigma_i=\Sigma(c_i)$, one has
\begin{prop}\label{p:value-at-0} Suppose that $c_1$ and $c_2$ are closed geodesics and that $[\Sigma_1]$ and $[\Sigma_2]$ are exact. Then, for every $T_0>0$ small enough,
 $$\zeta_{\Sigma_1,\Sigma_2}(0)=
 - \int_M \varphi^{T_0*} [\Sigma_1]\wedge R_{2},$$
where $[\Sigma_2]=dR_{2}$ and $R_2\in\mathcal{D}^{\prime 1}_{N^*(\Sigma_2)}(M)$. Moreover, $T_0$ can be chosen equal to $0$ if $\Sigma_1\cap \Sigma_2=\emptyset$.
\end{prop}
Note that 
$$\mathbf{L}(c_1,c_2):=\int_M \varphi^{T_0*} [\Sigma_1]\wedge R_{2}$$
can be understood as the linking number between the two Legendrian knots $\varphi^{-T_0} (\Sigma_1)$ and $\Sigma_2$ and that this choice is independent of the choice of $R_2$ as soon as $[\Sigma_1]$ is also exact. The exact value of this quantity will be computed below in terms of Euler characteristics. We also remark that we do not require the curves $(c_1,c_2)$ to be simple in the above Proposition, they can self--intersect. We do not even really need $[\Sigma_1]$ or $[\Sigma_2]$ to be Legendrian and the only needed properties are the Margulis transversality assumptions~\eqref{e:transversality-unstable} and~\eqref{e:transversality-stable} in view of making sense of the various pairings. In fact, as we shall see below, for any pair of (homologically trivial) knots $\Sigma_1$ and $\Sigma_2$ verifying these transversality assumptions, the value at $0$ of the corresponding Poincar\'e series is, \emph{up to some natural correction term}, a linking number (thus a rational number). Recall from Lemma~\ref{l:transversality} that, for general knots, one may have to consider geodesic trajectories of large enough length (i.e. $T_0>0$ large enough). Thus, the linking number and the value at $0$ may differ by an integer depending the choice of $T_0$ we pick in the definition of the Poincar\'e series.
\begin{proof} To prove this formula, we will start from the integral formula given by~\eqref{e:zeta-equal-resolvent}:
 $$\zeta_{\Sigma_1,\Sigma_2}(z)=-\int_M\varphi^{T_0*}[\Sigma_1]\wedge A_{\tilde{\chi}}(z)\iota_V[\Sigma_2]-\int_M[\Sigma_1^T]\wedge \iota_V (\ml{L}_V+z)^{-1}A_{\chi'}(z)[\Sigma_2^{-T}],$$
 We now decompose $[\Sigma_2]$ using the spectral projector at $0$ which, by duality with~\eqref{e:spectral-projector}, can be written as
 $$\pi_0^{(2)}([\Sigma_2])=\left(\frac{\int_M\alpha\wedge[\Sigma_2]}{\int_M\alpha\wedge d\alpha}\right)d\alpha+\sum_{j=1}^{b_1(M)}\left(\int_M\mathbf{S}_j\wedge [\Sigma_2]\right)\alpha\wedge \mathbf{U}_j.$$
 In particular, using that $[\Sigma_2]$ is exact and Legendrian, the limit as $z\rightarrow 0$ makes sense due to the fact that $\pi_0^{(2)}(A_{\chi'}(z)[\Sigma_2^{-T}])=0$ (as $\pi_0$ commutes with the pullback by the flow). Hence, the formula for the value at $0$ can be expanded as
 \begin{multline*}\zeta_{\Sigma_1,\Sigma_2}(0)=-\int_M\varphi^{T_0*}[\Sigma_1]\wedge A_{\tilde{\chi}}(0)\iota_V[\Sigma_2]
 \\
 -\int_M\varphi^{T_0*}[\Sigma_1]\wedge \varphi^{-T*}\iota_V\ml{L}_V^{-1}A_{\chi'}(0)\varphi^{-T*}[\Sigma_2].
 \end{multline*}

 Thus the natural candidate would be to take
 $$\tilde{R}_2:= A_{\tilde{\chi}}(0)\iota_V[\Sigma_2]+\varphi^{-T*}\iota_V\ml{L}_V^{-1}A_{\chi'}(0)\varphi^{-T*}[\Sigma_2].$$
 As $d[\Sigma_2]=0$ and as $d$ commutes with $\varphi^{-T*}$ and the operators $A_{\psi}$ for all $\psi$, one has in fact
 $$d\tilde{R}_2= A_{\tilde{\chi}}(0)\mathcal{L}_V[\Sigma_2]+\varphi^{-2T*}A_{\chi'}(0)[\Sigma_2].$$
 From the definition of $A_{\tilde{\chi}}$, one has
 $$d\tilde{R}_2= [\Sigma_2]+ A_{\tilde{\chi}'}(0)[\Sigma_2]+A_{\chi'(.-2T)}(0)[\Sigma_2]=[\Sigma_2],$$
 where we used that $\tilde{\chi}(t)+\chi(t-2T)=1$ on $\IR_+$. By taking $N_1$ arbitrarily large in the definition of the anisotropic Sobolev space, one can ensure that the term $\varphi^{-T*}\iota_V\ml{L}_V^{-1}A_{\chi'}(0)\varphi^{-T*}[\Sigma_2]$ has arbitrarily large Sobolev regularity outside a small neighborhood of $E_u^*$. See~\cite{DGRS18} for an explicit description of the relation between the Sobolev exponents and the size of the conical neighborhoods in the definition of the anisotropic Sobolev space. By construction, the term $A_{\tilde{\chi}}(0)\iota_V[\Sigma_2]$ belongs to $\mathcal{D}^{\prime 1}_{\Gamma}(M)$ where the cone $\Gamma$ does not intersect $N^*(\varphi^{-T_0}(\Sigma_1))$. Hence, up to enlarging the cone $\Gamma$ to include the small conical neighborhood of $E_u^*$, $\tilde{R}_2$ belongs to the space $\mathcal{D}^{\prime 1}_{\Gamma}(M)$. Unfortunately, this current does not have the expected wavefront set properties. Yet this problem can be handled as follows. By Theorem~\ref{t:hodge-derham} and as $[\Sigma_2]$ is exact, there exists $R_2\in\mathcal{D}^{\prime 1}_{N^*(\Sigma_2)}(M)$ such that $[\Sigma_2]= dR_2$. Moreover, as $d(R_2-\tilde{R}_2)=0$, one can find some smooth closed one-form $\omega$ and some $\theta\in\ml{D}^{\prime 0}_{\Gamma\cup N^*(\Sigma_2)} (M)$ such that $R_2=\tilde{R}_2+\omega+d\theta$. By Stokes formula, by the fact that $[\Sigma_1]$ is exact, and by noting that $\Gamma\cup N^*(\Sigma_2)$ does not intersect $N^*(\varphi^{-T_0}(\Sigma_1))$, we get the expected result.
\end{proof}

\begin{rema}
In our analysis, the exactness of $[\Sigma_1]$ was only necessary in the final step in view of replacing $\tilde{R}_2$ by the current $R_2$ with appropriate wavefront properties. We did not even use the closedness of $[\Sigma_1]$ before this step. 
\end{rema}

\begin{rema} As it was indicated to us by one of the referee, the above argument did not really need the fact that the curve $\Sigma_2$ is Legendrian if we make the appropriate correction for the value at $s=0$. Indeed, without the Legendrian assumption, one can write 
$$[\Sigma_2]=\left(\frac{\int_{\Sigma_2}\alpha}{\int_M\alpha\wedge d\alpha}\right)d\alpha+(\text{Id}-\pi_0^{(2)})([\Sigma_2]).$$
In particular, using that $\iota_Vd\alpha=0$, one has
\begin{multline*}\int_M\varphi^{T_0*}[\Sigma_1]\wedge A_{\tilde{\chi}}(z)\iota_Vd\alpha\\
+\int_M\varphi^{-T*}[\Sigma_1]\wedge \iota_V (\ml{L}_V+z)^{-1}A_{\chi'}(z)\varphi^{-T*}d\alpha =0.
\end{multline*}
Hence, we can replace $[\Sigma_2]$ by $(\text{Id}-\pi_0^{(2)})([\Sigma_2])$ in the above argument and the exact same proof works except that $d\tilde{R}_2$ will now be equal to $(\text{Id}-\pi_0^{(2)})([\Sigma_2])$ or equivalently
$$d\left(\tilde{R}_2+\left(\frac{\int_{\Sigma_2}\alpha}{\int_M\alpha\wedge d\alpha}\right)\alpha\right)=[\Sigma_2].$$
Hence, arguing as before, we can pick $R_2$ such that $dR_2=[\Sigma_2]$, $R_2\in\mathcal{D}^{\prime 1}_{N^*(\Sigma_2)}(M)$ and
\begin{equation}\boxed{\zeta_{\Sigma_1,\Sigma_2}(0)=-\int_M\varphi^{T_0*}([\Sigma_1])\wedge R_2+\frac{\int_{\Sigma_1}\alpha\int_{\Sigma_2}\alpha}{\int_M\alpha\wedge d\alpha}.}
\end{equation}
This gives the value at $0$ for general curves in $M=SX$. 
\end{rema}

\subsection{Proof of Theorem~\ref{t:zero}}\label{ss:proof-euler-zero} We are now ready to prove our Theorem on the value at $0$ which amounts to compute the value of $\mathbf{L}(c_1,c_2)$. In particular, we now suppose that $c_1$ and $c_2$ are homologically trivial and simple. The proof proceeds in three stages of increasing generality:
1) $(c_1,c_2)$ are points,
2) $c_1$ is a curve and $c_2$ a point, 3) $(c_1,c_2)$ are curves.

 We insist on the structure of our arguments: we always rely on the simpler case to deduce a more general case. Also $c_1$ and $c_2$ are not needed to be geodesic curves. The only thing that matters in the argument below is that they intersect transversally in $X$.

\begin{rema} Thanks to Remarks~\ref{r:tangent-surface} and~\ref{r:WF-bounding-surface} and to Proposition~\ref{p:value-at-0}, we will consider in this proof currents $T_1$ and $T_2$ having disjoint wavefront sets. In particular, according to Appendix~\ref{aa:product}, they will have well defined wedge product and we can write $d(T_1\wedge T_2)=dT_1\wedge T_2-T_1\wedge dT_2$ in view of applying Stokes formula. Thus all the manipulations we will do will be valid thanks to these observations even if we do not repeat them at every stages of the proof. 
\end{rema}

\subsubsection{Proof of Theorem~\ref{t:zero}: the case of points}\label{ss:trivial}

In the case where $c_1$ and $c_2$ are points, the proof of Theorem~\ref{t:zero} follows from the combination of Proposition~\ref{p:value-at-0} with

\begin{lemm}\label{l:constant} Suppose that $c_1$ and $c_2$ are points. Then one has
$$\chi(X)\mathbf{L}(c_1,c_2)=-1\ \text{if}\ c_1\neq c_2,$$
and
$$\chi(X)\mathbf{L}(c_1,c_2)=\chi(X)-1\ \text{otherwise}.$$
\end{lemm}

\begin{rema}Chantraine explained to us that the linking between $S_{c_1}^*X$ and $S_{c_2}^*X$ was equal to the inverse of the Euler characteristic and the Morse theoretic proof given in this paragraph was shown to us by Welschinger.

\end{rema}

\begin{proof} The geodesic curve $c_i$ is reduced to a point $q_i\in X$ for $i=1,2$. Hence, one has $\Sigma(c_i)=S_{q_i}^*X$. Recall from Lemma~\ref{l:boundary} that $[\Sigma(c_i)]$ is exact in that case and that thanks to Proposition~\ref{p:value-at-0}, $[\Sigma(c_i)]=dR_i$ for some $R_i\in\mathcal{D}^{\prime 1}_{N^*(\Sigma(c_i))}(S^*X)$. We shall write $R_i=R_{q_i}$ all along this proof in order to emphasize the dependence on the point $q_i\in X$. We now fix some $0\leqslant T_0<T_1$ such that $S_{q_2}^*X\cap\varphi^{-T_0}(S_{q_1}^*X)=\emptyset$, and we want to compute
$$\mathbf{L}(q_1,q_2)=\int_{S^*X} \varphi^{T_0*}([S_{q_1}^*X])\wedge R_{q_2}.$$

We begin with the case where $q_1\neq q_2$ for which one can take $T_0=0$ in the previous integral. We take $f$ to be a smooth Morse function which has no critical point at $q_1$. We denote its set of critical points by $\text{Crit}(f)$, and we define
$$S:=\left\{\left(q,\frac{d_qf}{\|d_qf\|}\right): q\notin\text{Crit}(f)\right\},$$
which is oriented by the orientation on $X$.

By lemma~\ref{l:boundary}, 
one finds that
$$d[S]=-\sum_{a\in\text{Crit}(f)}(-1)^{\text{ind}(a)}[S_a^*X]=-\sum_{a\in\text{Crit}(f)}(-1)^{\text{ind}(a)}dR_a.$$

As the intersection of $S_{q_1}^*X$ and $S$ is reduced to one point and taking into account the orientation of $S$ and $[S_{q_1}^*X]$ (see also Remark~\ref{r:tangent-surface}), we find that
\begin{multline*}1=\int_{S^*X} [S_{q_1}^*X]\wedge [S]=-\sum_{a\in\text{Crit}(f)}(-1)^{\text{ind}(a)}\int_{S^*X} R_{q_1}\wedge [S_a^*X]\\=-\sum_{a\in\text{Crit}(f)}(-1)^{\text{ind}(a)}\mathbf{L}(q_1,a).
\end{multline*}
Now, if we fix $a\in\text{Crit}(f)$ and if we modify the Morse function $f$ inside a small neighborhood of $a$, we can observe that the map $q\mapsto\mathbf{L}(q_1,q)$ is locally constant on $X\setminus\{q_1\}$ which is connected. Hence, $\mathbf{L}(q_1,a)$ is independent of the choice of $a$. Thanks to the 
Poincar\'e--Hopf formula, this yields
$$-1=\sum_{a\in\text{Crit}(f)}(-1)^{\text{ind}(a)}\mathbf{L}(q_1,a)=\chi(X)\mathbf{L}(q_1,q_2),$$
which concludes the proof when $q_1\neq q_2$.

Suppose now that $q_1=q_2$. In that case, we fix $f$ which has a local minimum at $q_1$ and no other critical points inside the disk bounded by $\Pi(\varphi^{-T_0}(S_{q_1}^*X))$. We also suppose that $f$ is constant on $\Pi(\varphi^{-T_0}(S_{q_1}^*X))$, say $f(q)=d_g(q,q_1)^2$. Then, we define
$$S:=\left\{\left(q,\frac{df(q)}{\|df(q)\|}\right): q\notin\text{Crit}(f)\right\}$$
which does not intersect $\varphi^{-T_0}(S_{q_1}^*X)$. Reproducing the above arguments, this yields
\begin{multline*}0=\int_{S^*X} \varphi^{T_0*}([S_{q_1}^*X])\wedge [S]=-\sum_{a\in\text{Crit}(f)}(-1)^{\text{ind}(a)}\int_{S^*X} \varphi^{T_0*}(R_{q_1})\wedge [S_a^*X]\\
=-\sum_{a\in\text{Crit}(f)}(-1)^{\text{ind}(a)}\mathbf{L}(q_1,a).
 \end{multline*}
Thanks to the case where $q_1\neq q_2$ and to the case of equality in Morse inequalities, we finally obtain 
$$0=-\mathbf{L}(q_1,q_1)+\frac{1}{\chi(X)}\sum_{a\in\text{Crit}(f)\setminus q_1}(-1)^{\text{ind}(a)}=-\mathbf{L}(q_1,q_1)+1-\frac{1}{\chi(X)},$$
which concludes the proof.

\end{proof}

\subsubsection{Proof of Theorem~\ref{t:zero}: 
$c_1$ is a curve, $c_2$ a point}\label{sss:trivial-geod} 
Using this first case, we will now be able to deal with the case where $c_1$ is a simple (non trivial) closed geodesic which is homologically trivial and where $c_2$ is reduced to a point $q_2$. For the sake of simplicity, we suppose that $c_1\cap c_2=\emptyset$ so that we can take $T_0=0$ in the definition of $\mathbf{L}(c_1,c_2)$.

We let $Y$ be the vector field given by Lemma~\ref{l:boundary-conormal} and take $S$ to be the surface defined in that Lemma. One has
$$d[S]=[\Sigma(c_1)]-\sum_{a\in\text{Crit}(Y)\cap X(c_1)}(-1)^{\text{ind}(a)}[S_a^*X].$$
If $q_2$ belongs to $X(c_1)$, one finds that
\begin{multline*}1=\int_{S^*X}[S]\wedge [S_{q_2}^*X]=\int_{S^*X}[\Sigma(c_1)]\wedge R_{q_2}\\
 -\sum_{a\in\text{Crit}(Y)\cap X(c_1)}(-1)^{\text{ind}(a)}\int_{S^*X}[S_a^*X]\wedge R_{q_2}.
\end{multline*}
From Lemma~\ref{l:constant}, we get
$$\int_{S^*X}[\Sigma(c_1)]\wedge R_{q_2}=1-\frac{1}{\chi(X)}\sum_{a\in\text{Crit}(Y)\cap X(c_1)}(-1)^{\text{ind}(a)}.$$
From Theorem~\ref{t:morse}, we find that, for $q_2\in X(c_1)$
$$\mathbf{L}(c_1,c_2)=\int_{S^*X}[\Sigma(c_1)]\wedge R_{q_2}=1-\frac{\chi(X(c_1))}{\chi(X)}.$$
If $q_2\notin X(c_1)$, $\int_{S^*X}[S]\wedge [S_{q_2}^*X]=0$ and the same argument gives
$$\mathbf{L}(c_1,c_2)=\int_{S^*X}[\Sigma(c_1)]\wedge R_{q_2}=-\frac{\chi(X(c_1))}{\chi(X)}.$$

\subsubsection{Proof of Theorem~\ref{t:zero}: both $c_1,c_2$ are curves}
We are now left with the case where both $c_1$ and $c_2$ are nontrivial simple closed geodesics which are homologically trivial. In that case, the difficulty comes from the fact that $c_1$ and $c_2$ may intersect each other and this is where we will use the full strength of the Poincar\'e-Hopf formula given in Theorem~\ref{t:morse}. For simplicity, we also suppose that\footnote{One may have $\sharp c_1\cap c_2=\infty$ is $c_1$ is equal to $c_2$ oriented in the converse sense. Recall that we supposed $\mathbf{c}_1$ and $\mathbf{c}_2$ to be distinct.} $\sharp c_1\cap c_2<\infty$ so that the intersection points are transverse by the geodesic equation and so that we can one more time take $T_0=0$ in the definition of $\mathbf{L}(c_1,c_2)$. For simplicity, we set $X_1:=X(c_1)$ and $X_2:=X(c_2)$ in this paragraph.

\begin{rema}At the end of this paragraph, we will briefly explain how to adapt the argument when $\mathbf{c}_1=\mathbf{c}_2^{-1}$ so that $c_1$ and $c_2$ are twice the same curve but with different orientations.
\end{rema}

We let $Y$ be the vector field given by Lemma~\ref{l:boundary-conormal} with $c=c_1$.

The goal of the next Lemma
will be to deform the vector field $Y$ from Lemma~\ref{l:boundary-conormal}
in such a way that the resulting surface introduced in Lemma~\ref{l:boundary-conormal},
\begin{equation}\label{eq:surface}
S:=\left\{\left(q,\frac{Y(q)^{\flat}}{\|Y(q)^\flat\|}\right): q\in X(c_1)\setminus\operatorname{Crit}(Y)\right\},
\end{equation}
intersects $\Sigma_2$ nicely:

\begin{lemm}\label{l:deformY}  
One can smoothly deform $Y$ in the interior of $X_1$ so that 
\begin{enumerate}
 \item the new vector field has the same number of critical points in $X_{1}$ which are still of real hyperbolic type with preserved index, but they do not lie in $X_1\cap c_2$,
  \item the surface $S$ associated with the resulting vector field intersects $S^*X_1\cap \Sigma_2$ at finitely many points and the intersection is transversal at these points,
 \item the vector field $\tilde{Y}$ induced on $c_2\cap X_1$ (as in Theorem~\ref{t:morse}) has hyperbolic zeroes (both for inward and outward zeroes),
 \item the boundary formula~\eqref{e:boundary-formula}
\begin{eqnarray*}
d[S]=\Sigma_1-\sum_{a\in \operatorname{Crit}(Y)\cap X_1}(-1)^{\operatorname{Ind}(a)}[S^*_aX],
\end{eqnarray*} 
remains true with this new vector field,
 \item one has
 \begin{equation}\label{e:complicated-intersection}\int_{S^*X} [S]\wedge [\Sigma_2]=-\sum_{a\in\operatorname{Crit}_{\operatorname{out}}(\tilde{Y})\cap (X_1\cap c_2)}(-1)^{\operatorname{Ind}(a)}-\frac{1}{2}\chi(c_1\cap c_2),
  \end{equation}
 where $\operatorname{Crit}_{\operatorname{out}}(\tilde{Y})$ denotes the zeroes of $\tilde{Y}$ such that $Y$ points outside $X_2$ at these points.
\end{enumerate}
\end{lemm}

\begin{proof}
 Let us modify the initial vector field $Y$ from Lemma~\ref{l:boundary-conormal} into a new vector field so that the hyperbolic zeroes become disjoint of $c_2$ but their index is preserved. 
We only need to discuss the critical points that potentially belong to $c_2$. For any such point $a$, one can associate $b(a)\in X_1\setminus c_2$. We make the assumption that for every critical points $a\neq a'$ lying on $c_2$, we choose $b(a)\neq b(a')$.
 
 To every pair of points $(a,b(a))$, we associate a smooth curve (with no selfintersection points) $\gamma_a \subset X_1$ joining $a$ to $b(a)$ and such that $\gamma_a$ and $\gamma_{a'}$ do not intersect each other if $a\neq a'$ and never intersect the boundary $c_1$. One can then find tubular neighborhoods $O_a$ of these curves which are diffeomorphic to $\IR\times(0,1)$, which lie inside the interior of $X_1$ and which do not intersect each other. On each of these neighborhoods, one can build a diffeomorphism $\kappa_a$ which sends $a$ to $b(a)$ which is equal to the identity near the boundary of $O_a$. Gluing these ``local'' diffeomorphisms together by taking the identity outside the $O_a$ yields a global diffeomorphism $\kappa:X\mapsto X$. Taking the pullback of the initial vector field $Y$ under $\kappa$, we find a new vector field with the same number of hyperbolic zeroes in $X_1$ and the index of each zero is preserved by the diffeomorphism $\kappa$. The boundary formula~\eqref{e:boundary-formula} is still satisfied since the diffeomorphism $\kappa$ is the identity near $c_1$ so our modification does not affect the validity of equation~\ref{e:boundary-formula}. For simplicity, we still denote by $Y$ the modified vector field.

  In view of verifying (2) and (3), we first note that, as $c_1$ and $c_2$ intersect transversally, the surface $S$ associated with $Y$ does not intersect $\Sigma_2$ above points in $c_1\cap c_2$ (recall that $S$ consists of vectors pointing normally inside $X_1$ above points in $c_1$. Hence, the intersection between $S$ and $S^*X_1\cap\Sigma_2$ occurs necessarily away from $S^*X|_{c_1}$. By a small perturbation\footnote{This amounts to perturb the vector field on a finite number of compact intervals.} of the vector field $Y$ near $c_2$ (and in the interior of $X_1$), we can ensure that the number of intersection points in $S\cap \Sigma_2$ is finite as stated in point (2). We can in fact also ensure that there are only finitely many points where the resulting vector field points normally either inside or outside $X_2$. 

In view of verifying the transversality assumption and the hyperbolicity property, we fix a point $x_0=(q_0,p_0)$ where the new surface $S$ intersects $\Sigma_2$ or where $S$ is pointing normally outside $X_2$. Let $(U_0\subset X,\kappa_0)$ be a small chart near $q_0$. We can choose local coordinates $(\tilde{q}_1,\tilde{q}_2)$ such that $X_2$ is given in this local chart by the local coordinates of \S~\ref{r:ex-orientation}, i.e.
$X_{2}:=\{(\tilde{q}_1,\tilde{q}_2):\tilde{q}_2\geqslant 0\}.$ In this local chart, we know that $Y(0,0)$ is proportional to $\partial_{\tilde{q}_2}$. Hence locally, up to multiplying the vector field by a positive constant near $0$, it reads $Y(\tilde{q})=\pm\partial_{\tilde{q}_2}+f_1(\tilde{q})\partial_{\tilde{q}_1}+f_2(\tilde{q})\partial_{\tilde{q}_2}$ with $f_1(0)=f_2(0)=0$. The $+$ case corresponds to vectors pointing inside $X_2$ (thus intersection points with $\Sigma_2$) while the $-$ case corresponds to outward pointing vectors. If we consider the case of inward pointing vectors, recall using the conventions of~\S~\ref{r:ex-orientation} that we can write locally
$$\Sigma_2:=\left\{(\tilde{q}_1,\tilde{q}_2,\phi):\  \tilde{q}_2=0,\ \phi=\pi/2\right\}$$
while one has locally
$$S=\left\{\left(\tilde{q}_1,\tilde{q}_2,\phi\right):\phi=\phi(\tilde{q}):=\text{arccos}\left(\frac{f_1(\tilde{q})}{\sqrt{f_1(\tilde{q})^2+\left(1+f_2(\tilde{q})\right)^2}}\right)\right\}.$$
In particular, the conormal to $\Sigma_2$ is proportional to $d\tilde{q}_2\wedge d\phi$ while the one to $S$ is given by $d\phi-\partial_{\tilde{q}_1}\phi(\tilde{q})d\tilde{q}_1-\partial_{\tilde{q}_2}\phi(\tilde{q})d\tilde{q}_2$. Hence, the intersection would be transversal if $\partial_{\tilde{q}_1}\phi(0)\neq 0$. If we perturb the vector field in such a way that $\partial_{\tilde{q}_1}f_1(0)\neq 0$ and $\partial_{\tilde{q}_2}f_2(0)\neq 0$, then we get a transversal intersection at $0$. Note that this may result in adding locally extra zeroes to the vector field $\tilde{Y}$ along the curve $c_2\cap X_1$ (if the tangency is of higher order) but we can choose them in such a way that all the intersections with $\Sigma_2$ are transverse and generate hyperbolic zeroes on $c_2\cap X_1$. This procedure works as well for zeroes of $\tilde{Y}$ corresponding to points where $Y$ points normally outside $X_2$.

By partition of unity, we can use this procedure to make the intersection $S\cap \Sigma_2$ transversal at every intersection point. This concludes the proof of point (2) and we note that in the process, we also ensured that point (3) is satisfied. Note also that all along this construction, we can make perturbations that are away from the critical points of $Y$ so that point (1) is still satisfied. As we did not modify the vector field $Y$ near the boundary of $S$, the equation~(\ref{e:boundary-formula}) is also satisfied. 

Hence, we are left with the proof of point (5). We start by working locally near an intersection point and use the above expressions for $\Sigma_2$ and $S$ to write down locally
$$[\Sigma_2]=\delta_0(\tilde{q}_2)\delta_0\left(\phi-\frac{\pi}{2}\right)d\tilde{q}_2\wedge d\phi,$$
and 
$$[S]=\delta_0(\phi-\phi(\tilde{q}))(d\phi-\partial_{\tilde{q}_1}\phi(\tilde{q})d\tilde{q}_1-\partial_{\tilde{q}_2}\phi(\tilde{q})d\tilde{q}_2).$$
Hence, the intersection
\begin{multline*}\int_{S^*U_0}[S]\wedge[\Sigma_{2}]\\
=-\int_{\mathbb{R}^2\times\mathbb{S}^1}\delta_0(\tilde{q}_2)\delta_0\left(\phi-\frac{\pi}{2}\right)\delta_0(\phi-\phi(\tilde{q}_1))\partial_{\tilde{q}_1}\phi(\tilde{q})d\tilde{q}_1d\tilde{q}_2d\phi\\
=-\frac{\partial_{\tilde{q}_1}\phi(0)}{|\partial_{\tilde{q}_1}\phi(0)|}.
\end{multline*} 
Hence, it is equal to $1$ if $\partial_{\tilde{q}_1}f_1(0)>0$ and to $-1$ otherwise. We can now turn things globally and compute
$$\int_M[S]\wedge[\Sigma_2]$$
which is an alternate sum of $+1$ and $-1$ thanks to the above discussion. Each of these contributions correspond to a point of the curve $c_2\cap X_1$ where the new vector field $Y$ points normally inside $X_2$ (here we are using the fact that $\mathbf{c}_2$ is nontrivial). Recall that the vector field $Y$ induces a vector field $\tilde{Y}$ tangent to the curve $c_2\cap X_1$ and each contribution to the integral will come from the points where the induced vector field
$\tilde{Y}$ vanishes and where the vector field $Y$ is pointing inside $X_2$. When the point is attracting (resp. repulsing), it will give a contribution $-1$ (resp. $1$) to the integral thanks to the local calculation above. Hence, one finds
\begin{equation}\label{e:last-intersection}\int_{S^*X}  [S]\wedge [\Sigma_2]=\sum_{a\in\text{Crit}_{\text{in}}\left(\tilde{Y}\right)\cap \left( X_1\cap c_2\right)}(-1)^{\text{ind}(a)},\end{equation}
where $\text{Crit}_{\text{in}}(\tilde{Y})$ is the set of critical points of $\tilde{Y}$ where $Y$ points inside $X_2$. We will now apply the Poincar\'e-Hopf formula of Lemma~\ref{r:morse-1d} to the compact one-dimensional submanifold $X_1\cap c_2$ (this is diffeomorphic to a finite union of compact intervals). The boundary $c_1\cap c_2$ may be non empty and $Y$ is pointing inward on $c_1\cap c_2$ since the initial vector field $Y$ coincides with the inward normal of $X_1$ on $c_1$ and the intersection $c_1\cap c_2$ is transverse. Hence, according to Lemma~\ref{r:morse-1d}, one has
\begin{equation}\label{e:hopf-segment}
\sum_{a\in\text{Crit}\left(\tilde{Y}\right)\cap \left( X_1\cap c_2\right)}
(-1)^{\text{ind}(a)}
=-\frac{1}{2}\chi\left(c_1\cap c_2\right).
\end{equation}
Note that the right-hand side is an integer. Indeed, $c_1$ and $c_2$ are both homologically trivial by assumption. In particular, as the curves are transverse to each other, $\int_X[c_1]\wedge[c_2]=0$ and the two curves intersect each other an even number of times. Equation~\eqref{e:hopf-segment} applied to~\eqref{e:last-intersection} yields
\begin{eqnarray*}\int_{S^*X}  [S]\wedge [\Sigma_2] &=&-\sum_{a\in\text{Crit}_{\text{out}}\left(\tilde{Y}\right)\cap \left( X_1\cap c_2\right)}(-1)^{\text{ind}(a)}-\frac{1}{2}\chi\left(c_1\cap c_2\right).
\end{eqnarray*}
\end{proof}

With the vector field $Y$ given by Lemma~\ref{l:deformY} at hand and the resulting surface $S$ defined in Lemma~\ref{l:boundary-conormal}, we can now compute $\mathbf{L}(c_1,c_2)$ when $c_1$ and $c_2$ are nontrivial simple geodesics and thus conclude the proof of Theorem~\ref{t:zero}. Recalling that $[\Sigma_2]=dR_2$ from Proposition~\ref{p:value-at-0}, one can rewrite the term $\int_{S^*X}  [S]\wedge [\Sigma_2] $ on the left-hand side of~\eqref{e:complicated-intersection} as
\begin{multline*}\int_{S^*X}  [S]\wedge [\Sigma_2] =\int_{S^*X}d[S]\wedge R_2 \text{ by Stokes,}\\
=\underset{\mathbf{L}(c_1,c_2)}{\underbrace{\int_{S^*X}[\Sigma_1]\wedge R_2}}-\sum_{a\in\text{Crit}(Y)\cap X_1}(-1)^{\text{ind}(a)}\int_{S^*X}[S_a^*X]\wedge R_2,
\end{multline*}
where we used equation~(\ref{e:boundary-formula}) given by item (4) from Lemma~\ref{l:deformY}, to write down the second equality. Hence by equation ~\eqref{e:complicated-intersection}, one has
\begin{multline*}
 \mathbf{L}(c_1,c_2)=-\frac{1}{2}\chi\left(c_1\cap c_2\right)-\sum_{a\in\text{Crit}_{\text{out}}\left(\tilde{Y}\right)\cap \left( X_1\cap c_2\right)}(-1)^{\text{ind}(a)}\\+\sum_{a\in\text{Crit}(Y)\cap X_1}(-1)^{\text{ind}(a)}\int_{S^*X}[S_a^*X]\wedge R_2.
\end{multline*}
Equivalently, this can be rewritten using Stokes formula as
\begin{multline*}
 \mathbf{L}(c_1,c_2)=-\frac{1}{2}\chi\left(c_1\cap c_2\right)-\sum_{a\in\text{Crit}_{\text{out}}\left(\tilde{Y}\right)\cap \left( X_1\cap c_2\right)}(-1)^{\text{ind}(a)}\\+\sum_{a\in\text{Crit}(Y)\cap X_1}(-1)^{\text{ind}(a)}\int_{S^*X} [\Sigma_2]\wedge R_a.
\end{multline*}
We can now replace each term $\int_{S^*X} [\Sigma_2]\wedge R_a$ on the right hand side by the value we found for the linking number where one geodesic is reduced to a point. This yields
\begin{multline*}
 \mathbf{L}(c_1,c_2)=-\frac{1}{2}\chi\left(c_1\cap c_2\right)-\frac{\chi(X_2)}{\chi(X)}\sum_{a\in\text{Crit}(Y)\cap X_1}(-1)^{\text{ind}(a)}\\
 -\sum_{a\in\text{Crit}_{\text{out}}\left(\tilde{Y}\right)\cap \left( X_1\cap c_2\right)}(-1)^{\text{ind}(a)}+\sum_{a\in\text{Crit}(Y)\cap X_1\cap X_2}(-1)^{\text{ind}(a)}.
\end{multline*}
The second term $\sum_{a\in\text{Crit}(Y)\cap X_1}(-1)^{\text{ind}(a)}$ on the righthand side can be simplified using Poincar\'e-Hopf formula for surfaces with boundary (using that $Y$ points inside $X_1$):
\begin{multline*}
 \mathbf{L}(c_1,c_2)=-\frac{1}{2}\chi\left(c_1\cap c_2\right)-\frac{\chi(X_2)\chi(X_1)}{\chi(X)}\\
 -\sum_{a\in\text{Crit}_{\text{out}}\left(\tilde{Y}\right)\cap \left( X_1\cap c_2\right)}(-1)^{\text{ind}(a)}+\sum_{a\in\text{Crit}(Y)\cap X_1\cap X_2}(-1)^{\text{ind}(a)}.
\end{multline*}
We would now like to use Theorem~\ref{t:morse} one more time to simplify the last two sums as $Y$ is pointing normally inside $X_1\cap X_2$ on $c_1$. The only problem is that $X_1\cap X_2$ does not have a smooth boundary in general: there are corners at the points in $c_1\cap c_2$. This can be solved by smoothing these corners as illustrated in Figure~\ref{f:intersection}. 
\begin{figure}[ht]
\includegraphics[scale=0.24]{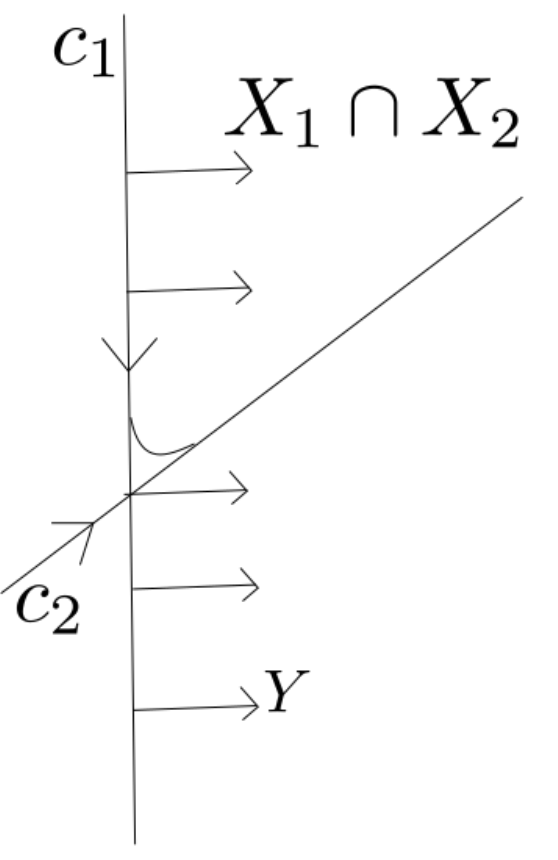}
\centering
\caption{\label{f:intersection}Smoothing corners.}
\end{figure}
As $c_1$ and $c_2$ intersect transversally, $Y$ will not point normally outward the resulting surfaces so that we can apply Poincar\'e-Hopf Theorem. The resulting manifold has the same Euler characteristic as $X_1\cap X_2$ and we can conclude using Theorem~\ref{t:morse} with the smoothed version of $X_1\cap X_2$. This leads to the expected formula:
$$\mathbf{L}(c_1,c_2)=\chi(X_1\cap X_2)-\frac{1}{2}\chi\left(c_1\cap c_2\right)-\frac{\chi(X_2)\chi(X_1)}{\chi(X)}.$$
\begin{rema} The case where $c_1$ and $c_2$ are two simple homologically trivial curves with the same support but not the same orientation, i.e. $X_2=\overline{X\setminus X_1}.$ In that case, the full strength of Lemma~\ref{l:deformY} is not necessary as $S$ do not intersect $\Sigma_2$ so that we can directly write $\int_{S^*X}[S]\wedge[\Sigma_2]$ and derive the result in that case using Stokes formula.
\end{rema}

\section{The case of non simple geodesics}\label{s:intersection}

We note that up to Proposition~\ref{p:value-at-0} the fact that $c_1$ and $c_2$ are simple curves is not used. The relevance of this assumption only appeared at the end of Section~\ref{s:morse} where we needed to work with the surfaces $X(c_i)$ whose oriented boundary is given by $c_i$. We will now discuss the case where $c_1$ and $c_2$ are not anymore simple geodesic curves but still homologically trivial and we will explain how to compute $\mathbf{L}(c_1,c_2)$ in that case using the formalism of constructible functions appearing in symplectic topology. The main statement towards this is Theorem~\ref{t:linking-integer-microlocal} below.

Using the conventions of Proposition~\ref{p:value-at-0}, we want to compute, for $T_0+T_0'>0$ small enough and for two homologically trivial closed geodesics $c_1$ and $c_2$,
$$\mathbf{L}(c_1,c_2)=\int_M\varphi^{(T_0+T_0')*}([\Sigma(c_1)])\wedge R_2,$$
where $[\Sigma(c_2)]=dR_2$. Equivalently, one has
$$\mathbf{L}(c_1,c_2)=\int_M\varphi^{T_0*}([\Sigma(c_1)])\wedge R_2^{-T_0'},$$
where $\varphi^{-T_0'*}([\Sigma(c_2)])=dR_2^{-T_0'}.$ In the following, we shall write things a little bit more compactly by setting $[\Sigma_1^{T_0}]=\varphi^{T_0*}([\Sigma(c_1)])$ which is the current of integration over the smooth submanifold $\varphi^{-T_0}(\Sigma(c_1))$. Similarly, $[\Sigma_2^{-T_0'}]=dR_2^{-T_0'}$ will denote the current of integration over the submanifold $\varphi^{T_0'}(\Sigma(c_2))$. In both cases, we denote by $\tilde{c}_i$, the projection (via the canonical projection) on $X$ of these two curves of $S^*X$. 

\subsection{First properties of the perturbed curves $\tilde{c}_1$ and $\tilde{c}_2$}\label{sss:goodcurves}

The first difficulty is that the new curves may have complicated intersections and selfintersections. This is solved by the following statement:

\begin{prop}\label{p:findingrightcurves}
There exist $T_0>0$ and $T_0'>0$ small enough (with $T_0'$ depending on $T_0$), such that the following properties hold:
\begin{itemize}
 \item for $i=1,2$, one can find some smooth map $\tilde{c}_i:\IR/\ell_i\IZ\rightarrow X$ (with $\ell_i>0$) representing the projected curve $\tilde{c}_i$ and such that $\tilde{c}_i'(t)\neq 0$ for every $t\in  \IR/\ell_i\IZ$;
 \item for $i=1,2$, for every $t\in[0,\ell_i)$,
 \begin{equation}\label{e:selfintersection}
  \sharp\{s\in[0,\ell_i):\ s\neq t\ \text{and}\ \tilde{c}_i(t)=\tilde{c}_i(s)\}\leqslant 1. 
 \end{equation}
In other words, the selfintersections of each curve $\tilde{c}_i$ is made of double points\footnote{This will be referred as a simple selfintersection point.};
 \item if $q_0=\tilde{c}_1(t)=\tilde{c}_2(s)$ for some $(t,s)\in[0,\ell_1)\times[0,\ell_2)$, then $q_0$ is neither a double point of $\tilde{c}_1$, nor of $\tilde{c}_2$.
\end{itemize}
Finally, if $\mathbf{c}_2$ (resp. $\mathbf{c}_1$) is trivial in $\pi_1(X)$, $T_0'$ (resp. $T_0$) can be taken equal to $0$ and\footnote{When $c_1$ and $c_2$ are distinct points, we can take $T_0=T_0'=0$ but we already treated this case in Lemma~\ref{l:constant}.} $T_0>0$ (resp. $T_0'> 0$) such that $\tilde{c}_1\cap c_2=\emptyset$ (resp. $\tilde{c}_2\cap c_1=\emptyset$).
\end{prop}

\begin{proof} Let us explain how to find $T_0$ and $T_0'$ with the above properties. The first point is clear and one only needs to discuss the two other items. We start by acting on $c_1$ (i.e. we will fix the range of $T_0$). We would like to remove all the selfintersections of the curve $c_1$ that correspond to points with multiplicity $>2$. We note that any such point $q_0$ of the curve $c_1$ is isolated in the sense that one can find some $r>0$ such that $B(q_0,r)$ contains no other selfintersection point of $c_1$. 
By picking $T_0>0$ small enough and by applying the flow $\varphi^t$ to the curves $\Sigma(c_1)$ and $\Sigma(c_2)$, we can argue by contradiction to show that any point $q_0$ of multiplicity $>2$ can be transformed in a family of double points. By compactness, this allows to find some $\tilde{T}_0>0$ such that, for any $0<T_0,T_0'\leqslant\tilde{T}_0$, the second property holds for $\tilde{c}_1$ and $\tilde{c}_2$.

We now fix some $T_0>0$ so that the curve $\tilde{c}_1(t)$ has only simple selfintersection points. Taking $T_0'>0$ small enough, we saw that the curve $\tilde{c}_2$ has also only simple selfintersection points that correspond to the perturbation of selfintersection points of the initial curve $c_2$. Then, one can verify that, by eventually taking $T_0'>0$ slightly smaller in a way that depends on $\tilde{c}_1$ (thus on $T_0$), none of these new self intersection points belong to the curve $\tilde{c}_1$ and $\tilde{c}_2$ does not intersect the selfintersection points of $\tilde{c}_1$. We note that, when $\mathbf{c}_2$ is a trivial homotopy class, one can in fact take $T_0'=0$ (but not $T_0=0$ in general) as we can choose $T_0>0$ such that $c_2\notin\tilde{c}_1$.
\end{proof}

\subsection{Decomposing the curve $\tilde{c}_i$ into elementary pieces}\label{sss:decomposition}

Our goal in this paragraph is to provide some algorithm which allows to decompose
each curve $\tilde{c}_1,\tilde{c}_2$ as a union of simple, closed, piecewise smooth curves in view of applying the results of Section~\ref{s:morse} to these elementary curves.
We shall verify afterwards that this decomposition leads to a decomposition of the Legendrian curves $\Sigma_1^{T_0},\Sigma_2^{-T_0^\prime}$ (lifting $\tilde{c}_1,\tilde{c}_2$) to $S^*X$ into a sum of conormals.

Let $i\in\{1,2\}$. Our algorithm is based on the construction of a nice function $f_i:X\mapsto \mathbb{Z}$ which takes constant value in each connected component of $X\setminus\tilde{c}_i$. Even if the next definition may slightly differ from what can be found in the literature~\cite[p.~5]{CGR}, 
such a function is often referred to as a \emph{constructible function}:

\begin{def1}\label{d:fconstruct} Let $i\in\{1,2\}$ and suppose that $\mathbf{c}_i$ is homologically trivial. Write
$$X\setminus\tilde{c}_i=\bigsqcup_{j\in J} \Omega_j,$$
where each $\Omega_j$ is an open connected subset of $X$ and where $\tilde{c}_i$ is the curve from Proposition~\ref{p:findingrightcurves}. We say that $f_i:X\setminus\tilde{c}_i\rightarrow \mathbb{Z}_+$ is a constructible function associated with $\tilde{c}_i$ if
\begin{itemize}
 \item $f_i^{-1}(0)=\Omega_{j_0}$ for some $j_0\in J$;
 \item there exists $x_0\in \Omega_{j_0}$ such that, for every $j\in J$, for every $y\in \Omega_j$, 
 $$f_i(y)=\int_X [\tilde{c}_i]\wedge[\gamma],$$
where $\gamma$ is any smooth path going from $x_0$ to $y$ which is transverse to $\tilde{c}_i$.
\end{itemize}
\end{def1}

From the definition, the existence of such a constructible function is almost immediate when $\mathbf{c}_i$ is homologically trivial.
 Just start by fixing some open connected component $\Omega_{j_1}$ of $X\setminus\tilde{c}_i$ and some point $x_1$ in $\Omega_{j_1}$. Then, given some $j\in J$ and some $y\in\Omega_j$, we define
$\tilde{f}_i(y)=\int_X  [\tilde{c}_i]\wedge[\gamma],$
where $\gamma$ is any smooth path going from $x_1$ to $y$ which is transverse to $\tilde{c}_i$ and the result does not depend on the choice of $\gamma$ by the triviality of $\tilde{c}_i$ in homology. The procedure we follow is illustrated in figure~\ref{f:knot} and, up to adding some constant, we can ensure that the function is nonnegative as expected.
\begin{figure}[ht]
\includegraphics[scale=0.3]{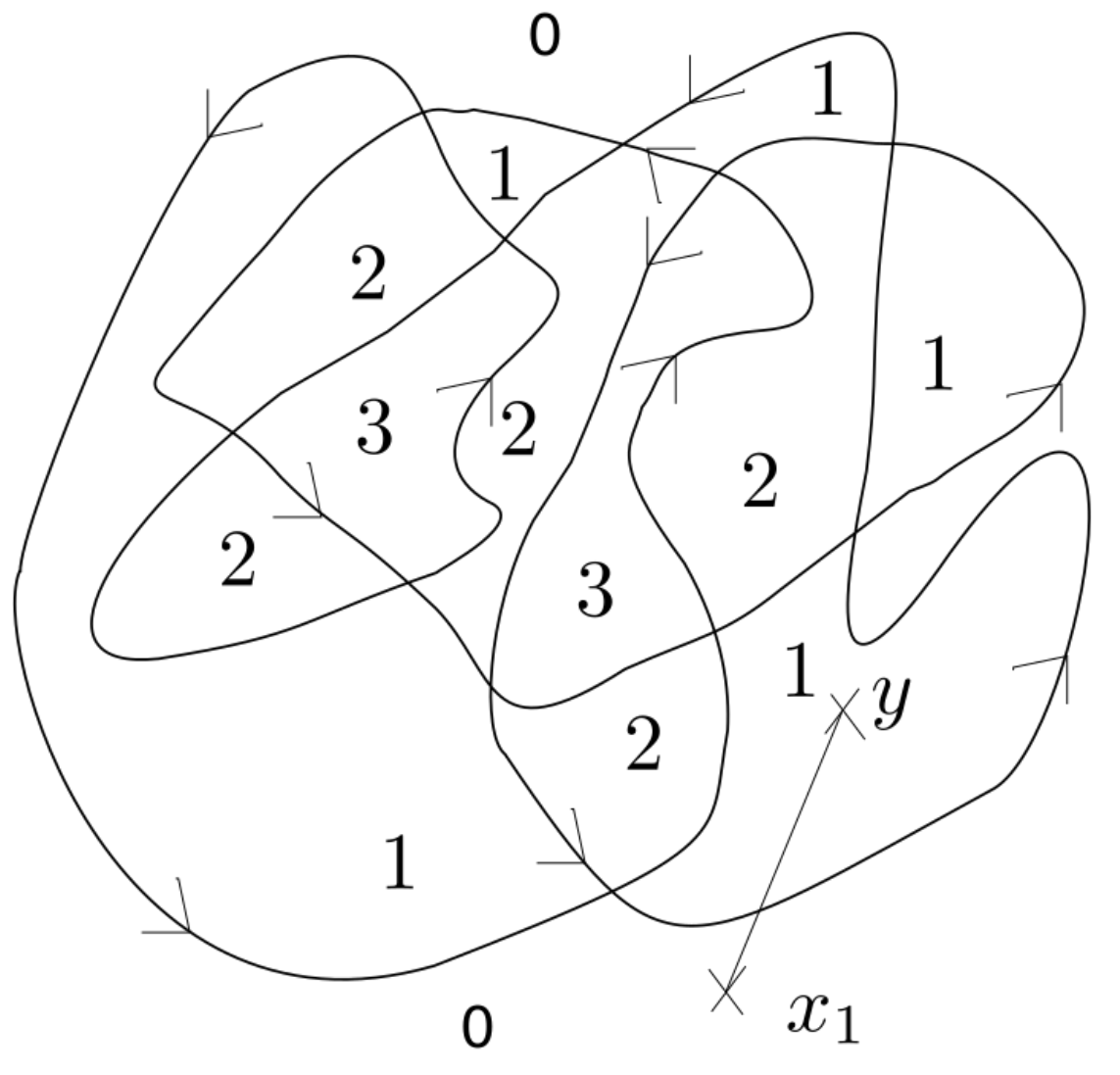}
\centering
\caption{\label{f:knot}Values of a constructible function.}
\end{figure}


Now let us define an algorithm which extracts surfaces from the constructible function $f_i$. These surfaces are going to bound the decomposition of the curve $\tilde{c}_i$ we are looking for. We also note that $f_i$ is defined on $X\setminus \tilde{c}_i$ for the moment. In particular, the sets $U_{i,j}:=\{f_i\geqslant j\}$ are open in $X$ and they have piecewise smooth boundaries. The following construction comes from Euler integration and motion sensing as in~\cite{BG,CGR}. 
Observe first that on $X\setminus \tilde{c}_i$, we have the identity:
$$f_i=\sum_{j=0}^\infty j \mathbf{1}_{\{f_i=j \}}=\sum_{j=1}^\infty \mathbf{1}_{\{ f_i\geqslant j \}}  $$
where both sums are finite since $f_i$ takes finitely many values.
If we set $X_{i,j}=\overline{U_{i,j}}$, we may extend $f_i$ to the whole manifold $X$ by the formula 
$$ f_i= \sum_{j=1}^\infty \mathbf{1}_{X_{i,j}}.$$
Each $X_{i,j}$ is a smooth manifold with piecewise smooth boundary $\tilde{c}_{i,j}:=\partial X_{i,j}$. Note that the singularities of the boundary only occur at the selfintersection points of the curve $\tilde{c}_i$. We have the following 
 chains of inclusions
$$X_{i,\sup(f_i)}\subset\dots\subset X_{i,0}=X.$$
Note that each $\tilde{c}_{i,j}$ is not necessarily connected since our surfaces $X_{i,j}$ may have several boundary components. We shall need the following important observation: 
\begin{lemm}\label{l:keylemmasurfaces} Let $i\in\{1,2\}$. Let $q$ be some element in $\tilde{c}_i$.  
If $q$ is not a selfintersection point of $\tilde{c}_i$, then $q\in \tilde{c}_{i,j}$ for exactly one index $j$. Moreover, in a neighborhood of such a point, one has $d[X_{i,j}]=-[\tilde{c}_{i,j}]$ in the sense of De Rham currents.

Otherwise, there exists $j\geqslant 0$ such that $q\in \tilde{c}_{i,j}\cap \tilde{c}_{i,j+1}$ and $q\notin \tilde{c}_{i,j'}$ if $j'\notin\{j,j+1\}$.
\end{lemm}

\begin{proof} We begin with the case where $q$ is not a selfintersection point.
Consider some small open neighborhood $\Omega$ of $q$ (diffeomorphic to some open ball).
The intersection $\Omega\cap\tilde{c}_i$ is just some open connected interval containing $q$ and having no self--intersection point. Let us consider the restriction of $f_i$ on $\Omega$. The open subset $\Omega\setminus \tilde{c}_i$ is divided into two connected components $\Omega\setminus \tilde{c}_i=\tilde{\Omega}_{j-1}\cup\tilde{\Omega}_{j}$ where $f_i|_{\tilde{\Omega}_{j-1}}=j-1,\ f|_{\tilde{\Omega}_{j}}=j$ and $j\geqslant 1$. We note that $f_i$ takes different values since 
we can choose to cross $\tilde{c}_i$ exactly one time to go from one component $\tilde{\Omega}_{j-1}$ to the other component $\tilde{\Omega}_{j}$. By construction of the surfaces $X_{i,0},\ldots,X_{i,\sup(f)}$, one has 
$$\tilde{\Omega}_{j-1}\subset X_{i,j-1}\subset\ldots\subset X_{i,0},\quad \tilde{\Omega}_{j-1}\cap X_{i,j}=\emptyset.$$
On the other hand, $\tilde{\Omega}_{j}\subset X_{i,j}\subset\dots\subset X_{i,0}$. This implies that $\Omega\cap\tilde{c}_i$ is a subset of a smooth part of $\tilde{c}_{i,j}$. Moreover, $\Omega\cap X_{i,j}=\tilde{\Omega}_{j}$ and one has $d[X_{i,j}]=-[\tilde{c}_{i,j}]$ near this point where the boundary of $X_{i,j}$ is smooth. To see this, it is sufficient to check the formula in the following toy model (which is equivalent to ours in a local chart $(\tilde{q}_1,\tilde{q}_2)$):
$$\tilde{\Omega}_{j}:=\{(\tilde{q}_1,\tilde{q}_2):\tilde{q}_2>0\}\quad\text{and}\quad \tilde{\Omega}_{j-1}:=\{(\tilde{q}_1,\tilde{q}_2):\tilde{q}_2<0\},$$
where $\tilde{c}_{i,j}:=\{(\tilde{q}_1,0)\}$ is oriented by $d\tilde{q}_1$, i.e. $[\tilde{c}_{i,j}]=-\delta_0(\tilde{q}_2)d\tilde{q}_2$. In fact, taking $[\gamma]=\delta_{0}(\tilde{q}_1)d\tilde{q}_1$ (which is oriented by $d\tilde{q}_2$), one finds that the value in the upper half-plane is larger than the value in the lower half plane. Hence, by a direct calculation, one finds that $d[X_{i,j}]=d\mathbf{1}_{\IR_+}(\tilde{q}_2)=\delta_0(\tilde{q}_2)d\tilde{q}_2=-[\tilde{c}_{i,j}]$.

Suppose now that $q$ is a selfintersection point of the curve $\tilde{c}_i$. In that case, the function $f_i$ takes exactly three values on the four connected components of $\Omega\setminus \tilde{c}_i$. By construction of the function $f_i$, these three values are given by $j-1$ (one time), $j$ (two times) and $j+1$ (one time) for some $j\geqslant 1$:
 \begin{center}
  \includegraphics[scale=0.2]{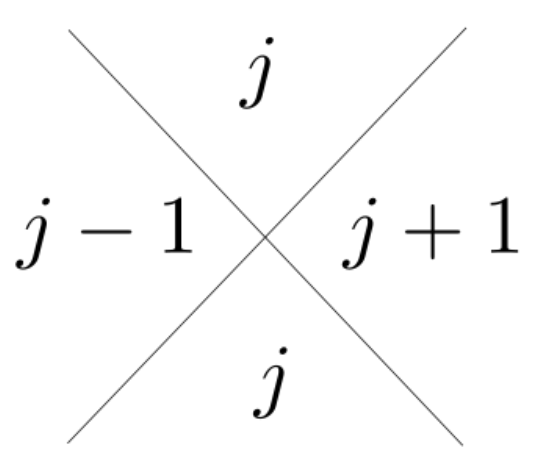}
\end{center}
and one can verify that $q\in \tilde{c}_{i,j}\cap \tilde{c}_{i,j+1}$.
\end{proof}

By construction, we obtain the expected decomposition of the curve $\tilde{c}_i$:
\begin{prop} Let $i\in\{1,2\}$.
We have the following decomposition of the current $[\tilde{c}_i]$:
$$[\tilde{c}_i]=\sum_{j=1}^\infty [\tilde{c}_{i,j}]$$
where each $\tilde{c}_{i,j}=\partial X_{i,j}$ is a finite union of closed, simple and piecewise smooth curves with $X_{i,j}:=\overline{\{ f_i\geqslant j\}}$. Moreover, for every $1\leqslant j\leqslant N_i:=\sup f_i$, one has
$$d\mathbf{1}_{X_{i,j}}=d[X_{i,j}]=-[\tilde{c}_{i,j}],$$
in the sense of De Rham currents.
\end{prop}

Note that using the dual operator $\partial T=-(-1)^{\text{deg}(T)} dT$ on $\mathcal{D}^{\prime}$, this would read equivalently $\partial [X_{i,j}]=[\tilde{c}_{i,j}]$. 
As a consequence , the orientation induced by $X_{i,j}$ on its boundary $\tilde{c}_{i,j}$ is the same as the orientation induced by $\tilde{c}_i$ and each $\tilde{c}_{i,j}$ is cohomologically trivial. Hence, $\tilde{c}_i$ is in this sense the (oriented) boundary of the system of surfaces $X_i=(X_{i,1},\ldots, X_{i,N_i})$. More precisely, in the sense of De Rham currents, one has
\begin{equation}\label{e:decomposition-curve}[\tilde{c}_i]=-\sum_{j=1}^{N_i}d[X_{i,j}]=-df_i.\end{equation}
\begin{proof}
 The first part is a direct consequence of our construction and of Lemma~\ref{l:keylemmasurfaces}. For the second part, it is a consequence of Stokes' Theorem for manifolds with corners~\cite[Th.~16.25]{Lee13}.
\end{proof}

\subsection{Euler characteristics}

\subsubsection{Classical definition of the Euler characteristic}
Let us recall the definition of the Euler characteristic of a CW-complex~\cite[App.~A]{Hat}. Let $\mathbf{X}$ be a space which can be written as a disjoint union of open cells, i.e. $\mathbf{X}=\bigsqcup_{j\in J}\mathbf{X}_j$, each cell being homeomorphic to some $\mathbb{R}^{\text{dim}\mathbf{X}_j}$. Then, the Euler characteristic of $\mathbf{X}$ is given by~\cite[p.~3]{CGR}
$$\chi(\mathbf{X})=\sum_{j\in J}(-1)^{\text{dim}\mathbf{X}_j},$$
which extends the classical formula for polyhedra. In particular, any continuous closed curve (without selfintersection points) on our closed surface $X$ has Euler characteristic equal to $0$. Similarly, any closed domain $X_1\subset X$ with piecewise smooth boundary $\partial X_1$ can be triangulated and it can be decomposed as above. As we are in dimension~$2$, one has
\begin{equation}\label{e:euler-closure}\chi(X_1)=\chi(X_1\setminus \partial X_1)+\chi(\partial X_1)=\chi(X_1\setminus \partial X_1)\quad\text{and}\quad \chi(X\setminus X_1)+\chi(X_1)=\chi(X).
\end{equation}

\subsubsection{Euler characteristics of surfaces and constructible functions}\label{sss:constructible}

As we have just seen, it is equivalent to think of the constructible functions $f_i$ associated with $\tilde{c}_i$ with $i\in\{1,2\}$ as the system of surfaces $X_i=\left(X_{i,1},\dots,X_{i,N_i}\right)$, $X_{i,j}=\{f_i\geqslant j\}$. Note that this system of surfaces with piecewise smooth boundary generates an abstract CW-complex that we denote by $X(\tilde{c}_i)$ and whose ``(oriented) boundary'' is given by $\tilde{c}_i$. This was already expressed more precisely in terms of De Rham currents by equality~\eqref{e:decomposition-curve}. In the case where the initial curve has no selfintersection points, one has $X(\tilde{c}_i)=X(c_i):=X_{i,1}$ which is a surface with smooth boundary. For more general geodesic curves $c_i$, our main formula on the value at $0$ of Poincar\'e series can be extended if we introduce these curves $\tilde{c}_i$ and if we define the appropriate notion of Euler characteristic for the system of surfaces $X_i=\left(X_{i,1},\dots,X_{i,N_i}\right)$ (or equivalently for the CW-complex~$X(\tilde{c}_i)$).

Thus we would like to assign a natural notion of Euler characteristic to the constructible function $f_i$ or equivalently to the system of surfaces $X_i=\left(X_{i,1},\dots,X_{i,N_i}\right)$. Our definition follows the presentation of Euler integration due to Viro~\cite{Viro} and Schapira~\cite{Scha1, Scha2, Scha3}:
\begin{def1}\label{d:Euler}[Euler characteristic of constructible functions]
We define the Euler characteristic of $f_i$ as
\begin{equation}
\chi(f_i):=\sum_{j=0}^{N_i}j \chi\left(\{f_i= j\} \right)=\sum_{j=1}^{N_i} \chi\left(\{f_i\geqslant j\} \right)=\sum_{j=1}^{N_i} \chi(X_{i,j}).
\end{equation}
\end{def1}
Note that the advantage of the second formulation for $\chi(f_i)$ is that the excursion sets $\{f_i\geqslant j\}$ are compact whereas $\{f_i=j\}$ is only relatively compact~\cite[Prop.~4.1]{BG}.  
\begin{rema}
 We can relate this definition with the classical one for CW-complex as follows: $\chi(f_i)=\chi(X(\tilde{c}_i))$. 
\end{rema}
We emphasize that the Euler integral is in fact defined for much more general bounded and constructible functions, $f:X\rightarrow\mathbb{Z}$ whose level sets are tame sets~\cite[\S 4]{CGR}. In that context, one can define 
$$\chi(f):=\int_Xfd\chi=\sum_{j=-\infty}^{+\infty}j\chi(f=j).$$
For instance, we can define the Euler characteristic $\chi(f_1f_2)$ of the product $f_1f_2$ as: 
\begin{multline}\label{e:Euler-product}
\chi(f_1f_2)=\int_Xf_1f_2d\chi=\sum_{1\leqslant j_1\leqslant N_1
}\sum_{1\leqslant j_2\leqslant N_2}\int_X\mathbf{1}_{X_{1,j_1}}\mathbf{1}_{X_{2,j_2}}d\chi\\
=\sum_{1\leqslant j_1\leqslant N_1
}\sum_{1\leqslant j_2\leqslant N_2} \chi(X_{1,j_1}\cap X_{2,j_2}),
\end{multline}
or the Euler characteristic of $\mathbf{1}_{\tilde{c}_1\cap\tilde{c}_2}$ as
$$\chi\left(\mathbf{1}_{\tilde{c}_1\cap\tilde{c}_2}\right)=\sum_{1\leqslant j_1\leqslant N_1
}\sum_{1\leqslant j_2\leqslant N_2} \chi(\partial X_{1,j_1}\cap \partial X_{2,j_2}).$$

\subsubsection{Statement of the main result}

Before going further, we are now ready to state a microlocal statement expressing $\mathbf{L}(c_1,c_2)$ in terms of Euler characteristics of constructible functions. Combined with Proposition~\ref{p:value-at-0}, this yields an extension of Theorem~\ref{t:zero} to any pair of homologically trivial geodesic curves:
\begin{theo}\label{t:linking-integer-microlocal}
Suppose that $c_1$ and $c_2$ are closed geodesics which are homologically trivial and let $\tilde{c}_1$ and $\tilde{c}_2$ be the two (small) homotopic deformations given by Proposition~\ref{p:findingrightcurves}.

Then there exists a pair $(f_1,f_2)$ of constructible functions associated with $(\tilde{c}_1,\tilde{c}_2)$ such that
\begin{equation}\label{e:representationlegendrian}
\sum_{j=1}^\infty [N_1^*\left(\{f_i\geqslant j\}\right)]=[\Sigma(\tilde{c}_i)],
\end{equation}
where the equality holds in the sense of De Rham currents.
Moreover, the linking of the Legendrians is given by the formula
\begin{equation}\label{e:main-linking-microlocal}
\mathbf{L}(c_1,c_2)=-\frac{\chi(f_1)\chi(f_2)}{\chi(X)}+\chi(f_1f_2)-\frac{1}{2}\chi\left(\mathbf{1}_{\tilde{c}_1\cap\tilde{c}_2}\right).
\end{equation}
\end{theo}
We note that, in the case where one of the 
$c_i$ is a point then our perturbed curves are chosen so that $\tilde{c}_1\cap\tilde{c}_2=\emptyset$.
In the terminology of symplectic topology (see \S\ref{ss:microlocal-symplectic} below), we say that the constructible function $f_i$ is quantizing the Legendrian knot $\Sigma(\tilde{c}_i)$, where $\Sigma(\tilde{c}_1)$ (resp. $\Sigma(\tilde{c}_2)$) is the Legendrian knot $\Sigma_1^{T_0}$ (resp. $\Sigma_2^{-T_0'}$) with the conventions of paragraph~\ref{sss:tangent-space}. The rest of this Section is devoted to the proof of this Theorem

\subsection{Lifting everything to $S^*X$}\label{sss:lift}

We would now like to turn the decomposition of the curve $\tilde{c}_i$ into a proper decomposition of $\Sigma(\tilde{c}_i)$. It is convenient to introduce the (unit) conormal bundle of $X_{i,j}$ -- see~\cite[Def. 2.4.1 p.~442]{Al} for the case of more general polyhedra. Recall that, for every vector $v\in T_xX$, we defined in Section~\ref{a:geometry} the covector $v^\flat\in T_x^*X$ as the image of $v$ by the isomorphism induced by the metric $g$ on $X$. In order to define the unit conormal bundle above $X_{i,j}$, we have three kind of points to distinguish:
\begin{itemize}
 \item The points in the interior of $X_{i,j}$. Above such points, the (unit) conormal bundle is obviously empty.
 \item The regular points of $\tilde{c}_{i,j}$. Here, we take the same convention as for $\Sigma(\tilde{c}_i)$, i.e. the points in the unit conormal bundle above some regular point $\tilde{c}_{i,j}(t_0)$ are given by the point
 $$\left(\tilde{c}_{i,j}(t_0),(\tilde{c}_{i,j}'(t_0)^\flat)^\perp\right),$$
 where $\tilde{c}_{i,j}(t)$ is parametrized by arc length.
 \item The singular points of $\tilde{c}_{i,j}$. Again, we take an arc-length (away from the singularities) parametrization $t\mapsto\tilde{c}_{i,j}(t)$  of the curve $\tilde{c}_{i,j}$. Above such a point $\tilde{c}_{i,j}(t_0)$, the derivative $\tilde{c}_{i,j}'(t_0)$ is not well defined. Yet, we have the existence of the two following limits:
 $$\tilde{c}_{i,j}'(t_0+)=\lim_{\tau\rightarrow 0,\tau>0}\tilde{c}_{i,j}'(t_0+\tau)\quad\text{and}\quad\tilde{c}_{i,j}'(t_0-)=\lim_{\tau\rightarrow 0,\tau>0}\tilde{c}_{i,j}'(t_0-\tau).$$
 Then, the conormal bundle above such a point is defined as the \emph{connected} set of unit covectors lying in $S_{\tilde{c}_{i,j}(t_0)}^*X$ and in the cone of cotangent vectors between $(\tilde{c}_{i,j}'(t_0-)^\flat)^\perp$ and $(\tilde{c}_{i,j}'(t_0+)^\flat)^\perp$ intersecting the covectors pointing inward $X_{i,j}$. Here, a covector $p$ is pointing inward $X_{i,j}$ if, for any curve $\gamma$ passing through $\tilde{c}_{i,j}(t_0)$ and cotangent to $p$ at $t=0$, one has $\gamma(t)\in X_{i,j}$ for every $t\geqslant 0$ small enough. See figure~\ref{f:corners}.
\begin{figure}[ht]
\includegraphics[scale=0.34]{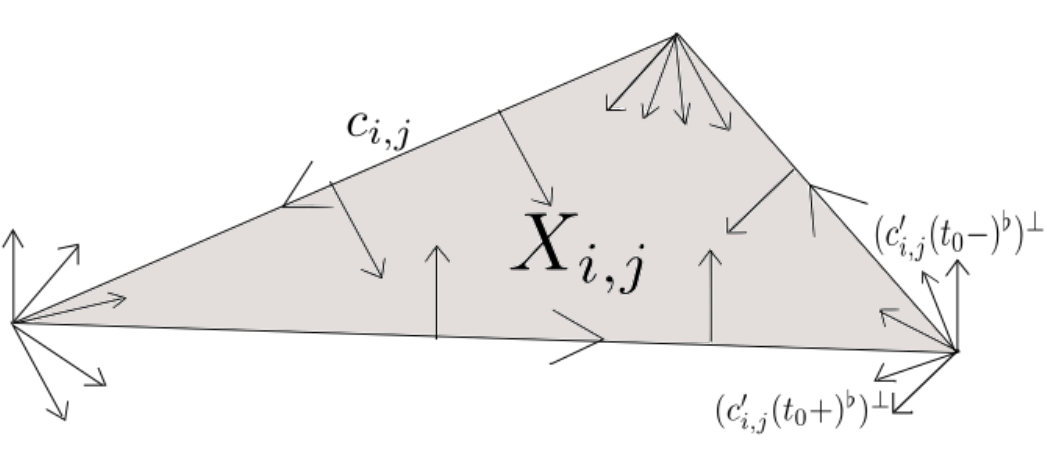}
\centering
\caption{\label{f:corners}Adding covectors at the singular points.}
\end{figure}
\end{itemize}

The union of all these covectors will be referred to as the (unit) conormal bundle to $X_{i,j}$ and we will denote it by $N_1^*(X_{i,j})$. This defines a closed, piecewise smooth and embedded curve in $S^*X$. Even if the curve $c_{i,j}$ is only piecewise $\mathcal{C}^1$, we emphasize that the resulting curve $N_1^*(X_{i,j})$ in $S^*X$ is $\mathcal{C}^1$ (but not $\ml{C}^2$ a priori). This can for instance be verified from the formulas in local coordinates given in Appendix~\ref{ss:proof-smoothing}. The orientation of the curve $N_1^*(X_{i,j})$ is naturally induced from the orientation of $\tilde{c}_i$. In particular, we can define the integration current $[N_1^*(X_{i,j})]$ along this curve and one has $d[N_1^*(X_{i,j})]=0$. We can also note that we still 
have a Legendrian curve, i.e. $[N_1^*(X_{i,j})]\wedge\alpha =0$, where $\alpha$ is the Liouville one-form. 

\begin{rema}
 We remark that, in this construction, we implicitely supposed that $\tilde{c}_i$ was not reduced to a point (i.e. $T_0,T_0'\neq 0$ if $\mathbf{c}_i$ is trivial in $\pi_1(X)$). In the case of a point, the conormal bundle of a point and its orientation were already defined in~\S\ref{ss:transversesubm}.
\end{rema}

Finally, we observe that as soon as one curve $\tilde{c}_{i,j}$ has singular points, the union $\cup_{j=1}^{N_i}N_1^*(X_{i,j})$ is larger than the set $\Sigma(\tilde{c}_i)$ (as it contains more cotangent vectors above each selfintersection point of $\tilde{c}_i$). Yet, in terms of currents, we can verify that the following holds:

\begin{lemm}\label{t:constructiblefun} With the above conventions, one has, in the sense of De Rham currents,
\begin{equation}
[\Sigma_1^{T_0}]=\sum_{j=1}^{N_1} [\Sigma_{1,j}]\quad\text{and}\quad[\Sigma_2^{-T_0'}]=\sum_{j=1}^{N_2} [\Sigma_{2,j}]
\end{equation}
where $[\Sigma_{i,j}]=[N_1^*X_{i,j}]$. 
\end{lemm}

Note that the currents $\Sigma_{i,j}$ depend implicitely on $T_0$ and $T_0'$ but we dropped this dependence to simplify notations. As a corollary of this result, we will be able to compute the linking between $\Sigma_1$ and $\Sigma_2$ in terms of the linking numbers between each elementary piece $\Sigma_{1,j}$ and $\Sigma_{2,j'}$ which are simple closed curves which is a case already treated. See next paragraph for more details.

\begin{proof} We only show $[\Sigma_1^{T_0}]=\sum_{j=1}^{N_1} [\Sigma_{1,j}]$. The other case is similar. Recall first that $[\Sigma_1^{T_0}]$ and $([N^*_1(X_{1,j})])_{1\leqslant j\leqslant N_1}$ are currents of integration over piecewise smooth, simple, closed curves in $S^*X$. For every $1\leqslant j\leqslant N_1$, the oriented curve $N^*_1(X_{1,j})$ coincides with $\Sigma_1^{T_0}$ away from the singularities of $\tilde{c}_{1,j}$. In particular, thanks to Lemma~\ref{l:keylemmasurfaces}, the expected equality holds away from these singularities. Hence, we only need to understand what happens in a neighborhood of such a singularity $q$. Thanks to Lemma~\ref{l:keylemmasurfaces}, the point $q$ belongs to exactly two curves $\tilde{c}_{1,j}$ and $\tilde{c}_{1,j+1}$ for some $j\geqslant 1$. Then, we explicitely see in Figure~\ref{f:conormal} that above the singular point $q$, the
contributions of $N^*_1(X_{1,j})$ and $N^*_1(X_{1,j+1})$ compensate each other, which concludes.
\begin{figure}[ht]
\includegraphics[scale=0.34]{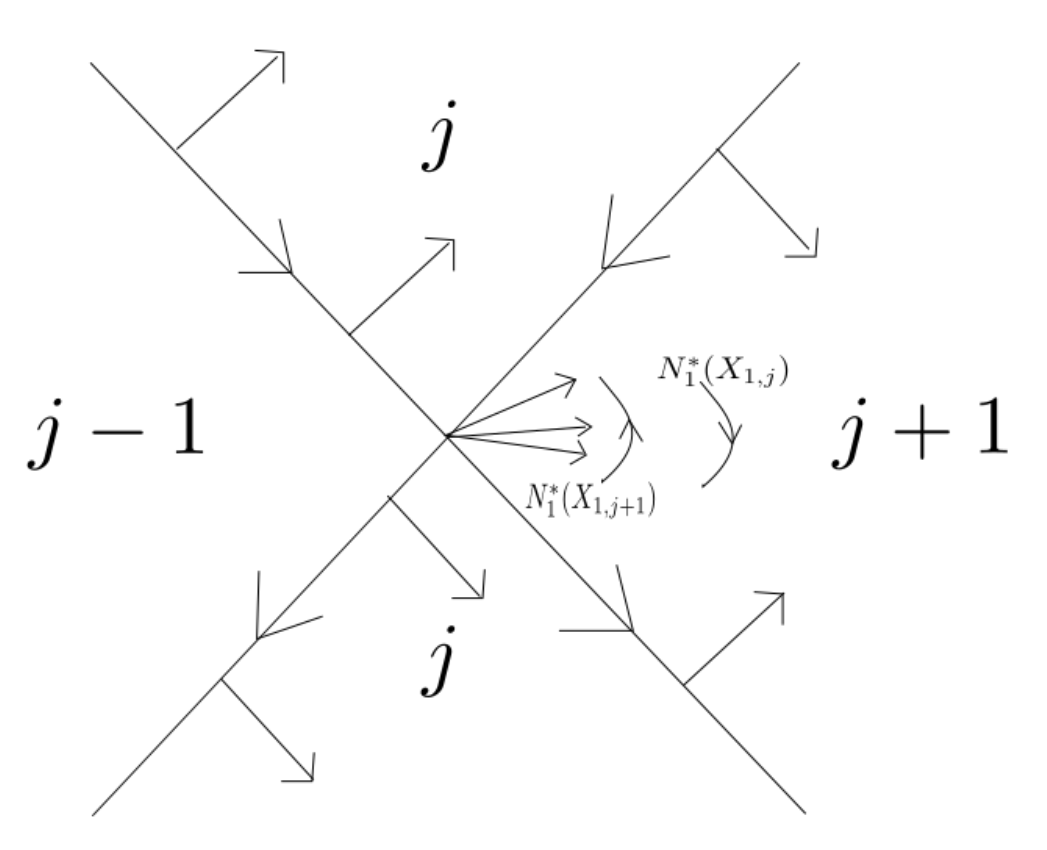}
\centering
\caption{\label{f:conormal}Contributions of $[N^*_1(X_{1,j})]$ and $[ N^*_1(X_{1,j+1})]$ at the singular points.}
\end{figure} 
\end{proof}

\subsection{Consequence of the decomposition for the linking numbers}\label{sss:decomposition-link}

Let us summarize the situation so far and fix some notations for the sequel. We started from our two geodesics $c_1$ and $c_2$ and we applied the geodesic flow to their Legendrian lifts $\Sigma(c_1)$ and $\Sigma(c_2)$. This gives rise to two curves $\tilde{c}_1$ and $\tilde{c}_2$ that are homotopic to $c_1$ and $c_2$ and to a new pair of Legendrian knots $\Sigma_1^{T_0}$ and $\Sigma_2^{-T_0'}$.
Then, we decomposed the curves $\tilde{c}_1$ and $\tilde{c}_2$ as a union of embedded, closed curves which are only piecewise smooth. In terms of De Rham currents, it reads
$$\forall i\in\{1,2\},\quad[\tilde{c}_i]=\sum_{j=1}^{N_i} [\tilde{c}_{i,j}],$$
where each $\tilde{c}_{i,j}$ is the oriented boundary of some surface $X_{i,j}$ with piecewise smooth boundary. In terms of currents, we have $[\tilde{c}_{i,j}]=-d[X_{i,j}]=\partial[X_{i,j}]$. Then, we defined the (unit) conormal bundle $N_1^*(X_{i,j})$ to each surface $X_{i,j}$. This conormal bundle is in fact a Legendrian knot in $S^*X$ (again piecewise smooth) and we denote it by $\Sigma_{i,j}$. This yields the following decompositions of our initial Legendrian knots:
\begin{equation}\label{e:splitting-current}
 [\Sigma_1^{ T_0}]=\sum_{j=1}^{N_1}[\Sigma_{1,j}]\quad \text{and}\quad [\Sigma_2^{- T_0'}]=\sum_{j=1}^{N_2}[\Sigma_{2,j}].
\end{equation}
We can rewrite the quantity we are interested in as
$$\mathbf{L}(c_1,c_2)=\sum_{j=1}^{N_1}\int_{S^*X} [\Sigma_{1,j}]\wedge R_{2}^{-T_0'}.$$
Hence, it remains to evaluate the ``linking number'' associated with every elementary piece $\Sigma_{1,j}$ to conclude the proof of Theorem~\ref{t:linking-integer-microlocal}.

\subsection{Smoothing each elementary piece}
\label{sss:smoothpiece}
In view of proving Theorem~\ref{t:linking-integer-microlocal}, we will use the analysis of \S\ref{s:morse} to reduce our problem to simple curves using the bilinearity of the linking number.
However, when we removed the selfintersections of our initial (smooth) curves $\tilde{c}_1$ and $\tilde{c}_2$, we introduced some families of embedded curves that are only piecewise smooth. We would now like to regularize these new curves without affecting the linking number we want to compute. 

First, by construction of $\Sigma_{i,j}=N_1^*(X_{i,j})$ (see \S~\ref{sss:lift}), we have
$$
\bigcup_{j=1}^{N_1}\text{supp}([\Sigma_{1,j}])\subset \Sigma_1^{T_0}\cup \left(\bigcup_{a\in \text{Cross}(\tilde{c}_1)} S_a^*X\right)$$
and
$$\bigcup_{j=1}^{N_2}\text{supp}([\Sigma_{2,j}])\subset \Sigma_2^{-T_0'}\cup \left(\bigcup_{a\in \text{Cross}(\tilde{c}_2)} S_a^*X\right)
$$
since we added some subset of the cotangent fibers over the selfintersection points $\text{Cross}(\tilde{c}_i)$ of $\tilde{c}_i$. 
Still, by construction of $\Sigma_{i,j}=N_1^*(X_{i,j})$ (see \S~\ref{sss:lift}), the following holds: 
\begin{lemm}
If $\tilde{c}_{i,j}$ has $k$ singular points, then $\Sigma_{i,j}$ is itself a piecewise smooth curve with exactly $2k$ singular points which are isolated. Over each singular point $a$ of $\tilde{c}_{i,j}$, there are exactly two singular points of $\Sigma_{i,j}$.
\end{lemm}

We denote by $\text{Sing}(\Sigma_{i,j})$ this finite subset of singular points. In terms of wavefront sets, this allows to give the simple upper bound:
\begin{multline*}
\forall j, \text{WF}\left([\Sigma_{1,j}] \right)\subset 
\\ 
\subset N^*\Sigma_1^{T_0}\cup \left(\bigcup_{a\in \text{Cross}(\tilde{c}_1)} N^*(S_a^*X)\right)\cup \bigcup_{j',b\in \text{Sing}(\Sigma_{1, j'})} T^*_b(S^*X)\setminus 0,
\end{multline*}
and the same for the part concerning $\Sigma_2=\Sigma_2^{-T_0'}$.
In order to smooth the curves $\tilde{c}_{i,j}$ near their singularities, we fix some conic neighborhood $\Gamma_i$ of 
$$N^*\Sigma_i\cup \left(\bigcup_{a\in \text{Cross}(\tilde{c}_i)} N^*(S_a^*X)\right)\cup \bigcup_{j, b\in \text{Sing}(\Sigma_{i,j})} T^*_b(S^*X)\setminus 0.$$
We begin by observing that
\begin{lemm}\label{r:transverse-WF}
There exists $\Gamma_1,\Gamma_2$ some closed conic subsets of $T^*(S^*X)\setminus\underline{0}$ s.t. $\Gamma_i$ is a conic neighborhood of 
$$N^*\Sigma_i\cup \left(\bigcup_{a\in \operatorname{Cross}(\tilde{c}_i)} N^*(S_a^*X)\right)\cup \bigcup_{j, b\in \operatorname{Sing}(\Sigma_{i,j})} T^*_b(S^*X)\setminus 0$$
and 
\begin{equation}
\Gamma_1\cap\Gamma_2=\emptyset.
\end{equation}
\end{lemm}
\begin{proof}
Thanks to the hypothesis following~\eqref{e:selfintersection},
the selfintersection points of $\tilde{c}_1$ do not meet $\tilde{c}_2$ 
and conversely the selfintersection points of $\tilde{c}_2$ do not meet $\tilde{c}_1$.
Moreover the intersection of both curves are transverse. 
It means that the following intersection is disjoint 
$$\left(\Sigma_1\cup \left(\bigcup_{a\in \text{Cross}(\tilde{c}_1)} S_a^*X\right) \right)\cap \left(\Sigma_2\cup \left(\bigcup_{a\in \text{Cross}(\tilde{c}_2)} S_a^*X\right) \right)=\emptyset.$$
Thus the union of supports $\cup_{j=1}^{N_1}\text{supp}([\Sigma_{1,j}]) $, $\cup_{j=1}^{N_2}\text{supp}([\Sigma_{2,j}]) $
are disjoint. Since the projection on $S^*X$ of the wave front set of a current in $\mathcal{D}^\prime(S^*X)$ is contained in the support of the current, this implies that one can choose $\Gamma_1,\Gamma_2$ with the expected properties.
\end{proof}

We now turn to the smoothing of our curves:

\begin{lemm}\label{l:smoothcurvelemma}
One can construct a family of smooth curves $(\tilde{c}_{i,j}^m)_{m\geqslant 1, i\in \{1,2\},1\leqslant j\leqslant N_i}$ on $X$ with the following properties:
\begin{itemize}
 \item for every $t$, $\|(\tilde{c}_{i,j}^m)'(t)\|=1$,
 \item $[\tilde{c}_{i,j}^m]$ converges weakly to $[\tilde{c}_{i,j}]$ in $\mathcal{D}^{\prime 1}(X)$,
 \item $[\tilde{c}_{i,j}^m]=[\tilde{c}_{i,j}]$ outside of some neighborhood (depending on $m$) of the singularities of $\tilde{c}_{i,j}$;
  \item one can attach above each point $\tilde{c}_{i,j}^m(t)$ of the curve, its normalized conormal vector $((\tilde{c}_{i,j}^m)'(t)^\flat)^\perp$ so that the closed curve $$t\mapsto (\tilde{c}_{i,j}^m(t),((\tilde{c}_{i,j}^m)'(t)^\flat)^\perp)$$ is smooth and if we denote by $\Sigma_{i,j,m}$ the image of this curve in $S^*X$, then one has $[\Sigma_{i,j,m}]=[\Sigma_{i,j}]$ away from the singularities and, as $m\rightarrow+\infty$,
 $$[\Sigma_{i,j,m}]\rightarrow[\Sigma_{i,j}],$$
 in $\mathcal{D}^{\prime 2}_{\Gamma_i}(S^*X)$.
\end{itemize}
\end{lemm}
We refer to the Appendix~\ref{aa:topology} for a reminder on the topology of $\mathcal{D}^{\prime}_{\Gamma_i}(S^*X)$. For the sake of exposition, we postpone the proof of this technical statement to Appendix~\ref{ss:proof-smoothing}. We underline that this Lemma shows that our piecewise Legendrian knots can be approximated by $\ml{C}^\infty$ Legendrian knots while keeping the same wavefront properties on the knots. Among other things, it will ensure that the product is well defined and that we can pass to the limit by sequential continuity of the product of currents with disjoint wavefront sets. See Appendix~\ref{a:WF}.

 By construction, we also remark that for every $m$ large enough, the curve $\tilde{c}_{i,j}^m$ bounds a compact surface $X_{i,j}^m$ with smooth boundary which has the same topology as $X_{i,j}$. In particular, for $m$ large enough, for $i\in\{1,2\}$ and for $1\leqslant j\leqslant N_i$
 \begin{equation}\label{e:euler-smooth}
  \chi(X_{i,j})=\chi(X_{i,j}^m).
 \end{equation}
Similarly, recall that the singularities of $\partial X_{1,j}$ and $\partial X_{2,j'}$ are away from each other by construction of $\tilde{c}_1$ and $\tilde{c}_2$. Thus, one finds that for $m,m'$ large enough, for $1\leqslant j\leqslant N_1$ and for $1\leqslant j'\leqslant N_2$,
 \begin{equation}\label{e:euler-smooth2}
  \chi(X_{1,j}\cap X_{2,j'})=\chi(X_{1,j}^m\cap X_{2,j'}^{m'})\quad\text{and}\quad \chi(\partial X_{1,j}\cap \partial  X_{2,j'})=\chi(\partial X_{1,j}^m\cap \partial X_{2,j'}^{m'}).
 \end{equation}

\subsection{Proof of Theorem~\ref{t:linking-integer-microlocal}}\label{sss:reformulation} We now come back to our computation of the linking number $\mathbf{L}(c_1,c_2)$.

Using the above regularization together with the continuity property of the wedge product on $\mathcal{D}^\prime_{\Gamma}(M)$ -- see Appendix~\ref{aa:product}, 
we obtain 
\begin{equation}\label{e:decomposition-linking}\mathbf{L}(c_1,c_2)=\sum_{j=1}^{N_1}\sum_{j'=1}^{N_2}\lim_{m\rightarrow+\infty}\lim_{m'\rightarrow+\infty}\int_{S^*X} [\Sigma_{1,j,m}]\wedge R_{2,j',m'},\end{equation}
where $[\Sigma_{2,j',m'}]=dR_{2,j',m'}$ with $R_{2,j',m'}\in\mathcal{D}^{\prime 1}_{\Gamma_2}(M)$. Hence, we are left with the computation of
$$\mathbf{L}\left(\tilde{c}_{1,j}^m,\tilde{c}_{2,j'}^{m'}\right):=\int_{S^*X} [\Sigma_{1,j,m}]\wedge R_{2,j',m'},$$
for every $(j,j')$ and for every $m,m'\geqslant 1$ large enough. Now, the curves $\tilde{c}_{1,j}^m$ and $\tilde{c}_{2,j'}^{m'}$ are simple, smooth and homologically trivial. Hence, we can apply the results of \S\ref{s:morse} together with~\eqref{e:euler-smooth} and~\eqref{e:euler-smooth2} to derive that
$$\mathbf{L}\left(\tilde{c}_{1,j}^m,\tilde{c}_{2,j'}^{m'}\right)=-\frac{\chi\left(X_{1,j}\right)\chi\left( X_{2,j'}\right)}{\chi(X)}\chi\left(X_{1,j}\cap X_{2,j'}\right)-\frac{1}{2}\chi\left(\partial X_{1,j}\cap\partial X_{2,j'}\right),$$
Summing over $j$ and $j'$, one finds by definition that
$$\mathbf{L}(c_1,c_2)=-\frac{\chi(f_1)\chi(f_2)}{\chi(X)}+\chi(f_1f_2)-\frac{1}{2}\chi\left(\mathbf{1}_{\tilde{c}_1\cap\tilde{c}_2}\right).$$

\subsection{Relation with microlocal index formulas}\label{ss:microlocal-symplectic}

Our derivation of the topological content of $\mathcal{N}_\infty(c_1,c_2,0)$ relied crucially on the \emph{Poincar\'e-Hopf index formula} as it was derived by Morse in~\cite{Mor29}. In the present section, we used this formula from a point of view which is inspired by microlocal geometry. In fact, the microlocal index theorems of Brylinski--Dubson--Kashiwara~\cite{BDK} and Kashiwara~\cite{K85}, later revisited by Kashiwara--Schapira~\cite[p.~384]{KS} and Grinberg--McPherson~\cite{GrMcP},
can be understood as generalizations of the Poincar\'e--Hopf index formula. As the comparison is relevant here, we briefly explain their content following the presentation of~\cite{GrMcP} to which we refer for more details. 

First, given a real algebraic manifold $X$ and a stratification $\mathcal{S}$ of $X$, one says that a function $f:X\mapsto \mathbb{Z}$ is constructible if it is constant on each stratum and the notion of Euler characteristic generalizes to constructible functions~\cite{Viro, Scha1, Scha2, Scha3}, $f\mapsto \chi(f):=\int_Xfd\chi,$ 
Now, given any stratum $S$ of $\mathcal{S}$, one can define its conormal bundle which is a Lagrangian submanifold $\Lambda_S\subset T^*X$. Then, the Lagrangian cycle $\text{Ch}(f)$ of $f$ is defined by assigning to each Lagrangian submanifold $\Lambda_S$ its multiplicity which roughly speaking is the value of $f$ on $S$ -- see~\cite[p.~277]{GrMcP} for details. 
Then, for every pair $f_1,f_2$ of constructible functions on $X$ which satisfy some appropriate transversality conditions, the microlocal index formula reads~\cite[p.~269]{GrMcP}:
\begin{equation}
\boxed{ \underset{\text{Euler integral}}{\underbrace{ \chi(f_1f_2)}}= \underset{\text{Lagrangian intersection}}{\underbrace{ [\text{Ch}(f_1)]\cap [\text{Ch}(f_2)]}}}
\end{equation}
where $[\text{Ch}(f_1)]\cap [\text{Ch}(f_2)]$ is the intersection of the two corresponding Lagrangian cycles. Hence, the microlocal index formula gives an interpretation of Lagrangian intersections as the Euler characteristic of some product of constructible functions.

As we saw when proving Theorem~\ref{t:linking-integer-microlocal}, we derived a formula in the spirit of the above microlocal index formula. Instead of computing the intersection of Lagrangian cycles, we rather considered the linking of Legendrian cycles but we also expressed it in terms of constructible functions. More precisely, for every pair of Legendrian cycles $\Sigma_1,\Sigma_2$ which are small deformations by Hamiltonian isotopies of the unit conormal bundle of our homologically trivial geodesics $c_1$ and $c_2$,  we associated a pair $(f_1,f_2)$ of constructible functions quantizing the two knots $\Sigma_1,\Sigma_2$. In that respect, Theorem~\ref{t:linking-integer-microlocal} can be viewed as a microlocal index formula:
\begin{eqnarray*}
\underset{\text{Euler integral}}{\underbrace{ \frac{\chi(f_1)\chi(f_2)}{\chi(X)} - \chi(f_1f_2)+\frac{1}{2}\chi(\mathbf{1}_{c_1\cap c_2})}}&=& \underset{\text{Legendrian linking}}{\underbrace{\pm \mathbf{Lk}\left( \Sigma_1,\Sigma_2 \right) }}\\
&=&\underset{\text{Poincar\'e series at zero}}{ \underbrace{ \lim_{s\rightarrow 0}\sum_{\gamma\in\mathcal{P}_{c_1,c_2}:\ell(\gamma)>0} e^{-\ell(\gamma)s} }}  .
\end{eqnarray*}
In the framework of symplectic topology, the Poincar\'e series is understood as a sum over the Reeb chords of the geodesic flow joining the two Legendrian curves $\Sigma_1$ and $\Sigma_2$. Hence, this index formula, which seems to be new\footnote{However see~\cite[Th.4]{Tu} and~\cite[Eq.~(10)]{Poly} for related results of Turaev regarding the first equality on $S^*\IS^2$.}, gives an interpretation of some linking of two Legendrian curves in terms of Euler integrals but also as a zeta regularized sum over the Reeb chords from $\Sigma_1$ to $\Sigma_2$. While the first equality is obtained by purely topological means, the second one is a consequence of our spectral approach to the problem. In fact, we conjecture that the first equality in this index-type formula should generalize to more general Legendrian knots and also to higher dimensional Legendrian boundaries for the appropriate notion of linking between higher dimensional objects. The generalization of the second equality is more subtle and it is related to the structure of Pollicott-Ruelle resonant states at $0$ as we already discussed in the introduction.

\appendix

\section{A brief reminder on the wavefront set of a distribution}\label{a:WF}

In this appendix, we briefly recall the notion of the wavefront set of a distribution and collect some classical properties that were used all along this article. The presentation is close to~\cite{DyZw13, BrDaHe16, DaRi17d} to which we refer for more informations and references.

The space $\ml{D}^{\prime k}_{\Gamma}(M)$ denotes the currents of degree $0\leqslant k\leqslant n=\text{dim}(M)$ whose wavefront set is contained in a fixed closed conic set $\Gamma \subset T^*M\setminus \underline{0}$, with $\underline{0}$ denoting the zero section. Recall first that an element in $\ml{D}^{\prime k}_{\Gamma}(M)$ is a current $u$ of degree $k$ such that, for every $N\geqslant 1$, for every open set $U$, for every closed cone $C$ such that
$\left(\text{supp }\chi\times C\right)\cap \Gamma=\emptyset$, one has
\begin{eqnarray}
\Vert u\Vert_{N,C,\chi,\alpha,U}:= \Vert (1+\|\xi\|)^{N} \ml{F}(u_\alpha\chi)(\xi) \Vert_{L^\infty(C)}<+\infty,
\end{eqnarray}
where $\chi$ is supported on the chart $U$, where $u=\sum_{\vert\alpha\vert=k} u_\alpha dx^\alpha$ where $\alpha$ is a multi--index and where $\ml{F}$ is the Fourier transform computed in the local chart $U$. Given a smooth, closed, embedded, oriented submanifold $\Sigma$ of dimension $n-k$ inside $M$, one can verify that the current of integration $[\Sigma]$ over $\Sigma$, defined as
$$\forall\psi\in\Omega^{n-k}(M),\quad\langle [\Sigma],\psi\rangle=\int_\Sigma\psi,$$
is an element in $\mathcal{D}^{\prime k}_{N^*(\Sigma)}(M)$, where
$$N^*(\Sigma):=\left\{(x,\xi)\in T^*M\setminus\underline{0}: x\in\Sigma\ \text{and}\ \forall v\in T_x\Sigma,\ \xi(v)=0\right\}.$$

\begin{rema}
 For a current $u$ of degree $k$, the wavefront set of $u$, denoted by $\text{WF}(u)$, is the smallest conic cone $\Gamma$ such that $u\in\mathcal{D}^{\prime k}_{\Gamma}(M)$.
\end{rema}


\subsection{Topology on the space $\mathcal{D}_{\Gamma}^\prime(M)$}\label{aa:topology}

Let us first recall the notion of bounded subsets in $\mathcal{D}^{\prime k}(M)$ following~\cite[Ch.~3, p.~72]{Schwartz-66}:
\begin{def1}\label{d:bounded}
A subset $B$ of currents is bounded if, for every test form $\varphi\in \Omega^{n-k}(M)$, 
$\sup_{t\in B}\vert \langle t,\varphi \rangle  \vert<+\infty$. 
\end{def1}
This definition is often referred as weak boundedness and it is equivalent to the notion of boundedness induced by the strong topology on $\mathcal{D}^{\prime k}(M)$~\cite[Ch.~3]{Schwartz-66}. We note that this is equivalent to $B$ being bounded in some 
Sobolev space $H^s(M,\Lambda^k(T^*M))$ of currents by suitable application of the uniform boundedness principle~\cite[\S~5, Lemma~23]{DabBr14}. We can now define the normal topology in the space of currents essentially following~\cite[Sect.~3]{BrDaHe16}:
\begin{def1}[Normal topology on the space of currents] For every closed conic subset $\Gamma\subset T^*M\setminus \underline{0}$,
the topology of $\mathcal{D}^{\prime k}_{\Gamma}(M)$ is 
defined as the weakest topology which makes continuous the seminorms of the strong topology of $\mathcal{D}^{\prime k}(M)$ and
the seminorms:
\begin{eqnarray}
\Vert u\Vert_{N,C,\chi,\alpha,U}= \Vert (1+\|\xi\|)^{N} \ml{F}(u_\alpha\chi)(\xi) \Vert_{L^\infty(C)}
\end{eqnarray}
where $\chi$ is supported on some chart $U$, where $u=\sum_{\vert\alpha\vert=k} u_\alpha dx^\alpha$ where $\alpha$ is a multi--index, where $\ml{F}$ is the Fourier transform computed in the local chart
and $C$ is a closed cone such that
$\left(\text{supp }\chi\times C\right)\cap \Gamma=\emptyset$. A subset $B\subset \mathcal{D}^{\prime k}_\Gamma$ is called bounded
in $\mathcal{D}^{\prime k}_{\Gamma}$ if it is bounded in $\mathcal{D}^{\prime k}$ and if all seminorms 
$\Vert .\Vert_{N,C,\chi,\alpha,U}$ are bounded on $B$.
\end{def1}

We emphasize that this definition is given purely 
in terms of local charts without loss of generality. The above topology 
is in fact \emph{intrinsic as a consequence of the continuity of the
pull--back}~\cite[Prop 5.1 p.~211]{BrDaHe16} as emphasized by H\"ormander~\cite[p.~265]{Ho90} (see below for a brief reminder). 
Note that it is the same to consider currents or distributions
when we define the relevant topologies since currents are just elements of the form
$\sum u_{i_1,\dots,i_k} dx^{i_1}\wedge \dots \wedge dx^{i_k}$ in local coordinates
$(x^1,\dots,x^n)$ where the coefficients $u_{i_1,\dots,i_k}$ are distributions. We note that the above seminorms involve the $L^{\infty}$ norm while the anisotropic spaces we deal with in this article are built from $L^2$ norms. This problem is handled by~\cite[App.~B]{DaRi17d}.
Let us now discuss some of the properties of the space $\mathcal{D}^{\prime k}_{\Gamma}(M)$ under standard operations: product, pullback, pushforward.

\subsection{Product of currents}\label{aa:product}

Given two closed conic sets $(\Gamma_1,\Gamma_2)$ which have empty intersection, the usual wedge product of smooth forms
$$\wedge:(\varphi_1,\varphi_2)\in \Omega^k(M) \times \Omega^l(M) \longmapsto \varphi_1\wedge\varphi_2\in \Omega^{k+l}(M)$$
extends uniquely as a hypocontinuous map for the normal topology~\cite[Th.~6.1]{BrDaHe16} 
$$\wedge:(\varphi_1,\varphi_2)\in \ml{D}^{\prime k}_{\Gamma_1}(M) \times \ml{D}^{\prime l}_{\Gamma_2}(M) \longmapsto \varphi_1\wedge
\varphi_2 \in\mathcal{D}^{\prime k+l}_{s(\Gamma_1,\Gamma_2)}(M),$$
with $s(\Gamma_1,\Gamma_2)=\Gamma_1\cup \Gamma_2\cup (\Gamma_1+\Gamma_2)$.
The notion of hypocontinuity is a strong notion of 
continuity adapted to bilinear maps from $E\times F\mapsto G$ where $E,F,G$ are locally convex spaces~\cite[p.~204-205]{BrDaHe16}.
It is weaker than joint continuity but implies that the bilinear map is separately continuous in each factor uniformly in the other factor
in a bounded subset\footnote{The tensor product of distributions for the strong topology is hypocontinuous but not continuous~\cite[p.~205]{BrDaHe16}}.

\subsection{Pullback of currents}\label{aa:pullback}

Let $\Gamma$ be a closed conic set and let $f$ be a smooth \emph{diffeomorphism} on $M$. The usual pullback operation on smooth forms,
$$f^*:\Omega^k(M)\rightarrow \Omega^k(M)$$
extends uniquely as a continuous map~\cite[Prop.~5.1]{BrDaHe16} from $\mathcal{D}^{\prime k}_\Gamma(M)$ to $\mathcal{D}^{\prime k}_{f^*\Gamma}(M)$ for the normal topology, with $f^*\Gamma$ defined as
$$f^*\Gamma:=\left\{\left(f^{-1}(x),(df^{-1}(x)^T)^{-1}\xi\right)\in T^*M\setminus\underline{0}:\ (x,\xi)\in\Gamma\right\}.$$

\subsection{Pushforward of currents}\label{aa:pushforward}
Let $\Gamma$ be a closed conic set and let $f:M\rightarrow N$ be a smooth map between the smooth, compact, boundaryless manifolds $M$ and $N$.
The usual pushforward operation on smooth forms,
$$f_*:\Omega^k(M)\rightarrow \Omega^k(M)$$
extends uniquely as a continuous map~\cite[Th.~6.3]{BrDaHe16} from $\mathcal{D}^{\prime k}_\Gamma(M)$ to $\mathcal{D}^{\prime k}_{f_*\Gamma}(M)$ for the normal topology, with $f_*\Gamma$ defined as
$$f_*\Gamma:=\left\{\left(y,\eta\right)\in T^*M\setminus\underline{0}:\ \exists(x,\xi)\in\Gamma\cup\underline{0}\ \text{s.t.}\ f(x)=y\ \text{and}\ \xi=df(x)^T\eta\right\}.$$

\section{The case of Anosov flows}\label{a:anosov}
The proof that we gave of Theorem~\ref{t:meromorphic} can in fact be easily adapted in the more general context of Anosov flows and of weighted Poincar\'e series. In order to state this more general result, we fix $\tilde{M}$ to be a smooth, closed, Riemannian, oriented manifold of dimension $\geq 3$. Recall that an Anosov flow is a flow $\varphi^t$ satisfying property~\eqref{e:Anosov} or more precisely its extension in higher dimension, i.e. $\text{dim}\ E_u$ and $\text{dim}\ E_s$ may be $\geq 1$ and of different dimensions. Note also that the Margulis transversality assumptions~\eqref{e:transversality-unstable} and~\eqref{e:transversality-stable} can also be extended to that set-up. Hence, we fix $\Sigma_1$ and $\Sigma_2$ to be two submanifolds verifying these assumptions and some large $T_0>0$ to ensure that $\sharp \varphi^{-T_0}(\Sigma_1)\cap\varphi^{t+T_0}(\Sigma_2)$ is finite for every $t\geqslant 0$. See Lemma~\ref{l:transversality} for the proof in the case of geodesic flows which in fact remains valid in this generalized framework. Given $W\in\mathcal{C}^\infty(\tilde{M},\IC)$, we can then define the following zeta function
\begin{multline*}\tilde{\zeta}_{\Sigma_1,\Sigma_2}(z)\\:=\sum_{t\geqslant 0: \varphi^{-T_0}(\Sigma_1)\cap\varphi^{t+T_0}(\Sigma_2)\neq\emptyset}e^{-zt}\left(\sum_{x\in \varphi^{-T_0}(\Sigma_1)\cap\varphi^{t+T_0}(\Sigma_2)}\epsilon_t(x)e^{-\int_{-t}^0 W\circ\varphi^{s}(x)|ds|}\right),
\end{multline*}
where $\epsilon_t(x)=1$ if 
$$d_{\varphi^{T_0}(x)}\varphi^{-T_0}\left(T_{\varphi^{T_0}(x)}\Sigma_1\right)\oplus\IR V(x)\oplus  d_{\varphi^{-T_0-t}(x)}\varphi^{T_0+t}\left(T_{\varphi^{-T_0-t}(x)}\Sigma_2\right)$$
has the same orientation as $T_x\tilde{M},$ and to $-1$ otherwise. Thanks to Lemma~\ref{l:exp-growth}, the function $\zeta_{\Sigma_1,\Sigma_2}$ is well defined and holomorphic for $\text{Re}(z)\gg 1$. The extension of Theorem~\ref{t:meromorphic} to Anosov flows and to weighted Poincar\'e series reads as follows:
\begin{theo} Let $W\in\ml{C}^\infty(\tilde{M},\IC)$ and let $\Sigma_1$ and $\Sigma_2$ verifying~\eqref{e:transversality-unstable} and~\eqref{e:transversality-stable}. Then, there exists $T_0,C_0>0$ such that 
$$\tilde{\zeta}_{\Sigma_1,\Sigma_2}:\{\operatorname{Re}(s)\geq C_0\}\rightarrow\IC$$
is well defined and holomorphic. Moreover, it extends meromorphically to $\IC$. 
\end{theo}
The proof of Theorem~\ref{t:meromorphic} applies directly to treat the case of this more general Theorem provided that we replace $\mathcal{L}_V$ by $\mathcal{L}_V+W$ to handle the exponential weights in the sum.

\section{Linking of closed geodesics}\label{ss:ghys}

When $\mathbf{c}_1$ and $\mathbf{c}_2$ are both nontrivial in $\pi_1(X)$, there are other natural curves in $S^*X$ that one may associate to $c_i:\IR/\ell_i\IZ\rightarrow X$:
$$\Sigma_{\text{geod}}(c_i):=\left\{(c_i(t),c_i'(t)^\flat):\ t\in\IR/\ell_i\IZ\right\}.$$
This is just the closed geodesic lifting $c_i$ in $S^*X$. It is then natural to ask whether the linking number $\mathbf{L}(c_1,c_2)$ is related to the linking number of $\Sigma_{\text{geod}}(c_1)$ and $\Sigma_{\text{geod}}(c_2)$ and this is indeed the case as we will establish. Here for simplicity, we always suppose that $c_1\neq c_2$ so that $\Sigma(c_1)\cap\Sigma(c_2)=\emptyset$. First of all, we can define using the conventions of Section~\ref{a:geometry} the following diffeomorphism of $S^*X$:
$$\mathcal{R}:x=(q,p)\in S^*X\mapsto (q,p^{\perp})\in S^*X.$$
This map is orientation-preserving as it is isotopic to the identity and, for $i=1,2$, one has
$$[\Sigma(c_i)]=\mathcal{R}^{-1*}[\Sigma_{\text{geod}}(c_i)].$$
In particular, $[\Sigma_{\text{geod}}(c_i)]$ is de Rham exact when $\mathbf{c}_i$ is homologically trivial as $\Sigma(c_i)$ is. Using the conventions of Proposition~\ref{p:value-at-0}, one has 
$$\mathbf{L}(c_1,c_2)=\int_{S^*X}[\Sigma(c_1)]\wedge R_2,$$
where $[\Sigma(c_2)]=dR_2$. Hence, using that $\mathcal{R}$ is orientation preserving and the continuity of the wedge product of currents whose wave front sets are transverse (see appendix~\ref{a:WF}), we can deduce that
$$\mathbf{L}(c_1,c_2)=\int_{S^*X}\mathcal{R}^{*}[\Sigma(c_1)]\wedge \mathcal{R}^{*}R_2=\int_{S^*X}[\Sigma_{\text{geod}}(c_1)]\wedge \mathcal{R}^{*}R_2,$$
where $[\Sigma_{\text{geod}}(c_2)]=d\left(\mathcal{R}^{*}R_2\right)$ (as $d$ commutes with $\mathcal{R}^*$). In other words, for nontrivial homotopy classes, the linking number $\mathbf{L}(c_1,c_2)$ we have been computing in this section is equal to the linking number of the geodesic curves lifting $c_1$ and $c_2$. Hence, we can reformulate Proposition~\ref{p:value-at-0} as follows:
\begin{theo}\label{t:ghyslike} Suppose that $\mathbf{c}_i$ are both nontrivial homotopy classes which are homologically trivial and distinct. 

Then $\mathcal{N}_\infty(c_1,c_2,0)=\mathbf{L}(c_1,c_2)$ is the linking number of the two closed geodesics $\Sigma_{\text{geod}}(c_1)$ and $\Sigma_{\text{geod}}(c_2)$ which lift $c_1$ and $c_2$ to $S^*X$. 


\end{theo}
In particular, this establishes a direct connection with the works of Duke--Imamo\={g}lu--T\'oth who expressed the linking number of two closed geodesics on the modular surface as the special value of a certain Dirichlet series~\cite{DIT}. Similarly, for suspension of toral automorphisms, the linking of closed orbits was identified with the special value of certain $L$-functions by Bergeron--Charollois--Garcia--Venkatesh~\cite{Ber18}.

\section{Proof of Lemma~\ref{l:smoothcurvelemma}}\label{ss:proof-smoothing} We finally prove the technical Lemma that was needed to smooth our piecewise smooth curves in Section~\ref{s:intersection}.
Up to the fact that we may have to reparametrize the curve (and up to using a partition of unity), we note that we only need to modify the curve $\tilde{c}_{i,j}$ in a small neighborhood of its singularities.
The point is that we will round the ``corners'' of the bounding surface $X_{i,j}$ to make the curve $\tilde{c}_{i,j}^m$ smooth.

 Thanks to assumption~\eqref{e:selfintersection}, we note that, in a small chart $(\tilde{q}_1,\tilde{q}_2)$ near a point $q_0$ at some corner of $X_{i,j}$, $X_{i,j}$ is given by $\{\tilde{q}_1\leqslant 0,\tilde{q}_2\geqslant 0 \}$ or $\{\tilde{q}_1\leqslant 0\}\cup\{\tilde{q}_2\geqslant 0\}$ (with the usual orientations on $\mathbb{R}^2$). We only discuss the first case and the second case is 
treated in a similar way.
In this local chart, the boundary of $X_{i,j}$ near the singular point has a local piecewise smooth parametrization which reads 
$$t\in [-1,1]\mapsto \gamma(t)= \left(-t\mathbf{1}_{[-1,0]}(t),0\right)+\left(0,t\mathbf{1}_{[0,1]}(t)\right)\in \mathbb{R}^2.$$ 
Observe that 
\begin{eqnarray*}\tilde{\gamma}_m(t)&:=& \left((t\mathbf{1}_{[-1,\frac{-1}{m}]}(t),0\right)+\left(0,t\mathbf{1}_{[\frac{1}{m},1]}(t)\right)\\
&+&\left(\frac{1}{m}\cos\left(\frac{m\pi t}{4}-\frac{\pi}{4}\right)-\frac{1}{m} ,\frac{1}{m}\sin\left(\frac{m\pi t}{4}-\frac{\pi}{4}\right)+\frac{1}{m} \right)\mathbf{1}_{[-\frac{1}{m},\frac{1}{m}]}(t)            
\end{eqnarray*}
is a $\mathcal{C}^1$--path, which is only piecewise $\mathcal{C}^\infty$, and lies in  
some $\frac{1}{m}$-neighborhood of $\gamma$. Hence
$\tilde{\gamma}_m$ bounds the domain 
$$\left\{\tilde{q}_1\leqslant -\frac{1}{m},\tilde{q}_2\geqslant 0 \right\}\cup\left\{\tilde{q}_2\geqslant \frac{1}{m}, \tilde{q}_1\leqslant 0\right\}\cup\left\{ \left(\tilde{q}_1+\frac{1}{m}\right)^2+\left(\tilde{q}_2-\frac{1}{m}\right)^2\leqslant \frac{1}{m^2} \right\}$$
which has $\mathcal{C}^1$ boundary. Hence, instead of the corner point $\{(0,0)\}$, we obtained a quarter circle. 
Now we fix $\chi\in C^\infty(\mathbb{R})$ such that $ \int_\mathbb{R}\chi=1, \chi\geqslant 0, \text{supp}(\chi)\subset [-\frac{1}{2},\frac{1}{2}]$ and we define $\chi_m(\tilde{q}_1,\tilde{q}_2)=\frac{1}{m^2}\chi(m\tilde{q}_1,m\tilde{q}_2)$. 
We can define the new parametrization $\gamma_m=\tilde{\gamma}_m*\chi_m\in \mathbb{R}^2 $
obtained by convolution. This new curve $\gamma_m$ converges to $\tilde{\gamma}_m$ in the $\mathcal{C}^1$-topology and the image of both curves coincide outside some $\frac{4}{m}$--neighborhood of the corner point $(0,0)$.

Define $X_{i,j}^m$ to be the new surface obtained from the above smoothing procedure at every corner point, this is a manifold with smooth boundary which is homotopic to $\partial X_{i,j}$ by construction.
Proceeding like this, one can verify that the first three properties are satisfied locally near the singular point (and thus globally via a partition of unity).

Regarding now the last property, we are in fact taking the (oriented and normalized) conormal to the curve $t\mapsto \tilde{c}_{i,j}^m(t)$. By construction, it has the expected properties away from the singularities of the initial curve. 
Near the singularity, one can write the above expression in local coordinates in $\mathbb{R}^2$ as above and verify that the current of integration along the curve
$$t\in [-1,1]\longmapsto 
\left(\gamma_m(t),\frac{(\gamma_m'(t)^\flat)^\perp}{\Vert((\gamma_m'(t)^\flat)^\perp\Vert } \right)$$ converges to the current of integration along
\begin{eqnarray*}N^*_1(\{\tilde{q}_1\leqslant 0,\tilde{q}_1\geqslant 0  \}) &=& \left\{(t,0;0,1); t\in (-1,0]  \right\}\cup \left\{(0,t;-1,0); t\in [0,1) \right\}\\
&\cup& \left\{(0,0;\cos(\theta),\sin(\theta)); \theta\in [\frac{\pi}{2},\pi] \right\}.
 \end{eqnarray*}
The only discussion is near the corner point. By construction, we see that the conormal lift of the $\mathcal{C}^1$-curve $\tilde{\gamma}_m$ which is the map $$ t\in [-1,1]\longmapsto \left(\tilde{\gamma}_m(t),\frac{(\tilde{\gamma}_m'(t)^\flat)^\perp}{\Vert((\tilde{\gamma}_m'(t)^\flat)^\perp\Vert } \right) $$
converges in the sense of currents to the 
conormal $N^*_{ 1}(\{\tilde{q}_1\leqslant 0,\tilde{q}_2\geqslant 0  \})$. Since $\gamma_m$ is $\mathcal{C}^1$ close to $\tilde{\gamma}_m$ for $m$ large enough, we are done.


\begin{thebibliography}{99}
\bibitem{Al} S.~Alesker, \emph{Theory of valuations on manifolds, II}, Adv. Math. $\mathbf{207}$ (2006), 420--454.
\bibitem{Ba} V.~Baladi, \emph{Dynamical zeta functions and dynamical determinants for hyperbolic maps. A functional approach}, Springer (2018)
\bibitem{BG} Y.~Baryshnikov and R.~Ghrist, \emph{Target enumeration via Euler characteristic integrals}, SIAM Journal on Applied Mathematics $\mathbf{70}$ (2009), 825--844.
\bibitem{Bas93} A.  Basmajian, \emph{The orthogonal spectrum of a hyperbolic manifold}, Amer. J. Math. $\mathbf{115}$ (1993), 1139--1159
\bibitem{Ber18} N.~Bergeron, \emph{Enlacement dans les fibr\'es en tore et fonctions $L$ de Hecke}, Congr\`es Soc. Math. Fr. (2018)
\bibitem{Be78} A.~Besse \emph{Manifolds All of Whose Geodesics Are Closed}, Ergeb. Math. $\mathbf{93}$, Springer-Verlag, New York, (1978)
\bibitem{BoTu82} R.~Bott and L.W.~Tu \emph{Differential forms in algebraic topology}, Springer-Verlag (1982).
\bibitem{BrDaHe16} C.~Brouder, N.V.~Dang and F.~H\'elein, \emph{Continuity of the fundamental operations on distributions having a specified wave front set}, 
Studia Math. 232 (2016), 201--226 
\bibitem{BDK} J.L.~Brylinski, A.~Dubson, and M.~Kashiwara, \emph{Formule de l'indice pour les modules holonomes et obstruction d'Euler locale.} C.R. Acad. Sci. Paris S\'er. I Math $\mathbf{293}$ (1981), 573--576.
\bibitem{BuPa} K.~Burns and G.P.~Paternain, \emph{On the growth of the number of geodesics joining two points}, International Conference on Dynamical Systems (Montevideo, 1995), 7--20, Pitman Res. Notes Math. Ser., 362, Longman, Harlow, 1996
\bibitem{BL07} O.~Butterley and C.~Liverani, \emph{Smooth Anosov flows: correlation spectra and stability}, J. Mod. Dyn. 1 (2007), 301--322
\bibitem{BL13} O.~Butterley and C.~Liverani, \emph{Robustly invariant sets in fiber contracting bundle}, J. Mod. Dyn. 7 (2013), 255--267.
\bibitem{CePa} M.~Ceki\'{c} and G.P.~Paternain, \emph{Resonant spaces for volume preserving Anosov flows}, Pure and Applied Analysis $\textbf{2}$ (2020), 795--840. 
\bibitem{CDDP22} M. Cekic, B.~Delarue, S.~Dyatlov and G.~Paternain, \emph{Ruelle zeta function at zero for nearly hyperbolic 3-manifolds}, Inv. Math. $\textbf{229}$ (2022), 303--394   
\bibitem{Ch21} Y.~Chaubet, \emph{Poincar\'e series for surfaces with boundary}, Nonlinearity, $\textbf{35}$ (2022) 5993
\bibitem{CGR} J.~Curry, R.~Ghrist and M.~Robinson, \emph{Euler calculus with applications to signals and sensing}, Proceedings of Symposia in Applied Mathematics $\mathbf{70}$ (2012)
\bibitem{DabBr14} Y. Dabrowski and C. Brouder, \emph{Functional properties of H\"ormander's space of distributions having a specified wavefront set}, 
Comm. Math. Phys. 332 (2014), 1345--1380
\bibitem{DGRS18} N.V.~Dang, C. Guillarmou, G.~Rivi\`ere, S. Shen \emph{The Fried conjecture in small dimensions}, Inventiones Math. $220$ (2020), 525--579
\bibitem{DaRi16} N.V.~Dang, G.~Rivi\`ere, \emph{Spectral analysis of Morse-Smale gradient flows}, Ann. Sci. ENS $\textbf{52}$ (2019), 1403--1458
\bibitem{DaRi17c} N.V.~Dang, G.~Rivi\`ere, \emph{Topology of Pollicott-Ruelle resonant states}, Ann. Sc. Norm. Sup. di Pisa, Vol. XXI, (2020), 827--871
\bibitem{DaRi17d} N.V.~Dang, G.~Rivi\`ere, \emph{Pollicott-Ruelle spectrum and Witten Laplacians}, J. Eur. Math. Soc. $\textbf{23}$ (2020), 1797--1857
\bibitem{Del42} J.~Delsarte, \emph{Sur le Gitter Fuchsien}, C.R. Acad. Sci. Paris $\mathbf{214}$ (1942), 147--149
\bibitem{Do} D.~Dolgopyat, \emph{On decay of correlations in Anosov flows}, Ann. of Math. (2) $\textbf{147}$ (1998), 357--390
\bibitem{DIT} W.~Duke, O. Imamo\={g}lu and A..~T\'oth, \emph{Linking numbers and modular cocycles}, Duke Math. J. $\textbf{166}$ (2017), 1179--1210.
\bibitem{DFG} S.~Dyatlov, F.~Faure and C.~Guillarmou, \emph{Power spectrum of the geodesic flow on hyperbolic manifolds}, Anal. PDE $\textbf{8}$ (2015), 923--1000 
\bibitem{DyGu} S.~Dyatlov and C.~Guillarmou, \emph{Pollicott-Ruelle resonances for open systems}, Ann. Henri Poincar\'e $\textbf{17}$ (2016), 3089--3146
\bibitem{DyZw13} S.~Dyatlov, M.~Zworski, \emph{Dynamical zeta functions for Anosov flows via microlocal analysis},  Ann. Sci. ENS \textbf{49} (2016), 543--577
\bibitem{DyZw} S. Dyatlov, M. Zworski, \emph{Ruelle zeta function at zero for surfaces},
Invent. Math. \textbf{210} (2017), 211--229 
\bibitem{Eb} P.~Eberlein, \emph{Geodesic flows on negatively curved manifolds I}, Ann. of Math. $\textbf{95}$ (1972), 492--510
 \bibitem{EN} K.J.~Engel,  R.~Nagel, \emph{One-Parameter  Semigroups  for  Linear  Evolution  Equations},  Grad.  Texts  in Math. 194, Springer-Verlag New York (2000)
\bibitem{FRS} F. Faure, N. Roy, J. Sj\"ostrand, \emph{Semi-classical approach for Anosov diffeomorphisms and Ruelle resonances}, Open Math. Journal \textbf{1} (2008), 35--81.
\bibitem{FaTs} F.~Faure, M.~Tsujii \emph{The semiclassical zeta function for geodesic flows on negatively curved manifolds}, Invent. math. \textbf{208} (2017), 851--998.
\bibitem{FaSj} F. Faure, J. Sj\"ostrand, \emph{Upper bound on the density of Ruelle resonances for Anosov flows}, Comm. Math. Phys. \textbf{308} (2011), 325--364.
\bibitem{Fr86} D.~Fried, \emph{Fuchsian groups and Reidemeister torsion},in The Selberg trace formula and related topics (Brunswick, Maine, 1984), Contemp. Math. $\mathbf{53}$, 141--163, Amer. Math. Soc., Providence, RI, 1986
\bibitem{Fr} D.~Fried, \emph{Meromorphic zeta functions for analytic flows},
Comm. Math. Phys. $\textbf{174}$ (1995), 161--190. 
\bibitem{Gh07} E.~Ghys, \emph{Knots and dynamics}, International Congress of Mathematicians. Vol. I, 247--277, Eur. Math.Soc., Z\"urich (2007).
\bibitem{GMS} M.~Giaquinta, G.~Modica, J. Soucek \emph{Cartesian currents in the calculus of variations I. Cartesian currents}, Springer (1998) 
\bibitem{GLP} P.~Giulietti, C.~Liverani, M.~Pollicott, \emph{Anosov flows and dynamical zeta functions}, Annals of Math. (2) \textbf{178} (2013), no. 2, 687--773
\bibitem{Goo} A. Good \emph{Local analysis of Selberg's trace formula}, Lecture Notes in Mathematics $\mathbf{1040}$, Springer-Verlag, Berlin, (1983)
\bibitem{Go} S.~Gou\"ezel, \emph{Spectre du flot g\'eod\'esique en courbure n\'egative [d'apr\`es F. Faure et M. Tsujii]}. Ast\'erisque $\textbf{380}$, S\'eminaire Bourbaki. Vol. 2014/2015 (2016), Exp. No. 1098, 325--353
\bibitem{GoSto} S.~Gou\"ezel and L.~Stoyanov, \emph{Quantitative Pesin theory for Anosov diffeomorphisms and flows}, Ergodic Theory Dynam. Systems $\textbf{39}$ (2019), 159--200
\bibitem{GrMcP} M.~Grinberg and R.~MacPherson, \emph{Euler characteristics and Lagrangian intersections}, Symplectic geometry and topology $7$ (Park City, UT, 1997) (1999), 265--293
\bibitem{GBWe17} Y.~Guedes Bonthonneau and T.~Weich, \emph{Ruelle-Pollicott Resonances for Manifolds with Hyperbolic Cusps}, preprint arXiv:1712.07832 (2017)
\bibitem{Gu} L.~Guillop\'e, \emph{Entropies et spectres}, Osaka J. Math. $\textbf{31}$ (1994), 247--289
\bibitem{Gul} R.~Gulliver, \emph{On the Variety of Manifolds without Conjugate Points}, Trans. AMS $\textbf{210}$ (1975), 185--201
\bibitem{Gun80} P. G\"unther \emph{Gitterpunkt probleme in symmetrischen Riemannschen Rumen vom Rang 1}, Math. Nachr. $\mathbf{94}$ (1980), 5--27
\bibitem{Ha18} C.~Hadfield, \emph{Zeta function at zero for surfaces with boundary}, preprint arXiv:1803.10982 (2018)
\bibitem{Hat} A. Hatcher, \emph{Algebraic topology}, Cambridge University Press (2002)
\bibitem{Ho90} L.~H\"ormander, \emph{The Analysis of Linear Partial Differential Operators I. 
Distribution Theory and Fourier Analysis}, second ed., Springer Verlag, Berlin, 1990
\bibitem{Ho90III} L.~H\"ormander, \emph{The Analysis of Linear Partial Differential Operators III. 
Pseudodifferential operators}, second ed., Springer Verlag, Berlin, 1990
\bibitem{Hub56} H. Huber, \emph{Uber eine neue Klasse automorpher Funktionen und ein Gitterpunktproblem inder hyperbolischen Ebene. I.}, Comment. Math. Helv. $\mathbf{30}$ (1956), 20--62
\bibitem{Hub} H. Huber, \emph{Zur analytischen Theorie hyperbolischen Raumformen und Bewegungsgruppen}, Math. Ann. $\mathbf{138}$ (1959) 1--26
\bibitem{Je19} M.~Jezequel, \emph{Global trace formula for ultra-differentiable Anosov flows}, preprint arXiv:1901.09576 (2019)
\bibitem{K85} M.~Kashiwara, \emph{Index theorem for constructible sheaves}, Ast\'erisque $\mathbf{130}$ (1985), 193--209
\bibitem{KS} M.~Kashiwara and P.~Schapira, \emph{Sheaves on manifolds}, Grundlehren der Mathematischen Wissenschaften $\mathbf{292}$, Springer, (1990)
\bibitem{Klin} H.~Klingen, \emph{\"Uber die Werte der Dedekindschen Zetafunktion}, Math. Ann. $\textbf{145}$ (1961/1962), 265--272
\bibitem{Kl} W.~Klingenberg, \emph{Riemannian Geometry}, De Gruyter Berlin-New York (1982)
\bibitem{KuWe} B.~K\"uster and T.~Weich, \emph{Pollicott-Ruelle resonant states and Betti numbers}, Comm. Math. Phys. $\textbf{378}$ (2020) 917--941
\bibitem{Lee13} J.M.~Lee, \emph{Introduction to smooth manifolds}, Graduate text in mathematics $\textbf{218}$, 2nd edition (2013)
\bibitem{Ma} R.~Ma\~{n}\'e, \emph{On the topological entropy of geodesic flows}, J. Differential Geom. $\mathbf{45}$ (1997), 74--93.
\bibitem{Mar69} G.A.~Margulis, \emph{Applications of ergodic theory for the investigation of manifolds of negative curvature}, Funct. Anal. Applic. $\mathbf{3}$ (1969) 335--336.
\bibitem{Mar} G.A.~Margulis, \emph{On some aspects of the theory of Anosov systems}, Springer (2004)
\bibitem{Mo} C.~Moore, \emph{Exponential decay of correlation coefficients for geodesic flows}, Group representations, ergodic theory, operator algebras, and mathematical physics (Berkeley, Calif., 1984), 163--181, Math. Sci. Res. Inst. Publ., $\mathbf{6}$, Springer, New York, 1987. 
\bibitem{Mor29} M.~Morse, \emph{Singular Points of Vector Fields Under General Boundary Conditions}, American J. Math. $\mathbf{51}$ (1929), 165--178
\bibitem{ParPau16} J.~Parkkonen and F.~Paulin, \emph{Counting arcs in negative curvature}, "Geometry, Topology and Dynamics in Negative Curvature" (Bangalore, 2010), London Math. Soc. Lect. Notes. $\mathbf{425}$, 289--344, Cambridge Univ. Press, (2016)
\bibitem{ParPau17} J.~Parkkonen and F.~Paulin, \emph{Counting common perpendicular arcs in negative curvature}, Erg. Theo. Dyn. Sys. $\textbf{37}$ (2017), 900--938. 
\bibitem{Pa92} G.P.~Paternain, \emph{On the topology of manifolds with completely integrable geodesic flows}, Ergodic Theory Dynamical Systems $\mathbf{12}$ (1992) 109--121
\bibitem{Pa99} G.P.~Paternain \emph{Geodesic Flows}, Progress in Mathematics, $\mathbf{180}$. Birkh\"auser Boston, Inc., Boston, MA, (1999).
\bibitem{Pa00} G.P.~Paternain \emph{Topological pressure for geodesic flows}, Ann. Sci. ENS $\mathbf{33}$ (2000), 121--138
\bibitem{PaPa} G.P.~Paternain and M.~Paternain, \emph{Topological entropy versus geodesic entropy}, Internat. J. Math. $\mathbf{2}$ (1994) 213--218 
\bibitem{PaSaU} Paternain, Gabriel P., Mikko Salo, and Gunther Uhlmann. "Geometric inverse problems, with emphasis on two dimensions." Text in preparation (2022).
\bibitem{Patt75} S.J.~Patterson \emph{A lattice-point problem in hyperbolic space}, Mathematika  $\mathbf{22}$ (1975), 81--88.
\bibitem{Po} M.~Pollicott, \emph{A  symbolic  proof  of  a  theorem  of  Margulis  on  geodesic  arcs  on  negatively curved manifolds}, Amer. J. Math. $\mathbf{117}$ (1995) 289--305
\bibitem{Poly} M.~Polyak, \emph{Shadows of Legendrian links and J+-theory of curves}, Singularities. Birkh\"auser, Basel (1998), 435--458.
\bibitem{Ra} M.~Ratner, \emph{The rate of mixing for geodesic and horocycle flows},
Ergodic Theory Dynam. Systems $\mathbf{7}$ (1987), 267--288
\bibitem{Rue} D.~Ruelle, \emph{Zeta-functions for expanding maps and Anosov flows},
Invent. Math. $\textbf{34}$ (1976), 231--242
\bibitem{Rug07} R.O.~Ruggiero \emph{Dynamics and global geometry of manifolds without conjugate points}, Ensaios Mat. Vol. $\mathbf{12}$, Soc. Bras. Mat. (2007)
\bibitem{Ru} H.H.~Rugh, \emph{Generalized Fredholm determinants and Selberg zeta functions for Axiom A dynamical systems}, Ergodic Theory Dynam. Systems $\textbf{16}$ (1996), 805--819.
\bibitem{Scha1} P.~Schapira, \emph{Tomographie topologique}, preprint
\bibitem{Scha2} P.~Schapira, \emph{Constructible functions, Lagrangian cycles and computational geometry}, The Gelfand Mathematical Seminars, 1990--1992. Birkh\"auser, Boston, MA, 1993.
\bibitem{Scha3} P.~Schapira, \emph{Tomography of constructible functions}, International Symposium on Applied Algebra, Algebraic Algorithms, and Error-Correcting Codes, Springer, Berlin, Heidelberg, 1995.
\bibitem{Schwartz-66} L.~Schwartz, \emph{Th{\'e}orie des distributions}, Hermann, Paris, second edition (1966)
\bibitem{Sel70} A.~Selberg \emph{Equidistribution in discrete groups and the spectral theory of automorphic forms}, http://publications.ias.edu/selberg/section/2491, file 2, p.7
\bibitem{Sie37} C.L. Siegel \emph{\"Uber die analytische Theorie der quadratischen Formen. III.}, Ann.
of Math. $\textbf{38}$ (1937) 212--291
\bibitem{ST76} I.M. Singer, and J.A. Thorpe. Lecture notes on elementary topology and geometry. Springer, 2015.
\bibitem{Ts10} M.~Tsujii, \emph{Quasi-compactness of transfer operators for contact Anosov flows}, Nonlinearity $\mathbf{23}$ (2010), 1495--1545. 
\bibitem{Ts12} M.~Tsujii, \emph{Contact Anosov flows and the Fourier-Bros-Iagolnitzer transform}, Ergodic Theory Dynam. Systems $\mathbf{32}$ (2012), 2083--2118
\bibitem{Tu} V.~Turaev, \emph{Quantum invariants of 3-manifold and a glimpse of shadow topology}, Quantum groups. Springer, Berlin, Heidelberg (1992), 363--366.
\bibitem{Viro} O.Y.~Viro, \emph{Some integral calculus based on Euler characteristic}, Topology and geometry, Rohlin seminar, Springer, Berlin, Heidelberg, 1988.
\bibitem{Ze92} S.~Zelditch, \emph{Kuznecov Sum Formulae and Szeg\"o Limit Formulae on Manifolds}, CPDE $\mathbf{17}$ (1992), 221--260
\end{thebibliography}
\end{document}